\documentclass[reqno,12pt]{amsart}

\usepackage{amssymb,graphicx,color}

\usepackage[colorlinks=true]{hyperref}
\hypersetup{urlcolor=blue, citecolor=red}

\theoremstyle{plain}
\newtheorem{thm}{Theorem}[section]
\newtheorem{lem}{Lemma}[section]
\newtheorem{prop}{Proposition}[section]
\newtheorem{cor}{Corollary}[section]
\newtheorem{defs}{Definition}[section]
\theoremstyle{definition}
\newtheorem{rmk}{Remark}[section]

\newcommand{\NN}{{\mathbb{N}}}

\newcommand{\RR}{{\mathbb{R}}}

\newcommand{\bu}{\mathbf{u}}
\newcommand{\bv}{\mathbf{v}}
\newcommand{\bw}{\mathbf{w}}

\newcommand{\bx}{\mathbf{x}}

\newcommand{\bbf}{\mathbf{f}}

\newcommand{\bF}{\mathbf{F}}

\newcommand{\bnabla}{\boldsymbol{\nabla}}

\newcommand{\Ccal}{{\mathcal C}}
\newcommand{\Acal}{{\mathcal A}}
\newcommand{\Bcal}{{\mathcal B}}
\newcommand{\Ecal}{{\mathcal E}}
\newcommand{\Fcal}{{\mathcal F}}
\newcommand{\Scal}{{\mathcal S}}
\newcommand{\Ncal}{{\mathcal N}}
\newcommand{\Rcal}{{\mathcal R}}

\newcommand{\Vcal}{{\mathcal V}}
\newcommand{\Wcal}{{\mathcal W}}

\newcommand{\loc}{{\text{\rm loc}}}

\newcommand{\Mcal}{{\mathcal M}}
\newcommand{\Dcal}{{\mathcal D}}

\newcommand{\Ucal}{{\mathcal U}}

\newcommand{\Ocal}{{\mathcal O}}

\newcommand{\dmu}{\;\text{\rm d}\mu}

\newcommand{\rd}{{\text{\rm d}}}
\newcommand{\rw}{{\text{\rm w}}}

\newcommand{\reg}{{\text{\rm reg}}}

\newcommand{\rb}{{\text{\rm b}}}
\newcommand{\rg}{{\text{\rm g}}}
\newcommand{\rc}{{\text{\rm c}}}

\newcommand{\ddt}[1]{\frac{\text{\rm d}#1}{\text{\rm dt}}}
\newcommand{\supp}{\operatorname*{supp}}

\newcommand{\Vinner}[1]{(\!({#1})\!)_{H^1}}
\newcommand{\dual}[1]{\langle{#1}\rangle}
\newcommand{\Lim}{\operatorname*{\textsc{Lim}}_{T\rightarrow \infty}}

\begin{document}
\numberwithin{equation}{section}


\title[Stationary statistical solutions of the 3D Navier-Stokes equations]{Properties of stationary statistical solutions of the three-dimensional Navier-Stokes equations}

\author[C. Foias]{Ciprian Foias}
\author[R. Rosa]{Ricardo M. S. Rosa}
\author[R. M. Temam]{Roger M. Temam}

\address[C. Foias]{Department of Mathematics, Texas A\&M University,
  College Station, TX 77843, USA.}
\address[R. M. S. Rosa]{Instituto de Matem\'atica, Universidade Federal do Rio de Janeiro, 
  Caixa Postal 68530 Ilha do Fund\~ao, Rio de Janeiro, RJ 21945-970,
  Brazil.}
\address[R. M. Temam]{Department of Mathematics, Indiana University, Bloomington,
  IN 47405, USA}
  
\email[R. M. S. Rosa]{rrosa@ufrj.br}
\email[R. M. Temam]{temam@indiana.edu}

\date{\today}

\thanks{This work was partly supported by the National Science Foundation under the grants NSF-DMS-1206438 and NSF-DMS-1109784, by the Research Fund of Indiana University, and by the CNPq, Bras\'{\i}lia, Brazil, under the grants 303654/2013-9 and 200826/2014-0 and the cooperation project 490124/2009-7.}



\begin{abstract}
 The stationary version of a modified definition of statistical solution for the three-dimensional incompressible Navier-Stokes equations introduced in a previous work  is investigated. Particular types of such stationary statistical solutions and their analytical properties are addressed. Results on the support and carriers of these stationary statistical solutions are also given, showing in particular that they are supported on the weak global attractor and are carried by a more regular part of the weak global attractor containing Leray-Hopf weak solutions which are locally strong solutions. Two recurrence-type results related to these measures are also proved.
\end{abstract}

\maketitle


\section{Introduction}

In a recent work \cite{frtssp1}, we have introduced and studied a modified definition of statistical solution for the Navier-Stokes equations. It is our aim here to investigate the particular case in which such statistical solutions do not vary with time, the so-called stationary statistical solutions.

The concept of statistical solution is connected with the dynamic behavior of the system and is directly related to the notion of ensemble average in the conventional theory of turbulence. More precisely, a statistical solution represents the evolution of the probability distribution of the state of the system. When this probability distribution does not vary in time, we have a stationary statistical solution. A stationary statistical solution may be viewed as a generalization of the concept of invariant measure, in the sense that the underlying system (the three-dimensional Navier-Stokes equations) is not known to generate a well-defined semigroup. 

A stationary statistical solution is an important object for the understanding of the asymptotic behavior of the system, with particular relevance for turbulence in statistical equilibrium in time. As already mentioned in \cite{frtssp1}, it is our belief that a better understanding of statistical solutions are of fundamental importance for a rigorous mathematical approach to the theory of turbulence.

There are essentially two main notions of statistical solution, one introduced by Foias and Prodi \cite{foias72,foiasprodi76} and the other, by Vishik and Fursikov \cite{vishikfursikov78}. The definition recently introduced in \cite{frtssp1} bridges these two notions, being based on a modification of the definition given by Vishik and Fursikov, and which becomes a particular case of a statistical solution in the sense of Foias and Prodi. 

The first step in the construction given in \cite{frtssp1} is to define a measure in the space of trajectories, in a way similar to the definition of a space-time statistical solution given by Vishik and Fursikov. This yields what we call a \emph{Vishik-Fursikov measure}. Then, projecting this measure into the phase space at each time yields a family of measures which we term a \emph{Vishik-Fursikov statistical solution}. A Vishik-Fursikov statistical solution is a particular type of what we call more generally a \emph{statistical solution}, which is the solution in the original sense of Foias and Prodi. The Vishik-Fursikov statistical solutions are more amenable to analysis and have a number of useful properties as showed in \cite{frtssp1} and as explored here.

In \cite{frtssp1}, we studied mostly the initial value problem, i.e. given an initial probability distribution for the velocity field, the aim was to define and prove the existence of a Vishik-Fursikov statistical solution with that initial probability distribution as the initial condition. This solution is a particular family, parametrized by the time variable, of probability distributions for the velocity field in subsequent times. Corresponding results for the initial value problem for the other previously defined notions of statistical solutions were given in \cite{foias72,vishikfursikov78}. 

In the current work, we investigate the case of statistical equilibrium in time, i.e. in which these probability distributions do not depend on time. We start in Section \ref{secprelim} by recalling the essential concepts and results from the mathematical theory of the Navier-Stokes equations, measure theory, generalized limits, and ergodic theory. We also recall some of the functional setting and results from the first part of the work \cite{frtssp1}, on time-dependent statistical solutions, and revisit the notion of weak global attractor and of its regular parts. Then, we define the multivalued solution map associated with the Navier-Stokes equations and study some dynamic sets such as orbits and sections of orbits, addressing their measurability, which is a delicate issue due to the lack of a well-defined semigroup. In Section \ref{secstatsol}, we return to the first part of the work \cite{frtssp1} on time-dependent statistical solutions and borrow from there the main concepts and results that we need here. 

In Section \ref{secsss} we start our main investigation on stationary statistical solutions. We first recall (Definition \ref{defsss}) the original definition of a \emph{stationary statistical solution} as given in \cite{foias73}, which arises when the family of probability distributions of a statistical solution does not change with time. We then define a \emph{Vishik-Fursikov stationary statistical solution} (Definition \ref{vfinvmeasdef}), which is obtained as the family of time-projections of an \emph{invariant} Vishik-Fursikov measure, in the sense of being invariant for the time-translation semigroup defined in the space of trajectories. A similar but potentially different type of stationary statistical solution is obtained when the family of time-projections of a Vishik-Fursikov measure does not depend on time, without the Vishik-Fursikov measure being known wheter it is invariant or not for the translation semigroup (see Remark \ref{rmkdiffstatvfsol}). This is what we call a \emph{stationary Vishik-Fursikov statistical solution}. Both are particular types of stationary statistical solutions, and also particular types of a Vishik-Fursikov statistical solution. We address more carefully the connection between these two notions and show that they actually coincide with each other (Theorem \ref{stationaryisstationary}).

Another particular type of stationary statistical solution is obtained through the limit of time averages of weak solutions, yielding what is called a \emph{time-average stationary statistical solution} (Definition \ref{deftimeavesss}; see also the references \cite{foiastemam75,fmrt2001a}). Similarly, invariant Vishik-Fursikov measures can also be obtained through the limit of time-averages of weak solutions (Definition \ref{deftimeaveVFinvmea}), and their projection yields a \emph{time-average Vishik-Fursikov stationary statistical solution} (Definition \ref{defvftimeaverageVFsss}).

Section \ref{seccarrier} is devoted to the study of the support of the invariant Vishik-Fursikov measures and of the Vishik-Fursikov stationary statistical solutions. We essentially prove that invariant Vishik-Fursikov measures are carried by the set of trajectories that exist and are uniformly bounded globally in time (Theorems \ref{thmcarriervfmeasure} and \ref{thmregularityofinvariantvfmeasures}), both in the past and in the future, with a corresponding result for the projected version of these measures, i.e. for the Vishik-Fursikov stationary statistical solutions (Corollary \ref{regrho0}). This allows us in particular to obtain some estimates in different norms for general Vishik-Fursikov stationary statistical solutions (Theorem \ref{corboundsrho0}) and to show that they are carried by the weak global attractor of the Navier-Stokes equations (Theorem \ref{carrierofrho0inaw}), two facts that were already known for time-average stationary statistical solutions. We then prove that any time-average stationary statistical solution is a time-average Vishik-Fursikov stationary statistical solution (Theorem \ref{thmtimeaveFPareVF}), and that they are carried by the $\omega$-limit set of the associated weak solution (Corollary \ref{timeavecarriedbyomegalimit}).

Next, in Section \ref{seclocalreg}, we address some local regularity results for Vishik-Fursikov stationary statistical solutions and Vishik-Fursikov measures. This is a step towards the \emph{Prodi invariance conjecture}, which is related to the support of a time-average stationary statistical solution being more regular in some sense, belonging to a space in which the solutions are strong and unique, globally in time. It is a kind of asymptotic regularity result in average, for the solutions of the 3D Navier-Stokes equations. We do not prove that such a stationary statistical solution is carried by this regular set, but we do prove that a Vishik-Fursikov stationary statistical solution is carried by a set in which the solutions are strong solutions at least locally in time (Theorem \ref{vfssscarriedbyaregprime}), with a similar result for a Vishik-Fursikov measure (Theorem \ref{invariantvfcarriedbywcalregprime}), thus partially solving the Prodi conjecture. 

In Section \ref{secaccretion} we address the accretion property of these measures (see Definition \ref{defaccretive}). We go back to time-dependent statistical solutions and prove an accretion property valid for any Vishik-Fursikov statistical solution (Theorem \ref{accretionvfss}). This automatically yields the accretion property for Vishik-Fursikov stationary statistical solutions (Corollary \ref{rho0accretive}), extending the previous result only known to be valid for time-average measures. Our proof also simplifies the proof given in the previous case (Corollary \ref{timeaverageaccretiveagain}).

Finally, in Section \ref{sectionrecurrence}, we study some recurrence properties of the flow. First, we prove a recurrence result for accretive measures of multivalued evolutionary systems (Theorem \ref{recurrenceforaccretivemeasures}). This result applies, in particular, to any accretive stationary statistical solution, and it is an adaptation of the classical Poincar\'e  Recurrence Theorem, known for invariant measures of semigroups.  As a consequence of the recurrence result we find that the support of an arbitrary accretive measure $\mu$ for the 3D Navier-Stokes equations is made only of points which are nonwandering with respect to the Leray-Hopf weak solutions (see Remark \ref{rmknonwandering}).  Secondly, in the particular case of Vishik-Fursikov stationary statistical solutions, the recurrence result is slightly improved (Theorem \ref{recurrenceforvfmeasurephasespace}). 

Concerning the relation with ergodic theory, there is, in fact, much more to be said for invariant Vishik-Fursikov measures, since they are classical invariant measures, in the trajectory space, with respect to the translation semigroup. Thus, all the classical results in ergodic theory apply, and most of the work left is to translate the result to the phase space. This has been exploited in our recent work \cite{frttimeaverage}. In particular, we obtain in \cite{frttimeaverage}, once we have the right framework developed here, that, at least for almost every Leray-Hopf weak solution with respect to any invariant Vishik-Fursikov measure, the generalized limit of time averages of weak solutions is actually a classical limit. See, also, the Remarks \ref{rmkergodicityonwcal} and \ref{wcalreginftyergodic}.

For the casual reader our presentation may at times seem pedantic, but from our personal experience working with this rigorous approach to the statistics of fluid dynamics we became aware that, without a very careful treatment, there are many lurking pitfalls.

\section{Preliminaries}
\label{secprelim}

\subsection{The Navier-Stokes equations and its mathematical setting}
\label{NSEsettingsec}

We recall, in this section, some fundamental and classical results about the individual solutions of the three-dimensional Navier-Stokes equations, which can be found, for instance, in \cite{constantinfoias,fmrt2001a,lady63,temam,temam2,temam3}.

We consider the three-dimensional incompressible Navier-Stokes equations, which we write in the Eulerian form as
\begin{equation}
  \label{nseeq}
  \frac{\partial \bu}{\partial t} - \nu \Delta \bu
       + (\bu\cdot\bnabla)\bu + \bnabla p = \bbf, \qquad
      \bnabla \cdot \bu = 0.
\end{equation}
The space variable is denoted by $\bx=(x_1,x_2,x_3)$, while $t$ represents the time variable. The three-dimensional velocity field is denoted by $\bu=(u_1,u_2,u_3)$, and is a function of $t$ and $\bx$; the term $\bbf=\bbf(\bx)$ represents the mass density of volume forces applied to the fluid and is assumed to be time-independent; the parameter $\nu>0$ is the kinematic viscosity; and the kinematic pressure is denoted by $p=p(t,\bx)$. 

We consider two kinds of boundary conditions: periodic and no-slip. In the periodic case the flow is assumed to be periodic in the space variables, with period $L_i$ in each direction $x_i$, $i=1,2,3$, and the fluid domain is set to $\Omega =\Pi_{i=1}^3(0,L_i)$. For simplicity, since the equations for the averages of the velocity field can be solved easily even when the averages of both the initial velocity field and the forcing term are nonzero (see e.g. \cite{fmrt2001a}), we assume that they vanish over $\Omega$\footnote{When the space average $\bar\bu$ of $\bu$ does not vanish, $\bar\bu$ is nevertheless constant in time, and the difference $\bu' = \bu - \bar\bu$ satisfies \eqref{nseeq} except for the addition of lower order terms involving $\bar\bu$. In this case, all that follows applies without any significant change. Hence, in the end, our results hold without significant change even when $\bar\bu\neq 0$.}, i.e.
\[ \int_\Omega \bu(\bx,t) \;\rd\bx = 0,
        \qquad \int_\Omega \bbf(\bx) \;\rd\bx = 0.
\]
In the case of no-slip, the flow is considered on a bounded domain $\Omega\subset\RR^3$, with smooth boundary $\partial\Omega$, with $\bu=0$ on $\partial \Omega$. Other homogeneous boundary conditions can be treated similarly, while non-homogeneous boundary conditions can be treated with the help of a background field.

For the function spaces associated with the boundary conditions, we start with a space of test functions $\Vcal$ appropriate to each case. In the space-periodic case, this space takes the form
\[ \Vcal=\left\{\bu=\bw|_\Omega; \;
     \parbox{4.2in}{ $\bw \in \Ccal^\infty(\RR^3)^3, \;\bnabla\cdot\bw=0,
                \;\int_\Omega \bw(\bx)\;\rd\bx = 0$,
        $\bw(\bx)$ is periodic with period $L_i$ in each
       direction $x_i$}
    \right\},
\]
while, in the no-slip case, one considers
\[ \Vcal=\left\{\bu\in \Ccal_\rc^\infty(\Omega)^3; \;
      \bnabla\cdot\bu = 0 \right\},
\]
where $\Ccal_\rc^\infty(\Omega)$ denotes the space of infinitely-differentiable real-valued functions with compact support on $\Omega$.

In either case, the space $H$ denotes the completion of $\Vcal$ in the norm $L^2(\Omega)^3$ of square-integrable vector fields, while the space $V$ denotes the completion of $\Vcal$ in the Sobolev norm of $H^1(\Omega)^3$ of the vector fields which belong to $L^2(\Omega)^3$ along with their partial derivatives in space. We identify $H$ with its dual and consider the dual space $V'$ of $V$, so that $V\subseteq H \subseteq V'$, with the injections being continuous, and each space dense in the following one. In either case, the space $V$ is compactly included in $H$.

The inner products in $H$ and $V$ are defined respectively by
\[ (\bu,\bv)_{L^2} = \int_\Omega \bu(\bx)\cdot\bv(\bx) \;\rd\bx,
     \quad \Vinner{\bu,\bv} = \int_\Omega \sum_{i=1,2,3}
        \frac{\partial \bu}{\partial x_i}
          \cdot \frac{\partial \bv}{\partial x_i}\; \rd\bx,
\]
and the associated norms by $|\bu|_{L^2}=(\bu,\bu)_{L^2}^{1/2}$, $\|\bu\|_{H^1}=\Vinner{\bu,\bu}^{1/2}$. We denote the duality product between $V$ and $V'$ by
\[ \dual{\bu,\bv}_{V',V}, \quad \forall \bu\in V', \;\bv\in V,
\]
and the usual norm in $V'$ is denoted by $\|\bu\|_{V'}$.

The Stokes operator $A: V \rightarrow V'$ is defined by duality through the formula
\[ \dual{A\bu,\bv}_{V',V} = \Vinner{\bu,\bv}, \quad \forall \bu,\bv\in V.
\]
The restriction of the operator $A$ to $D(A) = \{\bu\in V; \; A\bu\in H\}$ is a positive definite unbounded self-adjoint closed operator with compact inverse. We know that $|A\bu|_{L^2}$ is a norm on $D(A)$ equivalent to the $H^2$-norm. The operator $A$ has a countable number of eigenvalues $\{\lambda_j\}_{j\in \NN}$, counted according to their multiplicity, in increasing order, with each eigenvalue $\lambda_j$ associated with an eigenfunction $\bw_j$. The Galerkin projector onto the space spanned by the eigenfunctions associated with the first $m$ eigenvalues is denoted by $P_m$. The first eigenvalue $\lambda_1$ is positive and yields the optimal constant for the Poincar\'e inequality:
\begin{equation}
  \label{poincareineq}
  \lambda_1 |\bu|_{L^2}^2 \leq \|\bu\|_{H^1}^2, \quad \forall \bu\in V.
\end{equation}

We also recall the Agmon inequality
\begin{equation}
  \label{agmonineq}
   |\bu|_{L^\infty} \leq c_1\|\bu\|_{H^1}^{1/2}|A\bu|_{L^2}^{1/2}, \quad \forall \bu\in D(A),
\end{equation}
with a suitable non-dimensional constant $c_1$ depending only on the shape of the domain $\Omega$.

For the nonlinear term, we define the trilinear form
\[ b(\bu,\bv,\bw) = \int_\Omega (\bu\cdot\bnabla)\bv\cdot \bw \;\rd\bx, 
     \quad \bu,\bv,\bw\in V,
\]
which is continuous on $V$. By duality, a bilinear operator $B:V\times V\rightarrow V'$ is defined according to 
\[ \dual{B(\bu,\bv),\bw}_{V',V} = b(\bu,\bv,\bw), \quad\forall\bu,\bv,\bw\in V.
\]

Then, we obtain the following functional equation formulation for the time-dependent velocity field $\bu=\bu(t)$ corresponding to the function $\bx\in\Omega\mapsto \bu(\bx,t)$ at each time $t$:
\begin{equation}
  \label{nseeqfunctional}
  \ddt{\bu} + \nu A\bu + B(\bu,\bu) = \bbf.
\end{equation}

Since we are interested in stationary statistical solutions, we consider the autonomous case, with the following standing hypothesis on the time-independent forcing term:
\begin{equation}
  \bbf \in H.
\end{equation}
A non-dimensional parameter associated with the strength of the forcing term is the Grashof number
\begin{equation}
  \label{defgrashof}
  G = \frac{|\bbf|_{L^2}}{\nu^2\lambda_1^{3/4}}.
\end{equation}

Given a subset $X$ of $H$, we denote by $X_\rw$ this subset endowed with the weak topology of $H$. In particular, $H_\rw$ denotes the space $H$ endowed with its weak topology. The closed ball of radius $R$ in $H$ is denoted by $B_H(R)$.
Since $H$ is a separable Hilbert space, the space $B_H(R)_\rw$ is a completely metrizable metric space (i.e. it is a topological space such that there exists at least one compatible metric with which it is complete). 

Using this framework, we state the notion of Leray-Hopf weak solution as follows.
\renewcommand{\theenumi}{\roman{enumi}}
\begin{defs}
  \label{deflerayhopfweaksolution}
  A \emph{(Leray-Hopf) weak solution} on a time interval $I\subset\RR$
  is defined as a function $\bu=\bu(t)$ on $I$ with values in $H$
  and satisfying the following properties:
  \begin{enumerate}
    \item \label{lhuinhv} $\bu\in L_\loc^\infty(I;H)\bigcap L_\loc^2(I;V)$;
    \item \label{lhutinvprime} $\partial\bu/\partial t \in L_\loc^{4/3}(I;V')$;
    \item \label{lhuinhw} $\bu\in \Ccal_\loc(I;H_w)$, i.e. $\bu$ is weakly continuous 
      in $H$, which means
      $t\mapsto (\bu(t),\bv)$ is continuous from $I$ into $\RR$, for
      every $\bv\in H$;
    \item \label{lhfnse} $\bu$ satisfies the functional equation \eqref{nseeqfunctional}
      in the distribution sense on $I$, with values in $V'$
    \item \label{lheineq} 
      For almost all $t'$ in $I$, $\bu$ satisfies the following energy inequality:
      \begin{equation}
        \label{energyinequalityintegralform}
          \frac{1}{2}|\bu(t)|_{L^2}^2 + \nu \int_{t'}^t \|\bu(s)\|_{H^1}^2 \;\rd s
            \leq \frac{1}{2}|\bu(t')|_{L^2}^2 + \int_{t'}^t (\bbf(s),\bu(s))_{L^2}\;\rd s,
      \end{equation}     
      for all $t$ in $I$ with $t>t'$. These times $t'$ are characterized as the 
      points of strong continuity in $H$ from the right for $\bu$, and their set is of
      total measure.
    \item \label{lhinitialcont} If $I$ is closed and bounded on the left, with its left
      end point denoted by $t_0$, then the solution
      is strongly continuous in $H$ at $t_0$ from the right, i.e.
      $\bu(t)\rightarrow \bu(t_0)$ in $H$ as $t\rightarrow t_0^+$.
  \end{enumerate}
\end{defs}

From now on, for notational simplicity, \emph{a weak solution will always
mean a Leray-Hopf weak solution.} 

Using the Cauchy-Schwarz and Poincar\'e inequalities, condition \eqref{lheineq} yields
\begin{equation}
   \label{L2enstrophyestimatefinh}
   |\bu(t)|_{L^2}^2 + \nu \int_{t'}^t \|\bu(s)\|_{H^1}^2 \;\rd s
     \leq |\bu(t')|_{L^2}^2 + \frac{1}{\nu\lambda_1} |\bbf|_{L^2}^2 (t-t'),
\end{equation}    
for $t'$ and $t$ as in \eqref{lheineq}. A weak solution also satisfies
\begin{equation}
  \label{energyestimate}
  |\bu(t)|_{L^2}^2 \leq |\bu(t')|_{L^2}^2 e^{-\nu\lambda_1 (t-t')}
     + \frac{1}{\nu^2\lambda_1^2} |\bbf|_{L^2}^2
         \left(1-e^{-\nu\lambda_1 (t-t')}\right),
\end{equation}
for almost all $t'$ in $I$ and all $t$ in $I$ with $t'<t$. The allowed times
$t'$ are again the points at which the solution is strongly continuous from the right.
The allowed times $t'$ can be characterized as the Lebesgue points of the 
function $t\mapsto |\bu(t)|_{L^2}^2$ in the sense that
\begin{equation}
  \label{lebesguepoints}
  \lim_{\tau\rightarrow 0^+} \frac{1}{\tau} \int_{t'}^{t'+\tau} |\bu(t)|_{L^2}^2 \;\rd t
      = |\bu(t')|_{L^2}^2.
\end{equation}
Since $t\mapsto |\bu(t)|_{L^2}^2$ is locally integrable, these Lebesgue points
form a set of full measure. In the case of a weak solution on an interval closed and bounded on the left, with left end point $t_0$, then condition \eqref{lhinitialcont} implies that the point $t_0$ is a point of strong continuity from the right. In this case, estimate \eqref{energyestimate} is also valid for the initial time $t'=t_0$, which is a crucial property for obtaining a~priori estimates.

Let $R_0$ be given by
\begin{equation}
  \label{defR0}
  R_0 = \frac{|\bbf|_{L^2}}{\nu\lambda_1} = \frac{\nu}{\lambda_1^{1/4}}G.
\end{equation}
The energy estimate \eqref{energyestimate} implies the following invariance
property for any ball of radius larger than $R_0$: If $\bu$ is a weak solution
on $[0,\infty)$ and $R\geq R_0$, then
\begin{equation}
  \label{invarianceBR}
  \bu(0)\in B_H(R) \Rightarrow \bu(t)\in B_H(R), \;\forall t\geq 0.
\end{equation}

The notion of strong solution is also of interest in this work:
\begin{defs}
  A (Leray-Hopf) weak solution on an arbitrary interval $I$ is called regular
  or a strong solution if it satisfies furthermore
  \begin{itemize}
    \item[(vii)] $\bu\in \Ccal(I,V)$, i.e. $t\mapsto \bu(t)$ is strongly continuous from $I$ into $V$.
  \end{itemize}
\end{defs}

It is well known that given any initial time $t_0\in \RR$ and any initial condition $\bu_0$ in $H$, there exists at least one global weak solution $\bu$ on $[t_0,\infty)$ satisfying $\bu(t_0)=\bu_0$. It is also known that if $\bu_0$ belongs to $V$, then there exists a local strong solution $\bu$, defined on some interval $[t_0,t_1)$, $t_0<t_1\leq \infty$, with $\bu(t_0)=\bu_0$. Such strong solution is unique among the class of all weak solutions, i.e. when a strong solution $\bu$ exists on an interval of the form $[t_0, t_1)$, with $t_0<t_1\leq \infty$, then any weak solution on $[t_0,t_1)$ with $\bv(t_0)=\bu_0$ must coincide with $\bu$. Moreover, a strong solution on an open interval $(t_2,t_3)$, with $-\infty\leq t_2<t_3\leq \infty$, is analytic in time as a function from $(t_2,t_3)$ into $D(A)$.

Concerning the regularity points of a weak solution we introduce the following definition.

\begin{defs}
  \label{regandsingpoints}
  Let $\bu$ be a weak solution defined on an interval $I\subset \RR$. Then a point $t\in I$
  is called \emph{singular} if $\bu(t)\in H\setminus V$ and is called \emph{regular} if $\bu(t)\in V$.
  Morever, a regular point $t\in I$ is called a \emph{point of interior regularity} if
  there exists a $\delta>0$ such that $(t-\delta,t+\delta)$ is included in $I$ and $\bu$
  restricted to $(t-\delta,t+\delta)$ is a strong solution.
\end{defs}

Given a weak solution $\bu$ on an interval $I\subset\RR$, the set of interior regularity points of $\bu$ is an open set which is dense in $I$ and of full measure in $I$ (see e.g. \cite{leray,scheffer,foiastemam79,temam3,frt2010c}). Being open, the set of interior regularity points can be written as a countable collection of disjoint open intervals, say $\Ocal = \cup_{k}(\alpha_k,\beta_k)$.  Each interval $(\alpha_k,\beta_k)$ is an interval of maximal regularity of $\bu$, within $I$. If $\beta_k$ belongs to the interior of $I$, then $\|\bu(t)\|_{H^1}$ must blow up as $t$ approaches $\beta_k$ from the left, otherwise we would be able to extend the interval of regularity beyond $\beta_k$. 

Following \cite{foiasguillopetemam1981}, this structure of the set of interior regularity points can be used to yield an integral estimate for the norm of the weak solutions in $D(A)$; more precisely, an estimate in $L^{2/3}_\loc(I,D(A))$. We revisit this estimate here, in non-dimensional form. If $\bu$ is a strong solution on an interval $J$, then the enstrophy equation holds on $J$:
\begin{equation}
  \label{enstrophyequation}
  \frac{1}{2}\frac{\rd}{\rd t}\|\bu(t)\|_{H^1}^2 + \nu |A\bu(t)|_{L^2}^2 + b(\bu,\bu,A\bu) = (\bbf(t),A\bu(t))_{L^2}.
\end{equation}
Using H\"older's inequality with $L^6$, $L^3$, and $L^2$, respectively, followed by Sobolev's, interpolation and Young's inequalities, we obtain the following inequality for the trilinear term:
\begin{equation}
   \label{buuauest}
   |b(\bu,\bu,A\bu)| \leq \frac{\nu}{4}|A\bu|_{L^2}^2+\frac{c_2}{\nu^3}\|\bu\|_{H^1}^6, \quad \forall \bu\in D(A),
\end{equation}
for a suitable non-dimensional constant $c_2$ depending only on the shape of $\Omega$, and which we may assume to be greater than $1$. Using the estimate \eqref{buuauest}, we find the well-known inequality
\begin{equation}
  \label{enstrophyineq}
  \frac{1}{2}\frac{\rd}{\rd t}\|\bu(t)\|_{H^1}^2 + \frac{\nu}{2} |A\bu|_{L^2}^2 
        \leq \frac{1}{\nu}|\bbf|_{L^2}^2 + \frac{c_2}{\nu^3}\|\bu\|_{H^1}^6.
\end{equation}
Using that the forcing term is time-independent, we rewrite \eqref{enstrophyineq} as
\[ \frac{\rd}{\rd t} \left( (\nu|\bbf|_{L^2})^{2/3} + \|\bu\|_{H^1}^2\right) + \nu |A\bu|_{L^2}^2  \leq \frac{c_2}{\nu^3} \left( (\nu|\bbf|_{L^2})^{2/3} + \|\bu\|_{H^1}^2\right)^3.
\]
We divide this inequality by
\[  y(t) = (\nu|\bbf|_{L^2})^{2/3} + \|\bu\|_{H^1}^2,
\]  
to find that
\begin{equation}
  \label{inverseenstrophyineq}
  \frac{\rd}{\rd t} \left(-\frac{1}{y}\right) + \frac{\nu |A\bu|_{L^2}^2}{y^2} \leq \frac{c_2}{\nu^3}y.
\end{equation}
Integrating \eqref{inverseenstrophyineq} from $\alpha$ to $\beta$, where $\alpha,\beta \in J$, with $\alpha<\beta$, we have
\[ \frac{1}{y(\alpha)} + \int_\alpha^\beta \frac{\nu |A\bu|_{L^2}^2}{y^2}\;\rd s \leq \frac{1}{y(\beta)} + \frac{c_2}{\nu^3} \int_\alpha^\beta y \;\rd s.
\]
Then, we obtain
\begin{align*}
  \int_\alpha^\beta |A\bu|_{L^2}^{2/3} \;\rd s & = \int_\alpha^\beta \left(\frac{|A\bu|_{L^2}}{y}\right)^{2/3}y^{2/3} \;\rd s
  \leq \left( \int_\alpha^\beta \frac{|A\bu|_{L^2}^2}{y^2}\;\rd s\right)^{1/3} \left( \int_\alpha^\beta y \;\rd s \right)^{2/3} \\
    & \leq \left(\frac{1}{\nu y(\beta)} + \frac{c_2}{\nu^4} \int_\alpha^\beta y \;\rd s\right)^{1/3} \left( \int_\alpha^\beta y \;\rd s \right)^{2/3} \\
    & = \frac{1}{\nu^{1/3}y(\beta)^{1/3}}\left( \int_\alpha^\beta y \;\rd s \right)^{2/3} + \frac{{c_2}^{1/3}}{\nu^{4/3}} \int_\alpha^\beta y \;\rd s \\
    & \leq \frac{\nu^{5/3}}{3y(\beta)} + \frac{c_3}{\nu^{4/3}}\int_\alpha^\beta y \;\rd s,
\end{align*}
where $c_3 = 2/3+c_2^{1/3}$. Hence, substituting for $y$, we obtain
\begin{multline}  
  \label{estDAtwothirdslocal}
  \int_\alpha^\beta |A\bu|_{L^2}^{2/3} \;\rd s \leq \frac{\nu^{5/3}}{3((\nu|\bbf|_{L^2})^{2/3} + \|\bu(\beta)\|_{H^1}^2)} \\ + \frac{c_3}{\nu^{4/3}}\left(\nu^{2/3} |\bbf|_{L^2}^{2/3}(\beta-\alpha) + \int_\alpha^\beta  \|\bu\|_{H^1}^2 \;\rd s\right),
\end{multline}
for all $\alpha<\beta$ within an interval of strong regularity.

Consider, now, a weak solution $\bu$ on an interval $I$. Let $t'\in I$ be a point of strong continuity from the right for $\bu$, and let $t\in I$, $t>t'$. Consider the open set $\Ocal=\cup_{k}(\alpha_k,\beta_k)$ of interior regularity points of $\bu$ within $(t', t)$. Within each interval $(\alpha_k,\beta_k)$, the estimate \eqref{estDAtwothirdslocal} holds. If $\beta_k<t$, then $\|\bu(\beta)\|_{H^1}$ must blow up as $\beta$ approaches $\beta_k$, otherwise the interval of regularity could be extend beyond $\beta_k$. If $\beta_k=t$, then $\|\bu(\beta)\|_{H^1}$ may or may not blow up. In any case, we let $\alpha\rightarrow \alpha_k^+$ and $\beta\rightarrow \beta_k^-$ in the estimate \eqref{estDAtwothirdslocal} to find the upper bound
\[  \int_{\alpha_k}^{\beta_k} |A\bu|_{L^2}^{2/3} \;\rd s \leq M_k + \frac{c_3}{\nu^{4/3}}\left(\nu^{2/3} |\bbf|_{L^2}^{2/3}(\beta_k-\alpha_k) + \int_{\alpha_k}^{\beta_k}  \|\bu\|_{H^1}^2 \;\rd s\right),
\]
where $M_k=0$, if $\beta_k<t$, and $M_k=\nu/(3|\bbf|_{L^2}^{2/3})$, if $\beta_k=t$.

Summing up in $k$, using that $\Ocal=\cup_k (\alpha_k,\beta_k)$ is of full measure in $(t',t)$, and that $\beta_k=t$ at most for one index $k$, we obtain the desired result, which we state as follows.

\begin{prop}
  Let $\bu$ be a Leray-Hopf weak solution on an arbitrary interval $I$. Then, 
\begin{equation}
  \label{estDAtwothirdsglobal}
  \int_{t'}^t |A\bu|_{L^2}^{2/3} \;\rd s \\ \leq \frac{\nu}{3|\bbf|_{L^2}^{2/3}} + \frac{c_3}{\nu^{4/3}}\left(\nu^{2/3} |\bbf|_{L^2}^{2/3}(t-t') + \int_{t'}^t  \|\bu\|_{H^1}^2 \;\rd s\right),
\end{equation}
for any $t', t\in I$, $t'<t$, with $t'$ a point of strong continuity from the right for $\bu$.
\end{prop}
As mentioned earlier, such an estimate was proved in \cite[Theorem 3.1]{foiasguillopetemam1981}, with a bound that is not made explicit there.

Next, we recall two results which will be useful in the sequel. First, 
we will need to paste solutions together, according to the following result, which holds thanks in part to the condition of strong continuity from the right, of weak solutions, at the initial time:

\begin{lem}[Pasting Lemma]
  \label{pastinglemma}
  Let $\bu^{(1)}$ be a weak solution on an interval $(t_1,t_2]$ and
  $\bu^{(2)}$ be a weak solution on an interval $[t_2,t_3)$, with
  $-\infty\leq t_1<t_2<t_3\leq \infty$ and $\bu^{(1)}(t_2)=\bu^{(2)}(t_2)$.
  Then the function
  \begin{equation}
    \label{concatenation}
      \tilde\bu(t) = \begin{cases}
        \bu^{(1)}(t), & t_1<t< t_2, \\
        \bu^{(2)}(t), & t_2 \leq t < t_3,
    \end{cases}
  \end{equation}
  is a weak solution on $(t_1,t_3)$. 
\end{lem}
\medskip

The proof of this result is simple; see \cite[Lemma 2.4]{frt2010c} for the details. 
The following compactness result is also useful. It follows from arguments used in the classical proofs of existence of weak solutions.
\begin{lem}[Weak Compactness]
  \label{convergenceofsolutions1}
  Let $\{\bu_j\}_{j\in\NN}$ be a sequence of weak solutions on
  some interval $I=(t_0,t_1)$, $-\infty \leq t_0 < t_1 \leq \infty$,
  and suppose this sequence is uniformly bounded in $H$.
  Then, there exists a subsequence $\{\bu_{j'}\}_{j'}$ and a weak
  solution $\bu(\cdot)$ on $I$ such that $\bu_{j'}$ converges to $\bu$
  in $H_\rw$, uniformly on any compact interval in $I$.
\end{lem}
\medskip

\subsection{Elements of measure theory}
\label{measuretheory}

The statistical solutions are Borel probability measures in appropriate topological spaces. With that in mind, we recall in this section a few facts from measure theory, especially in connection to topology. For the results mentioned here, we refer the reader to the works \cite{aliprantisborder,bourbaki69,brownpearcy,rudin,moschovakis,kuratowski}.

A \emph{measurable space} is a pair $(X,\Mcal)$ where $X$ is a set and $\Mcal$ is a $\sigma$-algebra of subsets of $X$ called the \emph{measurable sets}. A \emph{measure space} is a triplet $(X,\Mcal,\mu)$ where $(X,\Mcal)$ is a measurable space and $\mu$ is a measure. A \emph{finite measure} is a measure such that $\mu(X)<\infty$, and a \emph{probability measure} is a finite measure with $\mu(X)=1$. A probability space is a measure space $(X,\Mcal,\mu)$ in which $\mu$ is a probability measure. A \emph{null set} is a measurable set with measure zero, and a measure is said to be \emph{complete} if any subset of a null set is also measurable.

Any measure is continuous from above and from below in the following sense (see e.g. \cite[Theorem 10.8]{aliprantisborder}). If $E_1 \supset E_2 \supset ...$ is a monotone decreasing sequence of measurable sets with $\mu(E_1)<\infty$, then $\cap_n E_n$ is measurable and $\mu(\cap_n E_n) = \lim_{n\rightarrow \infty} \mu(E_n)$, and if $E_1\subset E_2 \subset ... $ is a monotone increasing sequence of measurable sets, then $\cup_n E_n$ is measurable and $\mu(\cup_n E_n) = \lim_{n\rightarrow \infty} \mu(E_n)$.

When the set $X$ is a topological space, a natural $\sigma$-algebra to consider is the family of \emph{Borel subsets} of $X$, which is defined as the smallest $\sigma$-algebra containing the open sets of $X$. This $\sigma$-algebra is denoted by $\Bcal(X)$. We consider only topological spaces which are \emph{Hausdorff}, i.e. in which any pair of distinct points can be separated by disjoint open sets. If $X$ is a topological vector space, then the weak topology is also a natural topology to consider. In case the topological vector space is a separable Banach space then the Borel sets in the strong topology coincide with the Borel sets in the weak topology. In particular the Borel sets in the space $H$ coincide with the Borel sets in $H_\rw$. Moreover, Borel sets in $V$ or in $D(A)$ are also Borel sets in $H$ (see \cite[Appendix IV.A.1, p. 219-220]{fmrt2001a}).

A \emph{Borel measure} is a measure $\mu$ on a topological space $X$ defined on the Borel sets $\Bcal(X)$ of $X$. Given a Borel measure $\mu$ on a topological space $X$, the $\sigma$-algebra $\Bcal_\mu(X)$ is defined as the smallest $\sigma$-algebra containing the Borel sets and the subsets of Borel sets of $\mu$-measure zero. One has (see \cite[Theorem 1.8]{folland1999}) that $E\in\Bcal_\mu(X)$ if and only if there exists a Borel set $E_B$ and a subset $E_N$ of a Borel set of $\mu$-measure zero such that $E=E_B\cup E_N$. This representation of $E\in\Bcal_\mu(X)$ may not be unique but the $\mu$-measure of $E_B$ is independent of the representation, so that we can extend the Borel measure $\mu$ to a \emph{complete Borel measure} on $\Bcal_\mu(X)$ by defining $\mu(E)=\mu(E_B)$. Such a measure is called the \emph{completion} (or the \emph{Lebesgue extension}) of the Borel measure. When the Borel measure $\mu$ is $\sigma$-finite, the completion $\bar\mu$ coincides with the restriction to $\Bcal_\mu(X)$ to the Carath\'eodory extension of $\mu$ (see \cite[Chapter 1, Exercise 22]{folland1999} and also \cite[Chapter 10]{aliprantisborder}). The collection $\Bcal_\mu(X)$ of $\mu$-measurable sets is usually larger than the collection of Borel sets $\Bcal(X)$. \emph{In what follows, for the sake of notational simplicity, we still denote by $\mu$ the completion of a Borel measure $\mu$, and we call the elements in $\Bcal_\mu(X)$ as $\mu$-measurable sets.} 

A \emph{carrier} of a measure is any subset of \emph{full measure}, i.e. its complement is of zero measure. The \emph{support} of a Borel measure is the smallest closed set of full measure, i.e. $\supp \mu = \bigcap \left\{ F; \;F\subset X; \;F \text{ is closed in } X \text{ and } \mu (X\setminus F) = 0 \right\}.$

Given two measurable spaces $(X,\Mcal)$ and $(Y,\Ncal)$ and a function $f:X\rightarrow Y$, the function $f$ is said to be measurable, or $(\Mcal,\Ncal)$-measurable, if $f^{-1}(E)\in \Mcal$ for all $E\in \Ncal$.

When the target space of a sequence of measurable functions is metrizable (and the measure in the target space is the corresponding Borel measure) the following ``standard'' result holds (see \cite[Lemma 4.29]{aliprantisborder}):
\begin{equation}
  \label{limitofmeasurablefunctions}
  \parbox{5in}{The pointwise limit of a sequence of measurable functions from a measurable space into a metrizable space is measurable.}
\end{equation}

Given two topological spaces $X$ and $Y$ and a continuous function $f:X\rightarrow Y$, it follows that $f$ is a \emph{Borel map} in the sense that $f^{-1}(E)\in \Bcal(X)$ for all $E\in\Bcal(Y)$. Such continuous function is also $(\Bcal_\mu(X),\Bcal(Y))$-measurable, with respect to the completion of a Borel measure $\mu$ on $X$. However, given two completions $\mu$ and $\nu$ of Borel measures on $X$ and $Y$, respectively, the continuous function $f$ may not be \emph{measurable} from $(X,\Bcal_\mu(X),\mu)$ to $(Y,\Bcal_\nu(Y),\nu)$ since $f^{-1}(E)$ may not belong to $\Bcal_\mu(X)$ for all $E$ in $\Bcal_\nu(Y)$ (just take $f$ to be the identity, with $Y=X$ and two Borel measures $\mu$ and $\nu$ supported, say, on disjoint intervals $I$ and $J$, respectively, so that any non-Lebesgue-measurable subset within $I$ is a null $\nu$-measurable set but is not $\mu$-measurable, and vice-versa).

Given a topological space $X$, we denote by $\Ccal(X)$ the space of real-valued continuous function on $X$, by $\Ccal_\rb(X)$ the space of bounded real-valued continuous functions on $X$, and by $\Ccal_\rc(X)$ the space of compactly supported real-valued continuous functions on $X$.

The Kakutani-Riesz Representation theorem \cite{aliprantisborder,bourbaki69,rudin} asserts that a positive linear functional $f$ defined on the space of compactly supported continuous real-valued functions $\Ccal_\rc(X)$ on a \emph{locally compact Hausdorff space} $X$ (i.e. a Hausdorf topological space with the property that every point has a compact neighborhood), can be uniquely represented by a regular Borel measure $\mu$ on $X$, with 
\[ f(\varphi)=\int_X \varphi(\bu)\;\rd\mu(\bx), \quad \text{for all } \varphi \in \Ccal_\rc(X).
\]

The Stone-Weierstrass Theorem \cite{dunfordschwartz} asserts that if $X$ is a compact Haudorff space and if $\Acal$ is a closed sub-algebra of $\Ccal(X)$ that contains the unit element, then $\Acal=\Ccal(X)$ if and only if $\Acal$ distinguishes the points of $X$. We say that $\Acal$ distinguishes the points of $X$ when, given any points $t,s\in X$ such that $t\neq s$, there exists $f\in \Acal$ such that $f(t)\neq f(s)$. Then, if $\Scal$ is a sub-algebra of $\Ccal(X)$ that contains the unit element and distinguishes between the points of $X$, it follows that $\Scal$ is dense in $\Ccal(X)$.

Given two measurable spaces $(X,\Mcal)$ and $(Y,\Ncal)$, a measurable function $T:X\rightarrow Y$, and a probability measure $\mu$ on $(X,\Mcal)$, one obtains a probability measure $\mu_T$ on $Y$ by the formula $\mu_T(E)=\mu(T^{-1}(E))$, for all measurable subsets $E$ of $Y$. The measure $\mu_T$ is called the \emph{measure induced} on $Y$ from $\mu$ by $T$, and is sometimes denoted $T\mu$ or $\mu T^{-1}$ (see \cite[Section 13.12]{aliprantisborder}). It follows (see \cite[Theorem 13.46]{aliprantisborder}) that
\begin{equation}
  \label{changeofvariablesinducedmeasures}
  \int_Y \varphi(y) \;\rd\mu_T(y) = \int_X \varphi(T(x))\;\rd\mu(x), \qquad \forall \varphi\in L^1(\mu_T),
\end{equation}
and 
\begin{equation}
  \label{integrabilityinducedmeasures}
  \varphi\in L^1(\mu_T) \text{ if and only if $\varphi$ is $\mu_T$-measurable and $\varphi\circ T \in L^1(\mu)$}.
\end{equation}

If $X$ and $Y$ are two topological spaces, $\mu$ is a Borel measure on $X$, and $T:X\rightarrow Y$ is continuous, then $T\mu$ is a Borel measure on $Y$ and \eqref{changeofvariablesinducedmeasures} holds for any bounded continuous function $\varphi\in \Ccal_\rb(Y)$.

If $\mu$ is a regular Borel measure (as defined in \eqref{upperregmeas}-\eqref{lowerregmeas} below) on a locally compact Hausdorff space $X$, then $\Ccal_\rc(X)$ is dense in $L^p(\mu)$, for $1\leq p <\infty$ \cite[Theorem 13.9]{aliprantisborder}.

In the case $X$ and $Y$ are locally compact topological space, a continuous map $T:X\rightarrow Y$ induces an operator $L_T:\Ccal_\rc(Y)\rightarrow \Ccal_\rc(X)$ given by $L_T\varphi = \varphi\circ T$. Then, regarding a Borel probability measure $\mu$ on $X$ as an element of the dual space $\Ccal_\rc(X)'$, it is natural to view the induced measure $T\mu$ as $L_T^*\mu$, where $L_T^*:\Ccal_\rc(X)'\rightarrow \Ccal_\rc(Y)'$ is the adjoint of the operator $L_T$.

In a metrizable topological space $X$, the following statements concerning two Borel probability measures $\mu$ and $\nu$ are equivalent \cite[Theorem 15.1]{aliprantisborder}:
\begin{equation}
  \label{measequivmetric}
\begin{aligned}
   \mu=\nu 
   \Leftrightarrow & \mu(G)=\nu(G) \text{ for all open sets } G \\
   \Leftrightarrow & \mu(F)=\nu(F) \text{ for all closed sets } F \\
   \Leftrightarrow & \int_X \varphi(x)\;\rd\mu(x) = \int_X \varphi(x)\;\rd\nu(x), 
     \forall \varphi\in \Ccal_\rb(X) \\
   \Leftrightarrow & \int_X \varphi(x)\;\rd\mu(x) = \int_X \varphi(x)\;\rd\nu(x), \forall \varphi\in \Dcal, \\
    & \text{ where $\Dcal$ is any dense subset of $\Ccal_\rb(X)$} \\   
\end{aligned}
\end{equation}

In a metrizable topological space $X$ we say that a net $\{\mu_\alpha\}_\alpha$ of Borel probability measures on $X$ converges weak-star (see e.g. \cite[Section 15.1]{aliprantisborder}) to a Borel probability measure
$\mu$ on $X$ if 
\[ \int_X \varphi(x) \;\rd\mu_\alpha(x) \rightarrow \int_X \varphi(x) \;\rd\mu(x), 
  \;\forall \varphi\in\Ccal_\rb(X).
\]
This convergence is denoted by
\[ \mu_\alpha \stackrel{*}{\rightharpoonup} \mu.
\]

A Borel measure $\mu$ on a topological space $X$ is called \emph{regular} when
\begin{align}
  \label{upperregmeas}
  \mu(E) & = \inf\{\mu(O); \;O\supset  E, \;O \text{ open in } X\}, \text{ and} \\
  \label{lowerregmeas}
  \mu(E) & = \sup\{\mu(K); \;K\subset  E, \;K \text{ compact in } X\}.
\end{align}
The first relation is called \emph{upper regularity}, while the second one is called \emph{lower regularity}. 
The completion of a regular Borel measure is also a regular measure. The support of a regular Borel measure has the property that if $O$ is an open set and $O\cap \supp\mu\neq \emptyset$, then $\mu(O\cap \supp\mu)>0$ (this is essentially proved in \cite[Theorem 12.14]{aliprantisborder}, although their definition of support is slightly different).

A topological space is called a \emph{Polish space} when it is separable and completely metrizable. Polish spaces play an important role in measure theory. In particular,  \emph{any finite Borel measure on a Polish space is regular in the sense of \eqref{upperregmeas} and \eqref{lowerregmeas}} \cite[Theorem 12.7]{aliprantisborder}.

In our case, all the spaces $H$, $V$, $V'$ and $D(A)$ are Polish,
and so are the bounded, weakly closed subsets of $H$ endowed with the weak
topology, such as $B_H(R)_\rw$, $R>0$.  The space $H_\rw$, however, 
is not Polish. In fact, $H_\rw$ is separable and completely regular, but it is 
neither complete, nor metrizable, nor locally compact.\footnote{We recall 
here that a \emph{regular}
topological space is one in which any pair of a singleton and a closed set not 
containing this singleton can be separated by disjoint neighborhoods, while 
a \emph{completely regular} topological space is a regular space in which 
any pair of a singleton and a closed set not
containing this singleton can be separated by a function, i.e. there exists a continuous real-valued function defined on the space and which 
vanishes at the singleton and is equal to one everywhere on the closed set. 
The fact that $H_\rw$ is completely regular comes from
the fact that it is a topological vector space, hence it has a uniform structure 
\cite[Section I.1.4]{scheffer}, and is Hausdorff, and any Hausdorff topological space 
has a uniform structure if and only if it is completely regular 
\cite[Section B.6]{scheffer}.}

In a given topological space, a \emph{universally measurable set} is any set which is measurable with respect to any complete Borel measure defined on a $\sigma$-algebra containing the Borel sets. This notion is important when comparing different Borel measures and when taking continuous images of Borel sets, which may not be Borel anymore, but are still universally measurable. In a separable Banach space, since the Borel sets with respect to the weak and strong topologies coincide, the corresponding collections of universally measurable sets also coincide. 

An important notion related to universal measurability is that of an \emph{analytic set,} which is defined as a subset of a Polish space which is either empty or the continuous image of the \emph{Baire space} $\NN^\NN$ of sequences of natural numbers endowed with the product topology (which is itself Polish). A subset of a Polish space is \emph{coanalytic} if its complement is analytic. The following important facts can be found in \cite{aliprantisborder} (see also \cite{bourbaki69}).

\medskip
\textbf{Facts:}
\begin{enumerate} 
  \item \label{analyticcouniversallymeasurable} Any analytic or co-analytic subset of a Polish space is universally measurable
    \cite[Theorem 12.41]{aliprantisborder}; 
  \item Every Borel subset of a Polish space is analytic \cite[Section 12.5, pg. 446]{aliprantisborder}; 
  \item \label{continuousimageofanalyticsets} The continuous image of an analytic set from a Polish space into 
    a Polish space is analytic \cite[Section 12.5, pg. 446]{aliprantisborder}; 
  \item A subset of a Polish space is both analytic and coanalytic if and only if 
    it is a Borel set \cite[Corollary 12.27]{aliprantisborder};
  \item The collection of analytic sets in a Polish space is closed by countable 
    unions and countable intersections \cite[Theorem 12.25]{aliprantisborder}; 
  \item \label{universallymeasurablesigmaalgebra} The collection of universally measurable sets in a topological space is a $\sigma$-algebra 
    \cite[Section 12.5, pg. 456]{aliprantisborder}. 
\end{enumerate}

\subsection{Generalized limits}
\label{genlimsec}

In this section we recall the notion of generalized limit, which is used to define a special type of stationary statistical solution called a time-average stationary statistical solution (see \cite{bcfm1995,fmrt2001a}). A generalized limit is a linear continuous functional $\Lim$ on the space $\Ccal_\rb([0,\infty))$ of bounded real-valued functions on $[0,\infty)$ satisfying the properties
\begin{enumerate}
  \item $\displaystyle \Lim g(T) \geq 0, \;\forall g\in \Ccal_\rb([0,\infty))$, with $g(s)\geq 0, \;\forall s\geq 0$; and
  \item $\displaystyle \Lim g(T) = \lim_{T\rightarrow \infty} g(T), \;\forall g\in \Ccal_\rb([0,\infty))$, if the classical limit exists.
\end{enumerate}

The existence of such limits results from a simple application of the Hahn-Banach Theorem, extending the classical limit \cite{dunfordschwartz, brownpearcy}. A generalized limit has the following additional properties 
(see \cite{bcfm1995,brownpearcy,fmrt2001a}):
\begin{enumerate}
  \setcounter{enumi}{2}
  \item $\displaystyle \liminf_{T\rightarrow \infty} g(T) \leq \Lim g(T) \leq 
         \limsup_{T\rightarrow \infty} g(T),  \;\forall g\in \Ccal_\rb([0,\infty))$;
  \item $\displaystyle |\Lim g(T)| \leq \limsup_{T\rightarrow \infty} |g(T)|           \leq \sup_{t\geq 0} |g(T)|, \;\forall g\in \Ccal_\rb([0,\infty))$; and
  \item \label{genliminvariance} $\displaystyle \Lim \frac{1}{T}\int_0^T h(\tau)\;\rd \tau = \Lim \frac{1}{T}\int_0^T h(\tau + t) \;\rd \tau, \;\forall h\in \Ccal_\rb([0,\infty)), \;\forall t>0.$
\end{enumerate}

\subsection{Poincar\'e Recurrence Theorem.}
\label{secrergodic}

In this section, we recall the Poincar\'e Recurrence Theorem for continuous flows. We consider a probability space $(X,\Mcal,\mu)$ and a measurable semigroup $\{S(t)\}_{t\geq 0}$, which is a family of operators $S(t):X\rightarrow X$ with the properties that $S(0)$ is the identity in $X$; $S(t+s)=S(t)S(s)$, for all $t,s\geq 0$; and $S(t)^{-1}E\in \Mcal$ for every $E$ in the $\sigma$-algebra $\Mcal$ and all $t\geq 0$. We say that the semigroup $\{S(t)\}_{t\geq 0}$ is measure preserving, with respect to the measure $\mu$, or that $\mu$ is invariant by the semigroup $\{S(t)\}_{t\geq 0}$, when $\mu(S(t)^{-1}E)=\mu(E)$, for every $E\in\Mcal$ and every $t\geq 0$. In this context, the following continuous version of the Poincar\'e Recurrence Theorem holds (see \cite[Theorem 1.4]{walters1982} for the discrete version, which implies the continuous version when applied to $T=S(t)$, for any fixed time $t>0$):

\begin{thm}[Poincar\'e Recurrence Theorem]
  \label{poincarerecurrence}
  Let $(X,\Mcal,\mu)$ be a probability space and let $\{S(t)\}_{t\geq 0}$ be a measure-preserving semigroup on this probability space.  Then, given $E\subset X$ measurable, it follows that for $\mu$-almost-every $x\in E$, there exists a sequence $t_n\rightarrow \infty$ of nonnegative numbers such that $S(t_n)x\in E$ for all $n\in\NN$.
\end{thm}

\subsection{Time-dependent function spaces}
\label{timedependentfuncionspacessec}

Let $I$ be an arbitrary interval in $\RR$. Consider the spaces $\Ccal_\loc(I,H_\rw)$ and $\Ccal_\loc(I,B_H(R)_\rw)$, with $R>0$, endowed with the topology of uniform weak convergence on compact intervals in $I$. With this topology, the space $\Ccal_\loc(I,H_\rw)$ is a separable Hausdorff locally convex topological vector space and $\Ccal_\loc(I,B_H(R)_\rw)$ is a Polish space (see \cite[Section 2.4]{frtssp1}). In the case $I$ is compact, we may write simply $\Ccal(I,H_\rw)$ and $\Ccal(I,B_H(R)_\rw)$.

The topology of $\Ccal_\loc(I,H_\rw)$ can be characterized by a basis of neighborhoods of the origin given by 
\begin{equation}
  \label{basisneighborhoodclocihw}
   \Ocal(J,O_\rw) = \left\{\bv\in \Ccal_\loc(I,H_\rw); \; \bv(t)\in O_\rw, \;\forall t\in J\right\},
\end{equation}
where $J$ is a compact interval in $I$, and $O_\rw$ is a neighborhood of the origin in $H_\rw$.

For intervals $J\subset I \subset \RR$, we define the restriction operator given by
\begin{equation}
  \begin{aligned}
    \Pi_J :  \Ccal_\loc(I, H_\rw) & \: \rightarrow \;\Ccal_\loc(J, H_\rw) \\
       \bu & \;\mapsto \;(\Pi_J \bu)(t) = \bu(t), \;\forall t\in J.
  \end{aligned}
\end{equation}
These operators are continuous. In case $J$ is closed in $I$, $\Pi_J$ is also surjective and open, in the sense of taking an open set in $\Ccal_\loc(I,H_\rw)$ into an open set in $\Ccal_\loc(J,H_\rw)$. 

For each interval $I$ and each $t_0\in I$, we also define the projection operators
\begin{equation}
  \begin{aligned}
    \Pi_{t_0} : \Ccal_\loc(I, H_\rw) & \;\rightarrow \;H_\rw \\
       \bu & \;\mapsto \;\Pi_{t_0} \bu = \bu(t_0),
  \end{aligned}
\end{equation}
which are also continuous and open, as well as surjective.

For the sake of notational simplicity, we do not make explicit the dependence of $\Pi_J$ and of $\Pi_{t_0}$ on the interval $I$. This should be clear in the context. 

For any interval $I$ which is unbounded on the right and any $\tau>0$, we define the translation operator
\begin{equation}
  \label{sigmatdef}
  \begin{aligned}
    \sigma_\tau :\Ccal_\loc(I, H_\rw) & \;\rightarrow \;\Ccal_\loc(I,H_\rw) \\
       \bu & \;\mapsto \;(\sigma_\tau\bu)(t) = \bu(t+\tau), \qquad \forall t\in I.
  \end{aligned}
\end{equation}
The family $\{\sigma_\tau\}_{\tau\geq 0}$ is a continuous semigroup of linear operators on $\Ccal_\loc(I,H_\rw)$. Moreover, if $I$ is closed, then each $\sigma_\tau$ is an open map.

We also define, again for $I$ unbounded on the right, the map
\begin{equation}
  \label{sigmacontinuous}
  \begin{aligned}
    \sigma :[0,\infty) \times \Ccal_\loc(I, H_\rw) & \;\rightarrow 
           \;\Ccal_\loc(I,H_\rw) \\
         (\tau,\bu) & \;\mapsto
           \;(\sigma(\tau,\bu)=(\sigma_\tau\bu)(t) = \bu(t+\tau), 
              \qquad \forall t\in I.
  \end{aligned}
\end{equation}
Using the basis of neighborhoods of the origin \eqref{basisneighborhoodclocihw} of $\Ccal_\loc(I,H_\rw)$, it is not difficult to see that the map $\sigma$ is continuous. In case $I$ is closed, the map $\sigma$ is also open, but in a stronger sense: since $\sigma(J\times \Ocal) = \bigcup_{t\in J}\sigma_t(\Ocal)$ for arbitrary $J\subset \RR$ and $\Ocal\subset\Ccal_\loc(I,H_\rw)$, and since $\sigma_t$ is an open map in this case, it follows that $\sigma(J\times\Ocal)$ is open if so is $\Ocal$, regardless of $J$.

The map $\sigma$ can be used to write the orbit of a trajectory. For instance, if $\bu\in \Ccal_\loc([0,\infty),H_\rw)$, then the set of values assumed by $\bu$ in $H$ can be written as
\begin{equation}
  \bigcup_{t\geq 0} \{\bu(t)\} = \Pi_0\sigma([0,\infty)\times\{\bu\}).
\end{equation}
This expression is crucial for proving that orbits are (universally) measurable (see Section \ref{sectionmeasorbit}) and that the set of recurrent points are also (universally) measurable (see Section \ref{sectionrecurrence}).

\subsection{Trajectory spaces}
\label{trajectoryspaces}

We recall the spaces of weak solutions studied in \cite{frtssp1}. Consider $R\geq R_0$, where $R_0$ is given by \eqref{defR0}, and let $I$ be an arbitrary interval in $\RR$, with $I^\circ$ denoting its interior. We define
\begin{align}
  \Ucal_I
      & = \{ \bu\in \Ccal_\loc(I,H_\rw); \;\bu \text{ is a weak solution on } I\}, \\
  \Ucal_I(R)
      & = \left\{ \bu\in \Ccal_\loc(I,B_H(R)_\rw); \;\bu \text{ is a weak solution on } I
         \right\}, \\
  \Ucal_I^\sharp
      & = \{ \bu\in \Ccal_\loc(I,H_\rw); \;\bu \text{ is a weak solution on } I^\circ\}, \\
  \Ucal_I^\sharp(R)
      & = \left\{ \bu\in \Ccal_\loc(I,B_H(R)_\rw); \;\bu \text{ is a weak solution on } 
         I^\circ \right\},
\end{align}
with all the spaces being endowed with the topologies inherited from $\Ccal_\loc(I,H_\rw)$.

As discussed in \cite{frtssp1}, the space $\Ucal_I^\sharp$ is the sequential closure of $\Ucal_I$, with $\Ucal_I\subset\Ucal_I^\sharp$, in general, and $\Ucal_I=\Ucal_I^\sharp$, when $I$ is open on the left. In the bounded case, since $\Ucal_I^\sharp(R)$ is metrizable and complete, we have that $\Ucal_I^\sharp(R)$ is in fact the closure of $\Ucal_I(R)$, with $\Ucal_I(R)\subset\Ucal_I^\sharp(R)$, in general, and $\Ucal_I(R)=\Ucal_I^\sharp(R)$ when $I$ is open on the left.  Using \eqref{defR0}, it is proved in \cite[Lemma 2.6]{frtssp1} that the spaces $\Ucal_I(R)$ and $\Ucal_I^\sharp(R)$ are not empty for $R\geq R_0$.

The Leray-Hopf weak solutions belong to $\Ucal_I$, so this is the natural space to consider, but the larger space $\Ucal_I^\sharp$ is needed because each space $\Ucal_I^\sharp(R)$ is compact. 
  
We recall below a few results from \cite{frtssp1}.
\begin{lem}[{\cite[Proposition 2.7]{frtssp1}}]
  \label{Ucalspaceslem}
  Let $I\subset \RR$ be an arbitrary interval and let $R\geq R_0$. Then,
  \begin{itemize}
    \item[(i)] The spaces $\Ucal_I$ and $\Ucal_I^\sharp$ are separable Hausdorff spaces;
    \item[(ii)] The space $\Ucal_I(R)$ is a separable metrizable space;
    \item[(iii)] The space $\Ucal_I^\sharp(R)$ is a Polish space.
  \end{itemize}
  \qed
\end{lem}  
  
\begin{lem}[{\cite[Proposition 2.8]{frtssp1}}]
  \label{ucalitildercompact}
  Let $I\subset\RR$ be an arbitrary interval and let $R\geq R_0$. Then $\Ucal_I^\sharp(R)$ is compact in $\Ccal_\loc(I,B_H(R)_\rw)$ and, hence, it is a compact metric space. It is also compactly embedded in $L_\loc^2(I;H)$.
\end{lem}

\begin{lem}[{\cite[Proposition 2.9]{frtssp1}}]
  \label{lemcharacucaliopen}
  Let $I\subset \RR$ be an interval open on the left. Then, for any sequence $\{R_k\}_{k\in \NN}$ of positive numbers with $R_k\geq R_0$ and $R_k\rightarrow \infty$ and any sequence $\{J_n\}_{n\in \NN}$ of compact intervals in $I$ with $I=\cup_n J_n$, we have the characterization
  \begin{equation}
    \label{characucaliopen}
    \Ucal_I =\Ucal_I^\sharp=\bigcap_n \bigcup_k\Pi_{J_n}^{-1}\Ucal_{J_n}^\sharp(R_k).
  \end{equation}
  In particular, $\Ucal_I$ and $\Ucal_I^\sharp$ are Borel $\Fcal_{\sigma\delta}$ sets in   $\Ccal_\loc(I,H_\rw)$. 
\end{lem}

\begin{lem}[{\cite[Proposition 2.11]{frtssp1}}]
  Let $I\subset\RR$ be an interval closed and bounded on the left. Then for any sequence $\{R_k\}_{k\in \NN}$ with $R_k\geq R_0$ and $R_k\rightarrow \infty$, we have the representation
  \begin{equation}
    \label{characucalitildeclosed}
    \Ucal_I^\sharp = \bigcup_k\Ucal_I^\sharp(R_k).
  \end{equation}
  In particular, $\Ucal_I^\sharp$ is $\sigma$-compact in $\Ccal_\loc(I,H_\rw)$, i.e. it is a countable union of compact sets in $\Ccal_\loc(I,H_\rw)$. Moreover, any bounded subset of $\Ucal_I^\sharp$ must be included in $\Ucal_I^\sharp(R_k)$ for $k$ sufficiently large.
\end{lem}

\begin{lem}[{\cite[Proposition 2.12]{frtssp1}}]
  \label{Ucalileftborel}
  Let $I\subset\RR$ be an interval closed and bounded on the left and let $R\geq R_0$. Then, $\Ucal_I$ and $\Ucal_I(R)$ are Borel sets in $\Ccal_\loc(I,H_\rw)$. Moreover, for any sequence $\{R_k\}_{k\in \NN}$ of positive numbers with $R_k\rightarrow \infty$, we have the characterization
  \begin{equation}
    \label{characucaliclosed}
    \Ucal_I = \bigcup_k\Ucal_I(R_k).
  \end{equation}
  Furthermore, any bounded subset of $\Ucal_I$ must be included in $\Ucal_I(R_k)$, for $k$ sufficiently large.
\end{lem}

A slight variation of Lemma \ref{lemcharacucaliopen} can be given in terms of intervals $J_n$ which are closed on the left but not necessarily compact:
\begin{lem}
  \label{lemcharacucaliopen2}
  Let $I\subset \RR$ be an interval open on the left. Then, for any sequence $\{R_k\}_{k\in \NN}$ of positive numbers with $R_k\geq R_0$ and $R_k\rightarrow \infty$ and any sequence $\{J_n\}_{n\in \NN}$ of intervals in $I$ with $I=\cup_n J_n$, we have the characterization
  \begin{equation}
    \label{characucaliopen2}
    \Ucal_I =\Ucal_I^\sharp=\bigcap_n \bigcup_k\Pi_{J_n}^{-1}\Ucal_{J_n}^\sharp(R_k).
  \end{equation}
\end{lem}

\begin{proof}
  For each $n$, consider a compact subinterval $J_n'\subset J_n$ such that still $I=\cup_n J_n'$. Then, applying \eqref{characucaliopen} with $J_n'$ and using that $\Pi_{J_n'}^{-1}\Ucal_{J_n'}^\sharp(R_k) \subset \Pi_{J_n}^{-1}\Ucal_{J_n}^\sharp(R_k)$ we find that
  \[  \Ucal_I =\Ucal_I^\sharp=\bigcap_n \bigcup_k\Pi_{J_n'}^{-1}\Ucal_{J_n'}^\sharp(R_k) \subset \bigcap_n \bigcup_k\Pi_{J_n}^{-1}\Ucal_{J_n}^\sharp(R_k) \subset \Ucal_I,
  \]
  where in the last inclusion we use that $I=\cup_n J_n'$. This completes the proof.\end{proof}

Another fundamental space to consider in the particular case of stationary statistical solutions is that of uniformly bounded global weak solutions defined on $\RR$. More precisely, we define
\begin{equation}
  \label{defwcal}
  \Wcal = \left\{ \bu\in \Ccal_\loc(\RR,H_\rw); \;\bu \text{ is a weak solution on } \RR \text{ with }
      \sup_{t\in\RR} |\bu(t)|_{L^2}<\infty\right\},
\end{equation}
endowed with the topology inherited from $\Ccal_\loc(\RR,H_\rw)$.

Taking into account that translations in time of global weak solutions are also global weak solutions, we have that
\begin{equation}
  \label{Wcalinvariant}
  \sigma_\tau \Wcal = \Wcal, \qquad \forall \tau\geq 0,
\end{equation}
and
\begin{equation}
  \label{PitWcal}
  \Pi_t\Wcal \text{ is independent of $t\in\RR$.}
\end{equation}

The following result says that $\Wcal$ is compact in $\Ccal_\loc(\RR,H_\rw)$.

\begin{prop}
  \label{propwcalinr0}
  The space $\Wcal$ is a compact metrizable space and is included in the space $\Ccal_\loc(\RR,B_H(R_0)_\rw)$, where $R_0$ is given in \eqref{defR0}. In particular, we have 
  \begin{equation}
    \label{eqwurbr0}
    \Wcal = \Ucal_\RR \cap \Ccal_\loc(\RR,B_H(R_0)_\rw) = \Ucal_\RR(R_0).
  \end{equation}
\end{prop}

\begin{proof} 
  From the uniform boundedness in $H$ of an element $\bu$ in $\Wcal$ and the a~priori estimate \eqref{energyestimate} it follows that 
  \begin{equation}
    \label{energyestimateinwcal}
    |\bu(t)|_{L^2}\leq \frac{1}{\nu\lambda_1}|\bbf|_{L^2}, \qquad \forall t\in \RR,
         \; \forall\bu\in \Wcal.
  \end{equation}
  Thus $\Wcal$ is a subset of $\Ccal_\loc(\RR,B_H(R_0)_\rw)$ and we can write \eqref{eqwurbr0}. Then, since the space $\Ccal_\loc(\RR,B_H(R_0)_\rw)$ is metrizable, so is $\Wcal$. From the a~priori estimate \eqref{L2enstrophyestimatefinh}, we also have
  \begin{equation}
    \label{L2enstrophyestimateinwcal}
    \nu\int_{t_0}^{t_1} \|\bu(t)\|_{H^1}^2\;\rd t
       \leq \frac{1}{\nu^2\lambda_1^2}|\bbf|_{L^2}^2
          + \frac{1}{\nu\lambda_1}|\bbf|_{L^2}^2(t_1-t_0),
       \;\;\forall t_0,t_1\in \RR, \;t_0<t_1, \;\forall \bu\in\Wcal.
  \end{equation}
  Then, it follows from the estimates \eqref{energyestimateinwcal} and \eqref{L2enstrophyestimateinwcal} that $\Wcal$ is equi-bounded and equicontinuous in $\Ccal_\loc(\RR,B_H(R)_\rw)$, with $\{\bu(t)\}_{\bu\in\Wcal}\subset B_H(R)_\rw$ relatively compact for each $t\in \RR$. From the Arzela-Ascoli Theorem, it follows that $\Wcal$ is relatively compact in $\Ccal_\loc(\RR,B_H(R)_\rw)$. From Lemma \ref{convergenceofsolutions1}, $\Wcal$ is also closed. Therefore, $\Wcal$ is a compact metric space.
\end{proof}
\medskip

From \eqref{L2enstrophyestimatefinh} we also obtain the following uniform time-average bound for functions in $\Wcal$.
\begin{lem}
  \label{lemwcalh1ave}
  For every $\bu\in \Wcal$ it follows that
  \begin{equation}
    \label{ineqlemwcalh1ave}
    \frac{1}{(t-t')} \int_{t'}^t \|\bu(s)\|_{H^1}^2 \;\rd s 
       \leq \lambda_1^{1/2}\nu^2 G^2\left(1 + \frac{1}{\nu\lambda_1(t-t')}\right), 
  \end{equation}
  for almost every $t'\in \RR$ and for all $t>t'$.
\end{lem}

\begin{proof}
  Just divide \eqref{L2enstrophyestimatefinh} by $\nu (t-t')$ and use the bound in Proposition \ref{propwcalinr0}, along with Young's inequality and the definition \eqref{defgrashof} of the Grashof number.
\end{proof}

The following uniform time-average bound also holds in $\Wcal$.
\begin{lem}
  \label{lemwcaldaave}
  For every $\bu\in \Wcal$ it follows that
  \begin{multline}
    \frac{1}{(t-t')} \int_{t'}^t |A\bu(s)|_{L^2}^{2/3} \;\rd s \\
       \leq \frac{1}{3\nu^{1/3}\lambda_1^{1/2}G^{2/3}(t-t')} + c_3\lambda_1^{1/2}\nu^{2/3} G^2\left(1 + \frac{1}{2\nu\lambda_1(t-t')}\right),
  \end{multline}
  for almost every $t'\in \RR$ and for all $t>t'$.
\end{lem}

\begin{proof}
  Simply divide the estimate \eqref{estDAtwothirdsglobal} by $t-t'$ and use the inequality \eqref{ineqlemwcalh1ave}, along with Young's inequality and the definition \eqref{defgrashof} of the Grashof number.
\end{proof}

It is possible to derive an estimate for the time-average of the $L^\infty$ norm of the solutions in $\Wcal$, using the estimates in Lemmas \ref{lemwcalh1ave} and \ref{lemwcaldaave}, along with Agmon's inequality \eqref{agmonineq}. This is done explicitly in Theorem \ref{corboundsrho0}, in the context of stationary statistical solutions.

We also consider the subspaces of weak solutions which can be extended in time to a uniformly bounded global weak solution (as we will see shortly, the time-invariant Vishik-Fursikov measures are carried by these subspaces):
\begin{align}
  \label{wcalidef}
  & \Wcal_I = \left\{ \bu\in \Ucal_I; \;\exists \bw\in \Wcal,
         \;\bu(t)=\bw(t), \;\forall t\in I \right\}, \\
  & \Wcal_I^\sharp = \left\{ \bu\in \Ucal_I^\sharp; \;\exists \bw\in \Wcal,
         \;\bu(t)=\bw(t), \;\forall t\in I \right\}.
\end{align}
We have $\Wcal_I^\sharp=\Pi_I\Wcal$ and, clearly,
\begin{equation}
  \label{wcalsharpcompact}
  \Wcal_I^\sharp \emph{ is a compact subspace of } \Ucal_I^\sharp(R_0).
\end{equation}
Note also that $\Wcal_I\subset\Wcal_I^\sharp$. If $I$ is open on the left, then $\Wcal_I = \Wcal_I^\sharp$. In particular, $\Wcal_\RR=\Wcal^\sharp_\RR = \Wcal$. Moreover, $\Wcal_I=\Wcal_I^\sharp\cap\Ucal_I(R_0)$, which implies that $\Wcal_I$ is Borel in $\Ccal_\loc(I;B_H(R_0)_\rw)$.

A useful characterization of the set $\Wcal_I^\sharp$ can be given in terms of the spaces $\Ucal_I^\sharp(R)$, with the help of Lemma \ref{convergenceofsolutions1}.
\begin{lem}
  \label{lemwcalasintersectionofucals}
  Let $I\subset \RR$ be an interval unbounded on the right. Then, for any $R>R_0$ and any sequence $\{\tau_k\}_{k\in \NN}$ of nonnegative times with $\tau_k\rightarrow \infty$, as $k\rightarrow \infty$, we have the characterization
\begin{equation}
  \label{wcalsharpasintersectionofucals}
  \Wcal^\sharp_I =  \bigcap_{k\in \NN}\sigma_{\tau_k}\Ucal^\sharp_I(R).
\end{equation}
\end{lem}

\begin{proof}
  If $\bu\in \Wcal^\sharp_I$, then there exists $\bw\in \Wcal$ such that $\bu(t) =\bw(t)$, $\forall t\in I$. Using the invariance \eqref{Wcalinvariant}, we have $\bu(t) = \bw(t-\tau_k+\tau_k) = \sigma_{\tau_k}\bw(t-\tau_k)$, with $t\mapsto \bw(t-\tau_k)$ belonging to $\Wcal$. In particular, the restriction of $t\mapsto \bw(t-\tau_k)$ to the interval $I$ belongs to $\Wcal^\sharp_I\subset \Ucal^\sharp_I(R_0)\subset \Ucal^\sharp_I(R)$. Thus, $\bu\in \sigma_{\tau_k}\Ucal^\sharp_I(R)$, for any $k\in \NN$, proving the inclusion $\Wcal^\sharp_I \subset  \bigcap_{k\in \NN}\sigma_{\tau_k}\Ucal^\sharp_I(R)$.

  Now suppose that $\bu\in \bigcap_{k\in \NN}\sigma_{\tau_k}\Ucal^\sharp_I(R)$. Thus, $\bu=\sigma_{\tau_k}\bu_k$, for some $\bu_k\in \Ucal^\sharp_I(R)$, for any $k\in \NN$. Define $\bv_k(t)=\bu_k(t+\tau_k)$, for $t\in I -\tau_k$. Since $I$ is unbounded on the right and $\tau_k \rightarrow \infty$, we find that $\bigcup_k (I-\tau_k) = \RR$. Moreover, each $\bv_k$ is a weak solution on $I-\tau_k$ which is uniformly bounded in $H$ by $R$ (with the bound $R$ independent of $k$). Thus, using Lemma \ref{convergenceofsolutions1} together with a diagonal argument, we find that there is a subsequence of $\{\bv_k\}$ that converges, in the weak topology of $H$, to a weak solution $\bw$ on $\RR$, uniformly on any compact interval in $I$. Clearly, $\bw$ is bounded by $R$ in $H$, so that $\bw\in \Wcal$. Morever, $\bu(t) = \sigma_{\tau_k}\bu_k(t) = \bu_k(t+\tau_k) = \bv_k(t)$, for any $t\in I$. Hence, at the limit, $\bw(t)=\bu(t)$, for any $t\in I$, proving that $\bu \in \Wcal^\sharp_I$. This concludes the proof of \eqref{wcalsharpasintersectionofucals}.
 \end{proof}

\subsection{The weak global attractor and its regular parts}
\label{weakglobalattractor}

At this point we recall the weak global attractor $\Acal_\rw$ introduced in \cite{foiastemam85}. It is defined as the set of all points in $H$ which belong to a global weak solution defined on $\RR$ and uniformly bounded in $H$. By definition this set is directly related to the set $\Wcal$ and can be written as
\begin{equation}
  \label{awdef}
 \Acal_\rw = \left\{\bu_0\in H; \;
    \;\exists\bu\in \Wcal, \;\bu(0)=\bu_0 \right\} = \Pi_0\Wcal.
\end{equation}
Taking \eqref{PitWcal} into consideration, we have
\begin{equation}
  \label{awpitw}
  \Acal_\rw = \Pi_t\Wcal, \quad \forall t\in \RR.
\end{equation}

Since $\Wcal$ is compact (Proposition \ref{propwcalinr0}) and the projection operators are continuous it is clear that $\Acal_\rw=\Pi_0\Wcal$ is compact in $H_\rw$. It is in fact included and compact in $B_H(R_0)_\rw$.

We consider now certain regular parts of the weak global attractor. An important subset of $\Acal_\rw$ of regular solutions is the set $\Acal_\reg$ defined by
\begin{equation}
  \label{defaregorg}
  \Acal_\reg = \left\{ \bu_0\in V; \;
     \parbox{3.8in}{$\exists\delta>0$ and $\exists \bu\in\Wcal$ with
         $\bu(0)=\bu_0$, such that $\bu$ is regular
         on $(-\delta,\delta)$ and any other global weak solution
         $\bv\in \Wcal$ with $\bv(0)=\bu_0$ is such that
         $\bv(t) = \bu(t)$, $\forall t\in (-\delta,\delta)$}
   \right\}.
\end{equation}
Loosely speaking, the condition in \eqref{defaregorg} says, essentially, that $\bu\in \Wcal$ is regular on $(-\delta,\delta)$ and is unique, over the interval $(-\delta,\delta)$, among all the global weak solutions in $\Wcal$ with value $\bu_0$ at time $t=0$. This set was introduced in \cite{foiastemam85} and it was proved there to be open and dense in $\Acal_\rw$ in the weak topology of $H$. This proof was given in more details in \cite{frt2010c}, where this set was also showed to be characterized as
\begin{equation}
    \label{defareg}
    \Acal_\reg = \left\{\bu_0\in V; \;
      \parbox{3.8in}{$\forall\bu\in\Wcal$ with $\bu(0)=\bu_0$, $\exists\delta_\bu>0$
     such that $\bu$ is regular on a neighborhood $(-\delta_\bu,\delta_\bu)$ of $t=0$}
      \right\}.
\end{equation}
Another regular subset considered in \cite{frt2010c} is the subset $\Acal_\reg'$ defined as
\begin{equation}
  \label{defaregprime}
  \Acal_\reg' = \left\{\bu_0\in V; \;
     \parbox{3.8in}{$\exists\bu\in\Wcal$ with $\bu(0)=\bu_0$ and $\exists\delta_\bu>0$
     such that $\bu$ is regular on a neighborhood $(-\delta_\bu,\delta_\bu)$ of $t=0$}
   \right\}.
\end{equation}
It is clear that
\[ \Acal_\reg \subset \Acal_\reg'\subset \Acal_\rw,
\]
so that $\Acal_\reg'$ is also dense in $\Acal_\rw$. It was also proved in \cite{frt2010c} that $\Acal_\reg'$ is a $\sigma$-compact set in $H_\rw$. The following two characterizations were given in \cite{frt2010c}:
  \begin{align*}
    \Acal_\reg & = \left\{ \bu_0\in H; \;\forall \bu\in\Wcal \text{ such that } \bu(0)=\bu_0,
           \text{ we have } \liminf_{t\rightarrow 0^-}\|\bu(t)\|_{H^1}<\infty \right\}; \\
    \Acal_\reg' & = \left\{ \bu_0\in H; \;\exists \bu\in\Wcal \text{ such that } \bu(0)=\bu_0
         \text{ and } \liminf_{t\rightarrow 0^-}\|\bu(t)\|_{H^1}<\infty \right\}.
  \end{align*}
Notice that these last characterizations are related to how the solutions blow-up, or not, in the $H^1$ norm, as $t$ goes to zero from negative values. The rate of blow up can be further estimated according to the following characterizations given in \cite[Section 3.2, Corollary 1]{frt2010c}. 
\begin{multline}
    \label{charactaregminusaregeq}
    \Acal_\rw \setminus \Acal_\reg \\
      = \left\{ \bu_0\in H; \;\exists \bu\in \Wcal\cap \Pi_0^{-1}\{\bu_0\} \text{ such that } 
              \|\bu(t)\|_{H^1}^2 \geq \Gamma(t), \;\forall t< 0 \right\},
\end{multline}
\begin{multline}
    \label{charactaregminusaregprimeeq}
     \Acal_\rw \setminus \Acal_\reg' \\
       = \left\{ \bu_0\in H; \;\forall \bu\in \Wcal\cap\Pi_0^{-1}\{\bu_0\},
            \text{ we have } \|\bu(t)\|_{H^1}^2 \geq \Gamma(t), \;\forall t< 0 \right\},
\end{multline}
where 
\begin{equation}
    \label{blowupestimate}
    \Gamma(t)= \frac{\nu^{3/2}}{2c_4|t|^{1/2}} - \nu^{2/3}|\bbf|_{L^2}^{2/3}, \qquad t<0,
\end{equation}
with $c_4=\max\{1,c_2^{3/2}\}$.
  
We can be more precise in the characterization of the set $\Acal_\reg'\setminus\Acal_\reg$
and prove that the solution which is regular and passes through a given point in this set is 
in fact unique.

\begin{prop}
  \label{acalregprimeuniquestrong}
  Let $\bu_0\in \Acal_\reg'$. Then, there exists a unique solution
  $\bu\in \Wcal$ with $\bu(0)=\bu_0$ which is a strong solution on a open interval
  containing $t=0$. Moreover, if $\bu_0\in\Acal_\reg'\setminus\Acal_\reg$, then
  $\Pi_t^{-1}\bu_0\setminus\{\bu\}$ is not empty 
  and any solution $\bv\in \Wcal$ with $\bv(0)=\bu_0$ and different than $\bu$ 
  blows up at $t=0$, i.e. $\liminf_{t\rightarrow 0^-}\|\bv(t)\|_{H^1}=\infty$, or, 
  more precisely, $\|\bv(t)\|\geq \Gamma(t)$, for all $t<0$, where $\Gamma(t)$
  is given by \eqref{blowupestimate}.
\end{prop}  

\begin{proof}
  Since strong solutions are analytic in time, two strong solutions on a neighborhood of   $t=0$ with the same value at $\bu_0$ at this time must coincide in their common neighborhood, hence this strong solution is unique. If $\bu_0$ is in $\Acal_\reg'\setminus \Acal_\reg$, then it means that there are other solutions passing through $\bu_0$ and they must all blow up, otherwise they would be strong and would coincide with $\bu$ in a neighborhood of $t=0$.
\end{proof}

In view of Proposition \ref{acalregprimeuniquestrong}, we can actually write
\begin{equation}
  \label{acalregprimecharactunique}
    \Acal_\reg' = \left\{ \bu_0\in H; \;\exists! \bu\in\Wcal \text{ such that } \bu(0)=\bu_0
         \text{ and } \liminf_{t\rightarrow 0^-}\|\bu(t)\|_{H^1}<\infty \right\}.
\end{equation}

This result can be connected with the following result, which refers to a solution which belongs to $\Acal_\reg$ up to a certain time.
\begin{prop}
  Let $\bu\in\Wcal$ with $\bu(0)\in\Acal_\reg$. Let 
  \[ t_\reg^{\text{\rm max}}=\sup\{\tau>0, \;\bu(t)\in \Acal_\reg, \;\forall t\in [0,\tau)\}.
  \]
  Suppose $t_\reg^{\text{\rm max}}<\infty$. Then, 
  $\bu(t_\reg^{\text{\rm max}})\notin\Acal_\reg$. 
  Morever, if 
  \[ \liminf_{t\nearrow t_\reg^{\text{\rm max}}} \|\bu(t)\|_{H^1}<\infty, 
  \]
  then $\bu(t_\reg^{\text{\rm max}})\in \Acal_\reg'\setminus \Acal_\reg$ and 
  $\sigma_{t_\reg^{\text{\rm max}}}\bu$ is the 
  unique solution referred to in Proposition \ref{acalregprimeuniquestrong} for 
  $\bu_0=\bu(t_\reg^{\text{\rm max}})$.
\end{prop}

\begin{proof}
  The fact that $\bu(t_\reg^{\text{\rm max}})\notin\Acal_\reg$ follows immediately from the fact that $\Acal_\reg$ is weakly open in $\Acal_\rw$. Then, if $\bu$ does not blow up at $t_\text{\rm max}$ in the norm of $V$, then $\bu$ can be continued as a strong solution up to an open interval containing $t_\text{\rm max}$, so that $\bu(t_\reg^{\text{\rm max}})\in\Acal_\reg'$, and hence $\bu(t_\reg^{\text{\rm max}})\in\Acal_\reg'\setminus\Acal_\reg$. Moreover, the translation $\sigma_{t_\reg^{\text{\rm max}}}\bu$ has to be the unique strong solution with $\bu_0=\bu(t_\reg^{\text{\rm max}})$ referred to in Proposition \ref{acalregprimeuniquestrong}.
\end{proof}

Similar regular parts can be defined within the set $\Wcal$. In fact, the set $\Acal_\reg'$ is directly connected to the set
\begin{equation}
  \Wcal_\reg'=\left\{\bu\in \Wcal; \;0 \text{ is a point of interior regularity for } \bu\right\}.
\end{equation}
The connection is given by
\begin{equation}
  \label{wcalregprimeandacalregprime}
  \Wcal_\reg' = \Wcal \cap \Pi_0^{-1}\Acal_\reg'.
\end{equation}
Similarly to \eqref{charactaregminusaregprimeeq}, we have the characterization
\begin{equation}
  \label{charactwregprime}
    \Wcal \setminus \Wcal_\reg'
       = \left\{ \bu\in \Wcal; \; \|\bu(t)\|_{H^1}^2 \geq \Gamma(t), \;\forall t< 0 \right\}.
\end{equation}
A more precise set, with a lower bound on the size of the interval in which the solution is regular, can be defined as follows
\begin{multline}
  \label{defwcalregtauprime}
  \Wcal_{\reg,\tau}'=\Wcal \cap \Pi_{(-\tau,\tau)}^{-1}\Ccal((-\tau,\tau),V) \\
    = \left\{\bu\in \Wcal; \;\bu \text{ is a strong solution on } (-\tau,\tau)\right\},
\end{multline}
for $\tau>0$. By allowing $\tau=\infty$, we obtain the set of global regular solutions $\Wcal_{\reg,\infty}'$. Clearly, $\Wcal_{\reg,\tau_1}'\subset\Wcal_{\reg,\tau_2}'$, for $0<\tau_1 \leq \tau_2$, with 
\begin{equation}
  \label{wcalregprimeasunionwcalregshortprime}
  \Wcal_\reg'=\bigcup_{\tau>0} \Wcal_{\reg,\tau}' = \bigcup_{n\in\NN} \Wcal_{\reg,1/n}', 
\end{equation}
and
\begin{equation}
  \Wcal_{\reg,\infty}' = \bigcap_{\tau>0} \Wcal_{\reg,\tau}' = \bigcap_{n\in\NN} \Wcal_{\reg,n}'
\end{equation}
Moreover, each $\Wcal_{\reg, \tau}'$ can be written as
\[ \Wcal_{\reg,\tau}' = \bigcup_{k\in \NN} \Wcal_{\reg,\tau}'(kR_0),
\]
where
\begin{multline*}
  \Wcal_{\reg,\tau}'(R) = \Wcal \cap \Pi_{(-\tau,\tau)}^{-1}\Ccal((-\tau,\tau); B_V(R)) \\ = \left\{\bu\in \Wcal; \;\bu \text{ is a strong solution on } (-\tau,\tau) \text{ with } \|\bu(t)\|_{H^1} \leq R\right\}, 
\end{multline*}
for $R>0$, $\tau>0$. Since $\Ccal((-\tau,\tau); B_V(R))$ is closed in $\Ccal_\loc(\RR,H_\rw)$ and $\Wcal$ is compact in $\Ccal_\loc(\RR,H_\rw)$ it follows that $\Wcal_{\reg,\tau}'(R)$ is compact, for any $R>0$ and any $\tau>0$. Thus, $\Wcal_{\reg,\tau}'$ and $\Wcal_\reg'$ are $\sigma$-compact and $\Wcal_{\reg,\infty}'$ is an $\Fcal_{\sigma\tau}$-set in $\Ccal_\loc(\RR,H_\rw)$. In particular, all such sets are Borel.

Using \eqref{charactwregprime} (translating the estimate from the blow up at $0$ to a blow up at $\beta$) we see that the complement of the set $\Wcal_{\reg,\tau}'$ can be characterized by
\begin{multline}
  \label{charactwregtauprimecomplement}
  \Wcal\setminus \Wcal_{\reg,\tau}' = \left\{\bu\in \Wcal; \;\exists \beta \in (-\tau,\tau); \;\liminf_{t\rightarrow \beta^-}\|\bu(t)\|_{H^1} = \infty\right\} \\
  = \left\{\bu\in \Wcal; \;\exists \beta \in (-\tau,\tau); \;\|\bu(t)\|_{H^1}^2 \geq \Gamma(t-\beta), \;\forall t< \beta\right\}.
\end{multline}

\subsection{The multivalued evolution map}
\label{secmultvaluedmap}

Due to the lack of a result on the uniqueness of weak solutions, we cannot define a solution semigroup in the classical sense, associating a unique solution to a given initial condition. The possibility that certain initial conditions give rise to more than one solution is not ruled out. Thus, for each time $t \geq 0$ and each initial condition $\bu_0$, we may have more than one state $\bu(t)$ corresponding to the value at $t$ of different weak solutions $\bu$ starting at that initial condition, $\bu(0)=\bu_0$. This leads naturally to the definition of an evolution map acting in the collection of all subsets of the phase space. More precisely, we have the following definition.
\begin{defs}
  \label{defSigma}
  Given a set $E$ in $H$ and $t\geq 0$, we denote by $\Sigma_t E$ the set of all points $\bw\in H$, such that $\bw=\bu(t)$, for some $\bu$ in $\Ucal_{[0,\infty)}$ with initial condition $\bu(0)\in  E$.
\end{defs}

The following result concerns the composition of the evolution maps.
\begin{lem}
  \label{Sigmaproperty}
  For $E\subset H$ and $t,s\geq 0$, we have that $\Sigma_t\Sigma_s E  \subset \Sigma_{t+s} E.$
  \qed
\end{lem}
\begin{proof}
  The inclusion follows from Lemma \ref{pastinglemma}. In fact, if $\bw\in \Sigma_t\Sigma_s E$, then $\bw = \bu^{(2)}(t)$ for some weak solution $\bu^{(2)}$ with $\bu^{(2)}(0)\in \Sigma_s E$, hence $\bu^{(2)}(0)=\bu^{(1)}(s)$ for some weak solution $\bu^{(1)}$ with $\bu^{(1)}(0)\in E$. Then, by Lemma \ref{pastinglemma}, the concatenation $\bu$ of $\bu^{(1)}$ and $\bu^{(2)}$, as in the lemma, is such that $\bu$ is a weak solution with $\bu(0)\in E$ and $\bu(t+s)=\bw$, so that $\bw\in \Sigma_{t+s}E$. 
\end{proof}
\medskip

\begin{rmk}
If we could prove that any time $s\geq 0$ is a point of strong continuity from the right for any given weak solution $\bu$ with $\bu(0)\in E \subset H$, then we would have that $\bu$ would be the concatenation of two weak solutions as in Lemma \ref{pastinglemma} and, then, $\bu(t+s)$ would be in both $\Sigma_{t+s}E$ and $\Sigma_t\Sigma_s E$. Thus, these two sets would be equal. But the current state of knowledge only gives us that a weak solution is strongly continuous from the right almost everywhere, not everywhere. Hence, we can only assure one side of the inclusion, which is the content of Lemma \ref{Sigmaproperty}.
\end{rmk}

A fundamental expression for $\Sigma_t$ appears in relation with the operators $\Pi_t$ introduced in Section \ref{trajectoryspaces}, as follows.
\begin{lem}
\label{charactsigma}
For any $E\subset H$ and $t\geq 0$, 
\begin{equation}
  \label{sigmaandpi}
  \Sigma_t E = \Pi_t(\Ucal_{[0,\infty)}\cap\Pi_0^{-1}(E)),
\end{equation}
with $\Pi_0$ and $\Pi_t$ considered as defined on $\Ccal_\loc([0,\infty),H_\rw)$. Moreover, for any sequence $\{R_k\}_{k\in\NN}$ of positive real numbers with $R_k\geq R_0$ and $R_k\rightarrow \infty$, we have
\begin{equation}
  \label{sigmaandpirk}
  \Sigma_t E = \bigcup_{k\in\NN} \Pi_t(\Ucal_{[0,\infty)}(R_k)\cap\Pi_0^{-1}(E)).
\end{equation}
\end{lem}

\begin{proof}
Relation \eqref{sigmaandpi} is simply a symbolic way of expressing the definition of $\Sigma_t$ as given in Definition \ref{defSigma}. As for \eqref{sigmaandpirk}, it follows immediately from \eqref{sigmaandpi} and the representation \eqref{characucaliclosed} with $I=[0,\infty)$.
\end{proof}

The following representation is also useful.
\begin{lem}
\label{charactsigmatwice}
For any $E\subset H$ and for $t, s\geq 0$, 
\begin{equation}
  \label{sigmatwiceandpi}
  \Sigma_t\Sigma_s E = 
       \Pi_{t+s}\left(\Ucal_{[0,\infty)}\cap(\Pi_{[s,\infty)}^{-1}\Ucal_{[s,\infty)})\cap\Pi_0^{-1}E\right),
\end{equation}
with the projection and restriction operators $\Pi_0$ and $\Pi_{[s,\infty)}$ considered as defined on $\Ccal_\loc([0,\infty),H_\rw)$. Moreover, for any sequence $\{R_k\}_{k\in\NN}$ of positive real numbers with $R_k\geq R_0$ and $R_k\rightarrow \infty$, we have
\begin{equation}
  \label{sigmatwiceandpirk}
  \Sigma_t\Sigma_s E =\bigcup_k \Pi_{t+s}
        \left(\Ucal_{[0,\infty)}(R_k)\cap(\Pi_{[s,\infty)}^{-1}\Ucal_{[s,\infty)}(R_k))\cap\Pi_0^{-1}E\right).
\end{equation}
\end{lem}

\begin{proof}
  The expression in the right hand side of \eqref{sigmatwiceandpi} gives $\Sigma_t\Sigma_s E$ as the collection of elements of the form $\bu(t+s)$ where $\bu$ has the following two properties: $\bu\in \Ucal_{[0,\infty)}\cap \Pi_0^{-1}E$ (hence it is a Leray-Hopf weak solution on $[0,\infty)$ starting with $\bu(0)\in E$) and $\bu\in \Pi_{[s,\infty)}^{-1}\Ucal_{[s,\infty)}$ (which means that $\bu$ restricted to the interval $[s,\infty)$ is a Leray-Hopf weak solution in that interval). The latter property means that the translation $\bu(\cdot + s)$ is a Leray-Hopf weak solution on $[0,\infty)$ with initial condition $\bu(s)$. Since $\bu(s)\in \Sigma_s E$, we deduce that $\bu(t+s)$ is exactly an element of $\Sigma_t\Sigma_s E$. This gives the characterization of $\Sigma_t\Sigma_s E$ as the right hand side of \eqref{sigmatwiceandpi}.
  
  Concerning \eqref{sigmatwiceandpirk}, this representation follows immediately from \eqref{sigmatwiceandpi} and the corresponding representations for $\Ucal_{[0,\infty)}$ and $\Ucal_{[s,\infty)}$ according to \eqref{characucaliclosed}. We use the fact that the spaces $\Ucal_{[0,\infty)}(R_k)$ and $\Ucal_{[s,\infty)}(R_k)$ are increasing in $k$, so that we can use a single index in the union given in \eqref{sigmatwiceandpirk}.
\end{proof}

\subsection{Orbits and limit sets}
\label{measureorbits}

In this section, we recall two important types of dynamical sets which will be of interest in the sequel, namely the orbits and the $\omega$-limit sets. In particular, we give a more symbolic representation of the orbits, which will be useful for proving the universal measurability of such sets.

In the context of weak solutions and the multivalued map defined earlier, given a set $E$ in $H$, we define the positive orbit starting at $E$ by
\begin{equation}
  \gamma(E) = \bigcup_{t\geq 0} \Sigma_t E = \{ \bu(t); \;\bu \in \Ucal_{[0,\infty)}, \;\bu(0)\in E, \;t\geq 0\}.
\end{equation}
Using the operators $\Pi_0$ and $\sigma$ defined in Section \ref{timedependentfuncionspacessec}, this orbit can be written as
\begin{equation}
  \gamma(E) = \Pi_0\sigma([0,\infty)\times (\Ucal_{[0,\infty)}\cap\Pi_0^{-1} E)),
\end{equation}
where $\Pi_0$ is considered as restricted to the trajectory space 
$\Ccal_\loc([0,\infty),H_\rw)$, and $\sigma$ is considered from 
$[0,\infty)\times \Ccal_\loc[0,\infty),H_\rw)$ into $\Ccal_\loc([0,\infty),H_\rw)$. 
Given a positive sequence $\{R_k\}_{k\in\NN}$ with $R_k\geq R_0$ and $R_k\rightarrow\infty$, it follows from the characterization \eqref{characucaliclosed} that 
\begin{equation}
  \label{orbitrepresentedrk}
  \gamma(E) = \bigcup_{k\in\NN}\Pi_0\sigma([0,\infty)\times (\Ucal_{[0,\infty)}(R_k)\cap\Pi_0^{-1} E)).
\end{equation}

Now we consider $\omega$-limit sets. In fact, we consider two such sets, one in trajectory space and the other in phase space. Given $\bu\in \Ccal_\loc([t_0,\infty),H_\rw)$ and $t_0\in \RR$, we define the $\omega$-limit set in the trajectory space as
\begin{equation} 
  \label{omegalimittrajectory}
  \omega_{\Ccal_\loc([t_0,\infty),H_\rw)}(\bu)  = 
    \left\{\bw\in \Ccal_\loc([t_0,\infty),H_\rw); \;\parbox{2.3in}{there exists
        $\{t_j\}_{j\in\NN}$, $t_j\geq t_0$, $t_j\rightarrow \infty$ such that
        $\sigma_{t_j}\bu \rightarrow \bw$ in $\Ccal_\loc([t_0,\infty),H_\rw)$} \right\}.
\end{equation}
In the phase space, the $\omega$-limit set is given by
\begin{equation}
  \label{omegalimitphase}
  \omega_{H_\rw}(\bu)
     = \left\{ \bw\in H; \;\parbox{2.6in}{there exists $\{t_j\}_{j\in\NN}$, $t_j\geq t_0$,
         $t_j \rightarrow \infty$ such that $\bu(t_j)\rightarrow \bw$ in $H_\rw$}\right\}.
\end{equation}

Since the topology in the definition of these limit sets is the weak topology of $H$, we may sometimes refer to them as weak $\omega$-limit sets.

\subsection{Measurability of orbits and related dynamical sets}
\label{sectionmeasorbit}

Given a Borel set $E$ in $H$, we show that $\Sigma_t E$, the orbit $\gamma(E)$ and other related sets are universally measurable. We also prove that the weak $\omega$-limit sets are compact. We start with the following result.
\begin{lem}
  \label{orbituniversal}
  Let $E$ be a Borel set in $H_\rw$. Then $\gamma(E)$ is a universally measurable set in $H_\rw$.
\end{lem}

\begin{proof} 
  In view of the characterization \eqref{orbitrepresentedrk} of the orbit $\gamma(E)$ and the fact that the collection of universally measurable sets is a $\sigma$-algebra (see Fact \eqref{universallymeasurablesigmaalgebra} at the end of Section \ref{measuretheory}), it suffices to show that
  \[ \Pi_0\sigma([0,\infty)\times (\Ucal_{[0,\infty)}(R)\cap\Pi_0^{-1} E))
  \]
is universally measurable, for any given $R>0$.
 
  Since $E$ is Borel and $\Pi_0$ is continuous, the set $\Pi_0^{-1}E$ is a Borel subset of the space $\Ccal_\loc([0,\infty),H_\rw)$. From Lemma \ref{Ucalileftborel}, the set $\Ucal_{[0,\infty)}(R)$ is also Borel,  hence $\Ucal_{[0,\infty)}(R)\cap \Pi_0^{-1} E$ is a Borel subset of $\Ccal_\loc([0,\infty),B_H(R)_\rw)$.   Thus, $[0,\infty)\times (\Ucal_{[0,\infty)}(R)\cap\Pi_0^{-1} E)$ is a Borel subset of $[0,\infty)\times \Ccal_\loc([0,\infty),B_H(R)_\rw)$. 

  Since $\sigma$ is continuous from the Polish space $[0,\infty)\times \Ccal_\loc([0,\infty),B_H(R)_\rw)$ into the Polish space $\Ccal_\loc([0,\infty),B_H(R)_\rw)$ it follows that $\sigma$ takes Borel sets in the former space into analytic sets in the latter space. In particular, $\sigma([0,\infty)\times (\Ucal_{[0,\infty)}(R)\cap\Pi_0^{-1} E)$ is analytic in $\Ccal_\loc([0,\infty),B_H(R)_\rw)$. Then, since $\Pi_0$ is also continuous from the Polish space $\Ccal_\loc([0,\infty),B_H(R)_\rw)$ into the Polish space $B_H(R)_\rw$ it follows that $\Pi_0\sigma([0,\infty)\times (\Ucal_{[0,\infty)}(R)\cap\Pi_0^{-1}E))$ is analytic in $B_H(R)_\rw$, hence universally measurable in $H_\rw$. This completes the proof.
\end{proof}
\medskip

\begin{rmk}
  Notice that in the two-dimensional case, due to the uniqueness of global solutions, we have $\Ucal_I^\sharp=\Ucal_I$ regardless of the interval $I\subset\RR$. Moreover, due to the backward uniqueness property, the map $\sigma$ restricted to $[0,\infty)\times\Ucal_{[0,\infty)}$ is one-to-one, as well as open. And due to the well-posedness of the two-dimensional problem, the map $\Pi_0$ is a homeomorphism from $\Ucal_{[0,\infty)}$ onto $H_\rw$. Thus, both $\sigma$ and $\Pi_0$ take Borel sets into Borel sets. Therefore, the orbit $\gamma(E)$ is in fact Borel.
\end{rmk}

Next, we consider the measurability of $\Sigma_t E$, for a Borel set $E$. The following result says that $\Sigma_t E$ is universally measurable in $H_\rw$. This result has been proved already in \cite[Section 4.3]{foiasprodi76}, but we include the following proof for the sake of completeness.

\begin{lem}
  \label{sigmaomegauniversal}
  If $E$ is a Borel subset of $H_\rw$ and $t\geq 0$, then $\Sigma_t E$ is a universally measurable set in $H_\rw$. 
\end{lem}

\begin{proof}
  The proof is similar to that of Lemma \ref{orbituniversal}. In view of the characterization \eqref{sigmaandpirk} of the set $\Sigma_t E$ given in Lemma \ref{charactsigma} and the fact that the collection of universally measurable sets is a $\sigma$-algebra it suffices to show that
  \[ \Pi_t(\Ucal_{[0,\infty)}(R)\cap\Pi_0^{-1}(E))
  \]
is universally measurable, for any given $R>0$.
  
  Since $E$ is Borel in $H_\rw$ and $\Pi_0$ is continuous, the set $\Pi_0^{-1}(E)$ is a Borel subset of $\Ccal_\loc([0,\infty),H_\rw)$. Since $\Ucal_{[0,\infty)}(R)$ is Borel in $\Ccal_\loc([0,\infty),B_H(R)_\rw)$, then $\Ucal_{[0,\infty)}(R)\cap\Pi_0^{-1}(E))$ is a Borel set in the Polish space $\Ccal_\loc([0,\infty),B_H(R)_\rw)$. Hence, the image of this set through the continuous map $\Pi_t$ from the Polish space $\Ccal_\loc([0,\infty),B_H(R)_\rw)$ into the Polish space $B_H(R)_\rw$ is analytic in $B_H(R)_\rw$, hence universally measurable in $H_\rw$.
\end{proof}
\medskip

In the case the initial set is strongly compact, we obtain the following result.

\begin{lem}
  \label{sigmaomegacompact}
  If $K$ is strongly compact in $H$ and $t\geq 0$, then $\Sigma_t  K$ is weakly compact in $H$.
\end{lem}

\begin{proof}
  If $K$ is strongly compact in $H$, then $K$ is bounded and we can write
  \[ \Sigma_t(K) = \Pi_t(\Ucal_{[0,\infty)}(R)\cap\Pi_0^{-1}K),
  \] 
  for $R\geq R_0$ sufficiently large.  Since $K$ is strongly compact, one can check that   the limit of weak solutions in $\Ucal_{[0,\infty)}(R)\cap\Pi_0^{-1}K$ is also strongly continuous at the origin hence this limit belongs to this set, which is hence a compact set in $\Ccal_\loc([0,\infty),H_\rw)$. Thus the continuous image by $\Pi_t$ is also compact in $H_\rw$, which means that $\Sigma_t(K)$ is weakly compact in $H$.
\end{proof}
\medskip

In relation with the accretion property that we recall and use below, it is relevant to address the regularity of the composition of two multi-valued evolution operators. Notice that if we simply apply the previous results, then we have, for instance, $\Sigma_s E$ universally measurable, for $E$ Borel, but then we cannot guarantee that $\Sigma_t\Sigma_s E$ is universally measurable. We may however prove directly that this set is analytic by using the \eqref{sigmatwiceandpirk} given in Lemma \ref{charactsigmatwice}.

\begin{lem}
  \label{measurabilitysigmats}
  Let $E$ be a Borel set in $H$ and let $t,s\geq 0$. Then $\Sigma_t\Sigma_s E$ is universally measurable in $H_\rw$.
\end{lem}

\begin{proof}
  From the representation \eqref{sigmatwiceandpirk} given in Lemma \ref{charactsigmatwice} we have that $\Sigma_t\Sigma_s E$ is a countable union of the sets
  \begin{equation}
    \label{buildingblocksforcompositionofSigmas} 
      \Pi_{t+s}
        \left(\Ucal_{[0,\infty)}(R_k)\cap(\Pi_{[s,\infty)}^{-1}\Ucal_{[s,\infty)}(R_k))\cap\Pi_0^{-1}E\right),
  \end{equation}
  for any given sequence $\{R_k\}_{k\in\NN}$ of positive real numbers with $R_k\geq R_0$ and $R_k\rightarrow \infty$. In this case, for each $k$ we consider the projection $\Pi_{t+s}$ restricted to the Polish space $\Ccal_\loc([0,\infty),B_H(R_k)_\rw)$, with values in the Polish space $B_H(R_k)_\rw$, and obtain, as in the proof of Lemma \ref{sigmaomegauniversal}, that each set in \eqref{buildingblocksforcompositionofSigmas} is analytic in $B_H(R_k)_\rw$, hence universally measurable in $H_\rw$. Therefore, begin a countable union of such sets, $\Sigma_t\Sigma_s E$ is universally measurable. 
\end{proof}
\medskip

For the recurrence result also studied below, the following lemma is needed.

\begin{lem}
  \label{sigmagammameas}
  Let $E$ be a Borel set in $H$ and let $t\geq 0$. Then the set $\Sigma_t\gamma(E)$ is universally measurable in $H_\rw$.
\end{lem}

\begin{proof}
  Notice that we can write
  \[ \Sigma_t\gamma(E) = 
       \Pi_t \left(\Ucal_{[0,\infty)}\cap\sigma([0,\infty)\times(\Ucal_{[0,\infty)}\cap\Pi_0^{-1}E))\right).
  \]
  Then, the proof follows very much as in the previous results. We omit the details.
\end{proof}
\medskip

As in the classical case of a well-defined semigroup, one has the compactness of the $\omega$-limit sets. We omit the proof.
 
\begin{lem}
  \label{compactomega}
  Given $\bu\in \Ucal_{[0,\infty)}$, the $\omega$-limit set $\omega_{\Ccal_\loc([t_0,\infty),H_\rw)}(\bu)$ is compact in $\Ccal_\loc([t_0,\infty),H_\rw)$ and the $\omega$-limit set $\omega_{H_\rw}(\bu)$ is compact in $H_\rw$.
\end{lem}

\section{Time-dependent statistical solutions}
\label{secstatsol}

In this section we recall the definition and the main properties of time-dependent statistical solutions given in \cite{frtssp1}; they will be useful for our study of stationary statistical solutions (see also  \cite{foias72,foias73,foiasprodi76,vishikfursikov78,vishikfursikov88,fmrt2001a}).

\subsection{Cylindrical test functions}
For the definition of statistical solutions one considers appropriate test functions in $\Ccal(H_\rw)$, which we define as follows.
\begin{defs}
  \label{deftestfunction}
  The cylindrical test functions are the functionals $\Phi:H\rightarrow \RR$ of the form
  \begin{equation}
    \label{cylindricalfunctions}
    \Phi(\bu) = \phi((\bu,\bv_1),\ldots,(\bu,\bv_k)),
  \end{equation}
  where $k\in\NN$, $\phi$ is a $C^1$ real-valued function on $\RR^k$ with compact support, and $\bv_1,\ldots,\bv_k$ belong to $V$. For such $\Phi$, we denote by $\Phi'$ its Fr\'echet derivative in $H$, which has the form
  \[ \Phi'(\bu) = \sum_{j=1}^k
         \partial_j\phi((\bu,\bv_1),\ldots,(\bu,\bv_k))\bv_k,
  \]
  where $\partial_j\phi$ is the derivative of $\phi$ with respect to its $j$-th coordinate.
\end{defs}

As remarked in \cite{frtssp1}, the set of cylindrical test functions restricted to a bounded ball $B_H(R)_\rw$, $R>0$, is dense in the space $\Ccal(B_H(R)_\rw)$.

\subsection{Time-dependent statistical solutions}
Time-dependent statistical solutions are defined in the following way.
\renewcommand{\theenumi}{\roman{enumi}}
\begin{defs}
  \label{deftimedependetstatisticalsolution}
  For a given interval $I\subset \RR$, a family $\{\mu_t\}_{t\in I}$ of Borel probability measures on $H$ is called a statistical solution of the Navier-Stokes equations over $I$ if the following conditions hold:
  \begin{enumerate}
    \item \label{ssphimeas} The function
      \[ t\mapsto \int_H \Phi(\bu) \;\rd\mu_t(\bu)
      \]
      is measurable on $I$, for every bounded and continuous real-valued function $\Phi$ on $H$;
    \item \label{ssenergy} The function
      \[  t\mapsto \int_H |\bu|_{L^2}^2 \;\rd\mu_t(\bu)
      \]
      belongs to $L_\loc^\infty(I)$;
    \item \label{ssenstrophy} The function
      \[ t\mapsto \int_H \|\bu\|_{H^1}^2 \;\rd\mu_t(\bu)
      \]
      belongs to $L_\loc^1(I)$;
    \item \label{ssliouville} For any cylindrical test function $\Phi$, the Liouville-type equation
      \begin{equation}
        \label{liouvilleeq}
        \int_H \Phi(\bu) \;\rd\mu_t(\bu) = \int_H \Phi(\bu) \;\rd\mu_{t'}(\bu)
           + \int_{t'}^t \int_H \dual{\bF(\bu),\Phi'(\bu)}_{V',V}\;\rd\mu_s(\bu)\;\rd s
      \end{equation}
      holds for all $t', t\in I$, where $\bF(\bu) = \bbf - \nu A\bu - B(\bu,\bu)$, so that 
      \[ \dual{\bF(\bu),\Phi'(\bu)}_{V',V}=(\bbf,\Phi'(\bu))_{L^2}-\nu
           \Vinner{\bu,\Phi'(\bu)}-b(\bu,\bu,\Phi'(\bu));
      \]
    \item \label{ssstmeaneneineq}
      The strengthened mean energy inequality holds on $I$, i.e. there exists a set $I'\subset I$ of full measure in $I$ such that for any nonnegative, increasing, continuously-differentiable real-valued function $\psi:[0,\infty)\rightarrow \RR$ with bounded derivative, the inequality
      \begin{multline}
        \label{strengthenedmeanenergyineq}
        \frac{1}{2}\int_H \psi(|\bu|_{L^2}^2)\;\rd\mu_t(\bu)
         + \nu\int_{t'}^t \int_H \psi'(|\bu|_{L^2}^2)\|\bu\|_{H^1}^2
                 \;\rd\mu_s(\bu)\rd s \\
         \leq \frac{1}{2}\int_H \psi(|\bu|_{L^2}^2)\;\rd\mu_{t'}(\bu)
         + \int_{t'}^t \int_H \psi'(|\bu|_{L^2}^2)(\bbf,\bu)_{L^2}
                 \;\rd\mu_s(\bu)\rd s
      \end{multline}
      holds for any $t'\in I'$ and for all $t\in I$ with $t\geq t'$;
    \item \label{ssinitialtime}
      If $I$ is closed and bounded on the left with left end point $t_0$, then the function
      \[ t\mapsto \int_H \psi(|\bu|_{L^2}^2) \;\rd\mu_t(\bu)
      \]
      is continuous at $t=t_0$ from the right, for every function $\psi$ as in \eqref{ssstmeaneneineq}.
 \end{enumerate}
\end{defs}

The existence of time-dependent statistical solutions in the sense above was first proved in \cite{foias72} via Galerkin approximation (see Theorem 1 in page 254 and Proposition 1 in page 291 in \cite{foias72}). The existence result can be stated in the following way.

\begin{thm}[\cite{foias72}]
  \label{propexistencetdss}
  Let $t_0\in \RR$ and let $\mu_0$ be a Borel Probability measure on $H$ satisfying
  \[  \int_H |\bu|_{L^2}^2\;\rd\mu_0(\bu) < \infty.
  \]
  Then, there exists a time-dependent statistical solution $\{\mu_t\}_{t\geq t_0}$ satisfying $\mu_{t_0}=\mu_0$.
\end{thm}

\subsection{Vishik-Fursikov measures}

In the framework of Vishik and Fursikov \cite{vishikfursikov78}, the statistical solutions are obtained via measures in time-dependent function spaces. What makes them measures relevant to fluid flows is the condition that they be carried by the space of individual weak solutions. They should also have finite kinetic energy. Inspired by their approach, the following definition was introduced in \cite{frtssp1}, which is a slight modification of their original definition.

\renewcommand{\theenumi}{\roman{enumi}}
\begin{defs}
  \label{defvfmeasure}
  For a given interval $I\subset \RR$, a
  Vishik-Fursikov measure over $I$  
  is defined as a Borel probability measure $\rho$ on the space
  $\Ccal_\loc(I,H_\rw)$ with the following properties
  \begin{enumerate}
    \item \label{rhocarrierforvfm} $\rho$ is carried by $\Ucal_I^\sharp$;
    \item \label{rhomeankineticforvfm} We have
      \[ t\rightarrow \int_{\Ucal_I^\sharp} |\bu(t)|_{L^2}^2 \;\rd\rho(\bu) \in L_\loc^\infty(I);
      \]
    \item \label{meancontinuityatinitialtimeforvfm}
      If $I$ is closed and bounded on the left, with
      left end point $t_0$, then for any nonnegative, increasing 
      continuously-differentiable real-valued function $\psi:[0,\infty)\rightarrow \RR$ 
      with bounded derivative, we have
      \[ \lim_{t\rightarrow t_0^+} \int_{\Ucal_I^\sharp} \psi(|\bu(t)|_{L^2}^2) \;\rd\rho(\bu)
               = \int_{\Ucal_I^\sharp} \psi(|\bu(t_0)|_{L^2}^2) \;\rd\rho(\bu) < \infty.
      \]
  \end{enumerate}
\end{defs}

The following existence result was proved in \cite{frtssp1}:

\begin{thm}[{\cite[Theorem 3.7]{frtssp1}}]
  \label{propexistencevfmeasure}
  Let $t_0\in \RR$ and let $\mu_0$ be a Borel probability measure on $H$ with
  finite mean kinetic energy, i.e.
  \[ \int_H |\bu|_{L^2}^2 \;\rd\mu_0(\bu) < \infty.
  \]
  Then, there exists a Vishik-Fursikov measure $\rho$ over the time interval
  $[t_0,\infty)$ such that $\Pi_{t_0}\rho=\mu_0$.
\end{thm}

It has in fact been proved in \cite{frtssp1} that Vishik-Fursikov measures are carried by $\Ucal_I$ itself, recovering an important fact for these measures, namely that they are carried by Leray-Hopf weak solutions on $I$:

\begin{thm}[{\cite[Theorem 4.1]{frtssp1}}]
  \label{vfsscarriedbyus}
  Let $I\subset \RR$ be an arbitrary interval. Let $\rho$ be a Vishik-Fursikov measure over $I$.   Then $\rho$ is carried by $\Ucal_I$.
\end{thm}

\subsection{Time-dependent Vishik-Fursikov statistical solutions}

A Borel probability measure $\rho$ on $\Ccal_\loc(I,H_\rw)$ induces a time-dependent family
of Borel probability measures $\{\rho_t\}_{t\in I}$ on the phase space $H$ (or, equivalently, on $H_\rw$, since their Borel $\sigma$-algebras coincide) through the projections $\rho_t=\Pi_t\rho,$ with
\begin{equation}
  \label{2.15a}
  \int_H \Phi(\bu) \;\rd\rho_t(\bu)
    = \int_{\Ccal_\loc(I,H_\rw)} \Phi(\bv(t)) \;\rd\rho(\bv),
    \qquad \forall t\in I,
\end{equation}
for all $\Phi$ which belongs to $L^1(\rho_t)$ for any $t\in I$. In particular, this is valid for any
$\Phi$ in $\Ccal_\rb(H_\rw)$. Relation \eqref{2.15a} is also valid for $\Phi \in \Ccal_\rb(H)$, since such functions are Borel measurable, for they can be approximated by the sequence of weakly continuous functions $(\Phi\circ P_m)_m$, where $P_m$ are the Galerkin projectors. 

It has been proved in \cite[Theorems 3.13 and 3.14]{frtssp1} that if $\rho$ is Vishik-Fursikov measure over an arbitrary interval $I$, then the corresponding family of projections $\{\rho_t\}_{t\in I}$ is a statistical solution. This yields a particular type of statistical solution, as formally defined below.

\begin{defs}
  A Vishik-Fursikov statistical solution of the Navier-Stokes equations over an interval $I\subset \RR$ is a statistical solution $\{\rho_t\}_{t\in I}$ such that $\rho_t=\Pi_t\rho$, for all $t\in I$, for some Vishik-Fursikov measure $\rho$ over the interval $I$.
\end{defs}

The existence of Vishik-Fursikov statistical solutions was proved in \cite{frtssp1}.

\begin{thm}[{\cite[Theorem 3.16]{frtssp1}}]
  \label{thmexistencevfss}
  Let $t_0\in \RR$ and let $\mu_0$ be a Borel probability measure on $H$ satisfying
  \[  \int_H |\bu|_{L^2}^2\;\rd\mu_0(\bu) < \infty.
  \]
  Then, there exists a Vishik-Fursikov statistical solution $\{\rho_t\}_{t\geq t_0}$ over the interval $I=[t_0,\infty)$ satisfying $\rho_{t_0}=\mu_0$.
  \qed
\end{thm}

\section{Stationary statistical solutions}
\label{secsss}

In this section we start our main investigation on the concept of stationary statistical solution. We first recall the original definition given in \cite{foias73}, and then introduce and study the stationary statistical solutions which are associated with Vishik-Fursikov measures, followed by the study of these particular stationary statistical solutions which are obtained as generalized limits of time averages of individual Leray-Hopf weak solutions.

\subsection{Stationary statistical solutions}
The time-independent version of the statistical solution given in Definition \ref{deftimedependetstatisticalsolution} is known as a stationary statistical solution, whose definition is the following.
\renewcommand{\theenumi}{\roman{enumi}}
\begin{defs}
  \label{defsss}
  A stationary statistical solution on $H$ of the Navier-Stokes equations is a Borel probability measure $\mu$ on $H$ such that
  \begin{enumerate}
    \item \label{sssenergy} The mean enstrophy is finite, i.e. 
      \[ \int_H \|\bu\|_{H^1}^2 \;\rd\mu(\bu) < \infty; \]
    \item \label{sssliouville} For any cylindrical test function $\Phi$, the following stationary Liouville-type equation holds, 
      \[ \int_H (\bF(\bu),\Phi'(\bu))_{L^2}\;\rd\mu(\bu) = 0;\]
    \item \label{sssstmeaneneineq} For any nonnegative, increasing, continuously-differentiable real-valued function $\psi:[0,\infty)\rightarrow \RR$ with bounded derivative, we have
      \begin{equation}
        \label{strengthenedstationarymeanenergyineq}
          \int_H \psi'(|\bu|_{L^2}^2)\left( \nu \|\bu\|_{H^1}^2 - (\bbf,\bu)_{L^2} \right) 
            \;\rd\mu(\bu) \leq 0.
      \end{equation}
   \end{enumerate}
\end{defs}

\begin{rmk}
In \cite{fmrt2001a} a slightly different condition was used instead of \eqref{sssstmeaneneineq}, 
namely, that 
\begin{equation}
    \int_{e_1\leq \frac{1}{2}|\bu|_{L^2}^2 < e_2} \{\nu\|\bu\|_{H^1}^2 - (\bbf,\bu)_{L^2}\}
           \;\rd\mu(\bu) \leq 0,
\end{equation}
holds for all $e_1, e_2$ with $0\leq e_1 < e_2\leq \infty$. Both conditions are in fact equivalent, and we adopt the form in \eqref{sssstmeaneneineq} since it seems more natural given the condition \eqref{ssstmeaneneineq} of the Definition \ref{deftimedependetstatisticalsolution} of time-dependent statistical solutions. As we mentioned earlier, the energy inequality in that condition \eqref{ssstmeaneneineq} is one of those used in \cite{foias72} and is stronger than the one used in \cite{fmrt2001a}. Its advantage is that it makes the relation between stationary and time-dependent statistical solutions consistent, in the sense that with this definition stationary statistical solutions are in fact time-independent versions of time-dependent statistical solutions.
\end{rmk}

Condition \eqref{sssenergy} in the Definition \ref{defsss}, implies in particular, that $\mu$ is carried by $V$, i.e.
\begin{equation}
  \label{ssscarriedbyV}
  \mu(H\setminus V) = 0.
\end{equation}
Moreover, it is known that the energy inequality \eqref{sssstmeaneneineq} in the Definition \ref{defsss} implies that a stationary statistical solution is carried by $B_H(R_0)$ (see e.g. \cite{fmrt2001a}), i.e.
\begin{equation}
  \label{sssboundedinH}
  \mu(B_H(R_0))=1,
\end{equation}
where $R_0$ is given in \eqref{defR0}. 

\subsection{Vishik-Fursikov stationary statistical solutions}

The Vishik-Fursikov stationary statistical solutions are defined as projections of Vishik-Fursikov measures which are invariant by translations in time. For that to make sense, the time interval need to be unbounded on the right.

\begin{defs}
  \label{vfinvmeasdef}
  Let $I\subset \RR$ be an interval unbounded on the right. An invariant Vishik-Fursikov measure over $I$ is a Vishik-Fursikov measure $\rho$ over $I$ which is invariant with respect to the translation semigroup $\{\sigma_\tau\}_{\tau\geq 0}$, in the sense that $\sigma_\tau\rho = \rho$ for all $\tau\geq 0$,  i.e. $\rho(E) = \rho(\sigma_\tau^{-1}E)$, for all Borel sets $E$ in $\Ccal_\loc(I,H_\rw)$.
\end{defs}

\begin{rmk} 
 Since $H_\rw$ is the union of the balls $B_H(R)_\rw$, with $R>0$, and each such ball is metrizable it follows, using \eqref{measequivmetric} that the condition that $\sigma_\tau\rho=\rho$ on $\Ccal_\loc(I,H_\rw)$ is
 equivalent to 
  \begin{multline}
    \int_{\Ccal_\loc(I,B_H(R)_\rw)} \varphi(\sigma_\tau\bu)\;\rd\rho(\bu)
       = \int_{\Ccal_\loc(I,B_H(R)_\rw)} \varphi(\bu)\;\rd\rho(\bu), \\
         \forall \tau\geq 0, \;\forall \varphi\in\Ccal_\rb(I,B_H(R)_\rw), \;\forall R>0.
  \end{multline}
\end{rmk}

It is immediate to see that a Vishik-Fursikov statistical solution obtained from the projection of an invariant Vishik-Fursikov statistical solution is a time-independent statistical solution in the sense that the measures $\rho_t=\Pi_t\rho$ do not change with $t$, and hence this leads to a stationary statistical solution in the sense of Definition \ref{defsss}. Therefore, we make the following definition.

\begin{defs}\label{defvfsss}
  A Vishik-Fursikov stationary statistical solution on $H$ is a Borel probability measure $\rho_0$ on $H$ which is a projection $\rho_0=\Pi_t\rho$, at an arbitrary time $t\in I$, of an invariant Vishik-Fursikov measure $\rho$ over an interval $I\subset \RR$ unbounded on the right.
\end{defs}

\begin{rmk}
  \label{rmkdiffstatvfsol}
  At this point we would like to draw the attention of the reader to a subtle distinction between a Vishik-Fursikov stationary statistical solution and a stationary Vishik-Fursikov statistical solution. The former notion, as given in Definition \ref{defvfsss}, refers to the projection $\rho_0=\Pi_t\rho$ for which $\rho$ is invariant in the trajectory space, i.e. 
  \[ \int_{\Ucal_I^\sharp} \varphi(\sigma_t\bu)\;\rd\rho(\bu)
  \]
   is independent of $t$, for any continuous and bounded real-valued function $\varphi$ defined in the trajectory space $\Ucal_I^\sharp$. The second notion refers to a measure $\rho_0=\Pi_t\rho$ in which $\Pi_t\rho$ is invariant in phase space, i.e. for which 
   \[ \int_H \Phi(\bv)\;\rd\rho_0(\bv) = \int_{\Ucal_I^\sharp} \Phi(\Pi_t\bu)\;\rd\rho(\bu)
   \] 
   is independent of $t$, for any continuous and bounded real-valued function $\Phi$ defined on the phase space $H$. Choosing $\varphi=\Phi\circ \Pi_0$, so that $\varphi(\sigma_t \bu) = \Phi(\Pi_t\bu)$, it becomes clear that every Vishik-Fursikov stationary statistical solution is a stationary Vishik-Fursikov statistical solution but the converse statement is more delicate to establish. We prove, below, that the converse in fact holds, hence the two notions are equivalent.
\end{rmk}

\begin{thm}
  \label{stationaryisstationary}
  Let $\rho$ be a Vishik-Fursikov measure over an interval $I\subset \RR$ which is unbounded on the right. If $\Pi_t\rho$ is independent of $t\in I$, then there exists an invariant Vishik-Fursikov measure $\tilde\rho$ over $I$ such that 
  \[ \Pi_t \rho = \Pi_t \tilde\rho, \quad \forall t\in I.
  \]
  In other words, if the family $\{\Pi_t\rho\}_{t\in I}$ is a Vishik-Fursikov statistical solution whose measures do not change with time, then this family is a Vishik-Fursikov stationary statistical solution.
\end{thm}

\begin{proof}
  Let $\Lim$ be a given generalized limit (see Section \ref{genlimsec}). Since the map $(s,\bu) \mapsto \sigma_s \bu$ is continuous from $I\times \Ccal_\loc(I,H_\rw)$ into $\Ccal_\loc(I,H_\rw)$ (see \eqref{sigmacontinuous}), and in particular from $I\times \Ucal_I^\sharp(R)\rightarrow \Ucal_I^\sharp(R)$, it follows that the composition $(s,\bu)\mapsto \varphi(\sigma_s \bu)$ is also continuous from $I\times \Ucal_I^\sharp(R)$ into $\RR$, for any continuous and bounded functional $\varphi:\Ucal_I^\sharp(R)\rightarrow \RR$, with $R>0$. Thus, the real-valued map
  \[ T \rightarrow  \frac{1}{T}\int_0^T \int_{\Ucal_I^\sharp(R)} \varphi(\sigma_s \bu) \;\rd\rho(\bu)\;\rd s
  \]
  is well-defined for each $T>0$ and is uniformly bounded in $T$. Hence, the functional
  \[ \varphi \mapsto \Lambda_R(\varphi) = \Lim \frac{1}{T}\int_0^T \int_{\Ucal_I^\sharp(R)} \varphi(\sigma_s \bu) \;\rd\rho(\bu)\;\rd s
  \]
  is well-defined and is a positive bounded linear functional on the space of continuous and bounded functionals on $\Ucal_I^\sharp(R)$. Since $\Ucal_I^\sharp(R)$ is compact, it follows by the Kakutani-Riesz Representation Theorem that the functional $\Lambda_R$ can be represented by a finite Borel measure $\tilde\rho_R$ on $\Ucal_I^\sharp(R)$ in the sense that
  \[ \Lim \frac{1}{T}\int_0^T \int_{\Ucal_I^\sharp(R)} \varphi(\sigma_s \bu) \;\rd\rho(\bu)\;\rd s = \Lambda_R(\varphi) = \int_{\Ucal_I^\sharp(R)} \varphi(\bu)\;\rd\tilde\rho_R(\bu),
  \]
  for all $\varphi\in \Ccal_\rb(\Ucal_I^\sharp(R))$. By taking $\varphi$ identically equal to $1$, we see that
  \[ \tilde\rho_R(\Ucal_I^\sharp(R)) = \rho(\Ucal_I^\sharp(R)).
  \]
  Notice also that if $\varphi$ is a nonnegative, continuous and bounded function on $\Ucal_I^\sharp(R')$, for $R' > R$, then
  \[ \int_{\Ucal_I^\sharp(R)} \varphi(\bu)\;\rd\tilde\rho_R(\bu) \leq \int_{\Ucal_I^\sharp(R')} \varphi(\bu)\;\rd\tilde\rho_{R'}(\bu).
  \]
  Since $\Ucal_I^\sharp(R)$ is compact and metrizable, this implies that
  \[ \tilde\rho_R(\Ecal) = \tilde\rho_{R'}(\Ecal), \qquad \forall R' > R,
  \]
  for any Borel subset $\Ecal\subset \Ucal_I^\sharp(R)$. Then, if $\Ecal\subset \Ucal_I^\sharp$ is a Borel set, it follows that
  \[ \tilde\rho_R(\Ecal\cap \Ucal_I^\sharp(R)) = \tilde\rho_{R'}(\Ecal\cap\Ucal_I^\sharp(R)) \leq \tilde\rho_{R'}(\Ecal\cap\Ucal_I^\sharp(R')) \leq \rho_{R'}(\Ucal_I^\sharp(R')) = \rho(\Ucal_I^\sharp(R')) \leq 1.
  \]
  Thus, $\tilde\rho_R(\Ecal\cap\Ucal_I^\sharp(R))$ is monotonic and bounded. Therefore, given any Borel set $\Ecal\subset \Ucal_I^\sharp$, the limit
  \[ \lim_{R\rightarrow \infty} \tilde\rho_R(\Ecal\cap\Ucal_I^\sharp(R))
  \]
  is well-defined and is bounded by $1$. This allows us to define the set function
  \[ \tilde\rho(\Ecal) = \lim_{R\rightarrow \infty} \tilde\rho_R(\Ecal\cap\Ucal_I^\sharp(R)),
  \]
  on the Borel subsets of $\Ucal_I^\sharp$. It is not difficult to prove that this set function is a Borel measure on $\Ucal_I^\sharp$. Moreover, we also have that
  \[ \int_{\Ucal_I^\sharp} \varphi(\bu) \;\rd\tilde\rho(\bu) = \lim_{R\rightarrow \infty} \int_{\Ucal_I^\sharp(R)} \varphi(\bu) \;\rd\tilde\rho_R(\bu),
  \]
  for any $\tilde\rho$-integrable functions, and in particular for any continuous and bounded real valued function $\varphi$ defined on $\Ucal_I^\sharp$. Another way of viewing the limit measure $\tilde\rho$ is by extending the measures $\tilde\rho_R$ in a trivial way to $\Ucal_I^\sharp$ by defining $\tilde\rho_R'(\Ecal) = \tilde\rho_R(\Ecal \cap \Ucal_I^\sharp(R))$, for any Borel set $\Ecal$ in $\Ucal_I^\sharp$, and noticing that $\tilde\rho$ is the limit in a strong sense of the measures $\tilde\rho_R'$, i.e. $\tilde\rho(\Ecal) = \lim_{R\rightarrow \infty} \tilde\rho_R'(\Ecal)$.
  
  By taking $\Ecal=\Ucal_I^\sharp$, we see also that
  \[ \tilde\rho(\Ucal_I^\sharp) = \lim_{R\rightarrow \infty} \tilde\rho_R(\Ucal_I^\sharp(R)) = 1,
  \]
  so that $\tilde\rho$ is in fact a Borel probability measure on $\Ucal_I^\sharp$. Thus, $\tilde\rho$ is a Vishik-Fursikov measure over $I$. 
  
  For the projections $\Pi_t\tilde\rho$, we have that, for every continuous and bounded real valued function $\Phi$ on $H$,
  \begin{align*}
    \int_H \Phi(\bu)\;\rd\Pi_t\tilde\rho & = \int_{\Ucal_I^\sharp} \Phi(\Pi_t\bu)\;\rd \tilde\rho(\bu) \\
    & = \lim_{R\rightarrow \infty} \int_{\Ucal_I^\sharp(R)} \Phi(\Pi_t\bu)\;\rd \tilde\rho_R(\bu) \\
    & = \lim_{R\rightarrow \infty} \Lim \frac{1}{T}\int_0^T \int_{\Ucal_I^\sharp(R)} \Phi(\Pi_t\sigma_s \bu) \;\rd\rho(\bu)\;\rd s \\
    & = \lim_{R\rightarrow \infty} \Lim \frac{1}{T}\int_0^T \int_{\Ucal_I^\sharp(R)} \Phi(\Pi_{t+s}\bu) \;\rd\rho(\bu)\;\rd s.
  \end{align*}
  Now, since $\Pi_t \rho$ is independent of $t$, we find that
  \[ \int_{\Ucal_I^\sharp(R)} \Phi(\Pi_{t+s} \bu) \;\rd\rho(\bu) = \int_{\Ucal_I^\sharp(R)} \Phi(\Pi_t \bu) \;\rd\rho(\bu),
  \] 
  for all $s\geq 0$, and hence
  \begin{multline*}
    \int_H \Phi(\bu)\;\rd\Pi_t\tilde\rho = \lim_{R\rightarrow \infty} \Lim \frac{1}{T}\int_0^T \int_{\Ucal_I^\sharp(R)} \Phi(\Pi_t \bu) \;\rd\rho(\bu)\;\rd s \\
      = \lim_{R\rightarrow \infty} \int_{\Ucal_I^\sharp(R)} \Phi(\Pi_t \bu) \;\rd\rho(\bu) = \int_{\Ucal_I^\sharp} \Phi(\Pi_t \bu) \;\rd\rho(\bu) = \int_H \Phi(\bu)\;\rd\Pi_t\rho.
  \end{multline*}
  Since this is valid for arbitrary continuous and bounded real valued functions $\Phi$ on $H$, it follows that $\Pi_t\tilde\rho = \Pi_t\rho$, for every $t\in I$.
  
  It remains to show that $\tilde\rho$ is time invariant. Consider $R>0$ arbitrary and let $\varphi$ be a continuous and bounded real valued function on $\Ucal_I^\sharp(R)$. Since the map $\sigma_\tau$ maps the space $\Ucal_R^\sharp(R)$ into itself and is a continuous map within these spaces, the composition $\varphi\circ\sigma_\tau$ is also continuous and bounded on $\Ucal_I^\sharp(R)$. Then, using property \eqref{genliminvariance} of the generalized limits (Section \ref{genlimsec}), we find that
  \begin{align*} 
    \int_{\Ucal_I^\sharp(R)} \varphi(\sigma_\tau\bu) \;\rd\tilde\rho_R(\bu) & = \Lim \frac{1}{T}\int_0^T \int_{\Ucal_I^\sharp(R)} \varphi(\sigma_s \sigma_\tau \bu) \;\rd\rho(\bu)\;\rd s \\
    & = \Lim \frac{1}{T}\int_0^T \int_{\Ucal_I^\sharp(R)} \varphi(\sigma_s \sigma_\tau \bu) \;\rd\rho(\bu)\;\rd s \\
    & = \int_{\Ucal_I^\sharp(R)} \varphi(\bu) \;\rd\tilde\rho_R(\bu).
  \end{align*}
  Since $\Ucal_I^\sharp(R)$ is metrizable, this means that
  \[ \sigma_\tau \tilde\rho_R = \tilde\rho_R,
  \]
  for all $\tau \geq 0$, which means that $\tilde\rho_R$ is translation invariant. Here it is important to remark that when we say that $\tilde\rho_R$ is invariant by $\sigma_\tau$ we are actually considering the operator $\sigma_\tau$ restricted to $\Ucal_I^\sharp(R)$, so that when we take the inverse $\sigma_\tau^{-1}\Ecal$ of a Borel set $\Ecal$ in $\Ucal_I^\sharp(R)$, we are only considering the points within $\Ucal_I^\sharp(R)$ which are taken to $\Ecal$ by $\sigma_\tau$ (without this restriction, the inverse $\sigma_\tau^{-1}\Ecal$ could have elements outside $\Ucal_I^\sharp(R)$ which are taken inside that space). 
  
  Finally, since for any Borel set $\Ecal$ in $\Ucal_I^\sharp$ we have that $\tilde\rho_R(\Ecal\cap \Ucal_I^\sharp(R))$ converges to $\tilde\rho(\Ecal)$ when $R\rightarrow \infty$, we see that
  \begin{multline*} 
    \sigma_\tau \tilde\rho(\Ecal) = \tilde\rho(\sigma_\tau^{-1}\Ecal) = \lim_{R\rightarrow \infty} \tilde\rho_R(\sigma_\tau^{-1}\Ecal \cap \Ucal_I^\sharp(R)) \\
     = \lim_{R\rightarrow \infty} \tilde\rho_R(\sigma_\tau^{-1} (\Ecal \cap \Ucal_I^\sharp(R))) = \lim_{R\rightarrow \infty} \tilde\rho_R(\Ecal \cap \Ucal_I^\sharp(R)) = \tilde\rho(\Ecal),
  \end{multline*}
  This shows that $\tilde\rho$ is also translation invariant, which completes the proof.
\end{proof}

\subsection{Time-average stationary statistical solutions}
\label{timeavesssec}

A particular type of stationary statistical solution is obtained by taking the limit of time averages of weak solutions as the averaging time goes to infinity. This is based on the classical idea of Krylov and Bogoliubov for obtaining invariant measures for dynamical systems \cite{kb1937}. This idea was exploited in \cite{bcfm1995}, in connection with the notion of generalized limit, to construct and study invariant measures for the two-dimensional Navier-Stokes equations. The corresponding case of the three-dimensional Navier-Stokes equations was considered in \cite{fmrt2001a}. This generalized limit of time averages was used later in \cite{wang2009,lukaszewickrealrobinson2011,chekrounglattholtz2012} to yield invariant measures to a large class of dissipative systems. See also \cite{dapratozabczyk1996} for a related construction in the context of stochastic equations.

More precisely, for a given weak solution $\bu=\bu(t)$ defined for $t\geq t_0$, for some $t_0\in \RR$, we know, thanks to the a~priori estimate \eqref{energyestimate}, that $|\bu(t)| \leq R$, for all $t\geq t_0$, for a sufficiently large $R>0$. Then, the map 
\[ \Phi \mapsto \Lim \frac{1}{T}\int_{t_0}^{t_0+T} \Phi(\bu(t))\;\rd t
\]
defines a positive continuous linear functional on $\Ccal(B_H(R)_\rw)$. Since $B_H(R)_\rw$ is compact, it follows by the Kakutani-Riesz Representation Theorem that this functional defines a measure $\mu$ on $B_H(R)_\rw$ through the relation
\[ \int_H \Phi(\bv) \dmu(\bv)
        = \Lim \frac{1}{T}\int_{t_0}^{t_0+T} \Phi(\bu(t))\;\rd t,
        \quad \forall \Phi\in\Ccal(B_H(R)_\rw).
\]
We then extend this measure to all $H$ by simply setting $\mu(E)=\mu(E\cap B_H(R))$, for any Borel set $E$ in $H$. It is clear that this definition of $\mu$ is independent of $R$ and depends only on $\bu$ and the choice of the generalized limit. We call this measure $\mu$ a time-average stationary statistical solution:

\begin{defs}
  \label{deftimeavesss}
  Let $\Lim$ be a generalized limit, let $t_0\in \RR$, and let $\bu=\bu(t)$, $t\geq t_0$, be a weak solution on $[t_0, \infty)$. Let $R>0$ be such that $\bu(t)\in B_H(R)$, for all $t\geq t_0$. Then, the associated time-average stationary statistical solution is the Borel probability measure $\mu$ on $H$, with support on $B_H(R)$, which is given by the formula
  \begin{equation}
    \label{eqdeftimeaveragemeasure}
    \int_H \Phi(\bv) \dmu(\bv)
        = \Lim \frac{1}{T}\int_0^T \Phi(\bu(t_0+t))\;\rd t,
  \end{equation}
  for all $\Phi\in \Ccal(B_H(R)_\rw)$.
\end{defs}

We define a similar notion in the trajectory space. We consider an interval $I$ which is unbounded on the right and a bounded weak solution $\bu$ on $I$, say $|\bu(t)|_{L^2}\leq R$, for all $t\in I$, and some $R>0$. The set $\Ccal_\loc(I,B_H(R)_\rw)$ is not locally compact, hence we cannot apply the Kakutani-Riesz Representation Theorem directly. Note, however, that $\bu$ belongs to $\Ucal_I^\sharp(R)$, which is compact (Lemma \ref{ucalitildercompact}). Hence, there exists a probability measure $\rho$ on $\Ucal_I^\sharp(R)$ such that
\[ \int_{\Ucal_I^\sharp(R)} \varphi(\bv) \;\rd\rho(\bv)
    = \Lim \frac{1}{T}\int_t^{t+T} \varphi(\sigma_s\bu))\;\rd s,
\]
for all $\varphi\in \Ccal(\Ucal_I^\sharp(R))$ and any $t\in I$. Then, we extend the measure $\rho$ to all the space $\Ccal_\loc(I,B_H(R)_\rw)$ in the obvious way. Moreover, since any $\varphi$ continuous on $\Ccal_\loc(I,B_H(R)_\rw)$ can be restricted to $\Ccal(\Ucal_I^\sharp(R))$, the relation above can be extended to all such $\varphi$. This measure $\rho$ is invariant thanks to property \eqref{genliminvariance} of generalized limits. Thus, we have the following definition.

\begin{defs}
  \label{deftimeaveVFinvmea}
  Let $\Lim$ be a generalized limit and let $\bu$ be a weak solution on an interval $I$ unbounded on the right. Assume $\bu$ is bounded in $H$ and let $R>0$ be such that $\bu\in\Ucal_I^\sharp(R)$. Then, the associated time-average invariant Vishik-Fursikov measure $\rho$ over $I$ is given by the Kakutani-Riesz formula
  \begin{equation}
    \label{defvftimeaveragemeasure}
   \int_{\Ucal_I^\sharp(R)} \varphi(\bv) \;\rd\rho(\bv)
    = \Lim \frac{1}{T}\int_t^{t+T} \varphi(\sigma_s\bu))\;\rd s,
  \end{equation}
  valid for all $\varphi\in \Ccal(\Ucal_I^\sharp(R))$ and independent of $t\in I$. Moreover, \eqref{defvftimeaveragemeasure} is also valid for $\varphi\in\Ccal(\Ccal_\loc(I,B_H(R)_\rw))$.
\end{defs}

The corresponding projection on the phase space generate
the time-average Vishik-Fursikov stationary statistical solution:
\begin{defs}
  \label{defvftimeaverageVFsss}
  A time-average Vishik-Fursikov stationary statistical solution on $H$ is a Borel probability measure $\rho_0$ on $H$ which is a projection $\mu=\Pi_t\rho$, at an arbitrary time $t\in I$, of a time-average invariant Vishik-Fursikov measure $\rho$ over an interval $I$ unbounded on the right.
\end{defs}

\section{Properties of stationary statistical solutions}
\label{seccarrier}

In this section we study the support and carriers of invariant Vishik-Fursikov measures. The main result is a type of localization result, stating essentially that any invariant Vishik-Fursikov measure is carried by the set of trajectories that exists globally in time, both in the past and in the future, and are uniformly bounded in the phase space $H$. In other words, the support of any invariant Vishik-Fursikov measures is included in the set all the trajectories which constitute the weak global attractor. We then investigate the consequences of this result for Vishik-Fursikov stationary statistical solutions.

\subsection{Properties of invariant Vishik-Fursikov measures}

Considering an arbitrary interval $I$ unbounded on the right and a Vishik-Fursikov measure $\rho$ on $I$, it follows from Theorem \ref{vfsscarriedbyus} that $\rho$ is carried by $\Ucal_I$. If moreover, $\rho$ is an invariant measure, we expect $\rho$ to be carried by a set which is invariant by the translation semigroup. This is precisely the result of the following theorem.
\begin{thm}
  \label{thmcarriervfmeasure}
  Let $\rho$ be an invariant Vishik-Fursikov measure over an interval $I\subset \RR$ unbounded on the right. Then, $\rho$ is carried by $\Wcal_I$, as defined in \eqref{wcalidef}, and the support of $\rho$ is included in $\Wcal_I^\sharp$. In particular, $\rho$ is a regular measure in the sense of \eqref{upperregmeas} and \eqref{lowerregmeas}.
\end{thm}

\begin{proof}
  The first step of the proof is to show that $\rho$ is carried by $\Ucal^\sharp_I(R)$, for some $R>R_0$, where $R_0$ is given in \eqref{defR0}. We then use this fact together with the characterization \eqref{wcalsharpasintersectionofucals} of $\Wcal^\sharp_I$ to show that $\rho$ is supported on the compact set $\Wcal_I^\sharp$ and is a regular measure. Finally, we use Theorem \ref{vfsscarriedbyus} to deduce that $\rho$ is carried by $\Wcal_I$.
  
  The proof of the first step is divided into two cases, whether $I$ is bounded and closed on the left or $I$ is open on the left. 
  
  In the case $I$ is bounded and closed on the left, we use the characterization \eqref{characucalitildeclosed} of $\Ucal^\sharp_I$ as $\Ucal^\sharp_I = \bigcup_{k} \Ucal_I^\sharp(R_k)$, for a given sequence $\{R_k\}_{k\in \NN}$ of increasing positive numbers with $R_1>R_0$ and $R_k\rightarrow \infty$. It follows from the a~priori estimate \eqref{energyestimate} that, for each $k\in \NN$, there exists $\tau_k\geq 0$ such that
\[ \sigma_\tau\Ucal^\sharp_I(R_k) \subset \Ucal^\sharp_I(R_1), \quad \forall \tau\geq \tau_k.
\]
In particular,
\[ \Ucal^\sharp_I(R_k)\subset \sigma_{\tau_k}^{-1}\sigma_{\tau_k}\Ucal^\sharp_I(R_1) \subset \sigma_{\tau_k}^{-1}\Ucal^\sharp_I(R_1), \quad \forall k\in \NN.
\]
Then, since $\rho$ is a $\sigma_\tau$-invariant probability measure carried
by $\Ucal^\sharp_I$ and using the continuity from below of the measure $\rho$ and the fact that the sets $\Ucal_I^\sharp(R_k)$ are monotonic increasing with $k$, it follows that
\[ 1 = \rho(\Ucal^\sharp_I) = \rho(\bigcup_{k} \Ucal_I^\sharp(R_k)) = \lim_{k\rightarrow \infty} \rho(\Ucal^\sharp_I(R_k)) \leq \lim_{k\rightarrow\infty} \rho(\sigma_{\tau_k}^{-1}\Ucal^\sharp_I(R_1))
     = \rho(\Ucal^\sharp_I) \leq 1.
\]
Thus, $\rho(\Ucal^\sharp_I(R_1)) = 1$, at least in the case $I$ is closed and bounded on the left.

Now if $I$ is open on the left, we use the characterization \eqref{characucaliopen2} of the space $\Ucal^\sharp_I$ to write
\[ \Ucal^\sharp_I = \bigcap_{n\in \NN} \bigcup_{k\in \NN} \Pi_{J_n}^{-1} \Ucal^\sharp_{J_n}(R_k),
\]
for a increasing sequence $\{R_k\}_{k\in \NN}$ with $R_1>R_0$ and $R_k\rightarrow \infty$, and a sequence $\{J_n\}_{n\in \NN}$ of subintervals of $I$ which are unbounded on the right, monotonic increasing (i.e. $J_n\subset J_{n+1}$), and with $I=\cup_n J_n$. 

Since each $J_n$ is unbounded on the right, it follows from the a~priori estimate \eqref{energyestimate} that, for any $k\in \NN$ and any $n\in \NN$, there exists $\tau_{n,k}\geq 0$ such
that
\[ \sigma_\tau \Pi_{J_n}^{-1} \Ucal^\sharp_{J_n}(R) \subset \Ucal^\sharp_I(R_1), \qquad \forall \tau\geq \tau_{n,k}.
\]
In particular,
\[ \Pi_{J_n}^{-1} \Ucal^\sharp_{J_n}(R)\subset 
     \sigma_{\tau_{n,k}}^{-1}\sigma_{\tau_{n,k}}\Pi_{J_n}^{-1} \Ucal^\sharp_{J_n}(R)
     \subset \sigma_{\tau_{n,k}}^{-1}\Ucal^\sharp_I(R_1).
\]

Then, since the sets in the intersections and in the unions in the characterization of $\Ucal_I^\sharp$ above are monotonic, we find, using the continuity from above and from below of the measure $\rho$ and the invariance of the measure $\rho$, that
\begin{multline*}
  1 = \rho(\Ucal^\sharp_I) = \lim_{n\rightarrow \infty} \rho(\bigcup_{k\in \NN} \Pi_{J_n}^{-1} \Ucal^\sharp_{J_n}(R_k)) \\
    =  \lim_{n\rightarrow \infty} \lim_{k\rightarrow \infty} \rho(\Pi_{J_n}^{-1} \Ucal^\sharp_{J_n}(R_k)) \\
     \leq \lim_{k\rightarrow\infty} \rho(\sigma_{\tau_{R_k}}^{-1}\Ucal^\sharp_I(R_1))
     = \rho(\Ucal^\sharp_I(R_1)) \leq 1.
\end{multline*}
Thus, $\rho(\Ucal^\sharp_I)=1$ in this case, as well.

Now, for the second part of the proof, we use the characterization \eqref{wcalsharpasintersectionofucals} of $\Wcal^\sharp_I$, valid for any interval $I$ unbounded on the right, so that
\[ \Wcal^\sharp_I = \bigcap_{k\in \NN}\sigma_{\tau_k}\Ucal^\sharp_I(R_1),
\]
with $R_1>R_0$ as above, and with a positive, increasing sequence $\{\tau_k\}_{k\in\NN}$ of times with $\tau_k\rightarrow \infty$, as $k\rightarrow \infty$. Since the sequence of times is increasing, the sequence of sets in the union above is monotonic decreasing. Thus, using the continuity from above of the measure $\rho$ and that $\rho$ is a $\sigma_\tau$-invariant probability measure, we find
\[ 1\geq \rho(\Wcal^\sharp_I)
     = \lim_{k\rightarrow\infty}\rho(\sigma_{\tau_k}\Ucal^\sharp_I(R_1))
     = \lim_{k\rightarrow\infty}\rho(\sigma_{\tau_k}^{-1}\sigma_{\tau_k}\Ucal^\sharp_I(R_1))
     \geq \rho(\Ucal^\sharp_I(R_1))=1.
\]
Thus, $\rho(\Wcal^\sharp_{[t_0,\infty)})=1$, which means that $\rho$ is carried by $\Wcal^\sharp_I$. Since $\Wcal^\sharp_I$ is closed, this means that $\rho$ is supported by $\Wcal^\sharp_I$. Since $\Wcal_I^\sharp$ is actually compact and metrizable, hence a Polish space, and any finite Borel measure on a Polish space is regular (see Section \ref{measuretheory}), we also conclude that $\rho$ is regular.

Now, from Theorem \ref{vfsscarriedbyus}, $\rho$ is carried by $\Ucal_I$. Thus, $\rho$ is carried by the intersection $\Wcal^\sharp_I\cap \Ucal_I$, which is precisely $\Wcal_I$, and this completes the proof.
\end{proof}
\medskip

We now prove the following result, which shows that any invariant Vishik-Fursikov measure $\rho$ over an interval $I$ unbounded on the right can be extended to an invariant Vishik-Fursikov measure $\tilde\rho$ over the whole interval $\RR$. Since in this case $\Wcal=\Wcal_\RR = \Wcal_\RR^\sharp$, it follows from Theorem \ref{thmcarriervfmeasure}, that $\tilde\rho$ is a regular measure supported on the compact set $\Wcal$.

\begin{thm}
  \label{thmregularityofinvariantvfmeasures}
  Let $\rho$ be an invariant Vishik-Fursikov measure over an interval $I$ unbounded on the right. Then, there exists an invariant Vishik-Fursikov measure $\tilde\rho$ over $\RR$ such that $\rho=\Pi_I\tilde\rho$, with $\tilde\rho$ being a regular measure supported on the compact set $\Wcal$.
\end{thm}

\begin{proof} 
  If $I=\RR$, it follows from Theorem \ref{thmcarriervfmeasure} that $\rho$ is carried by $\Wcal_\RR=\Wcal$, and there is nothing else to prove, so the case of interest is when $I$ is bounded on the left. In any case, let $t'\in I$ and fix $R_1>R_0$, where $R_0$ is given by \eqref{defR0}. Consider the subset of $\Ccal(\Wcal)$ given by the functions $\Phi$ of the form $\Phi(\bv)=\varphi(\bv(t_1), \ldots, \bv(t_n))$, for arbitrary $\bv\in \Wcal$, where $n\in \NN$, $t_1<t_2<\ldots <t_n$, and $\varphi: (B_H(R_1)_\rw)^n \rightarrow \RR$ is a (weakly) continuous function.  Denote this subset by $\Scal$. It is not difficult to check that $\Scal$ satisfies the hypotheses of the Stone-Weierstrass Theorem, so that $\Scal$ is dense in $\Ccal(\Wcal)$ (recall, from Proposition \ref{propwcalinr0}, that $\Wcal$ is compact and metrizable). 
  
  Let now $\Phi$ belong to $\Scal$ and let $\varphi$ and $t_1<\ldots < t_n$ be associated with $\Phi$ as in the definition of $\Scal$. To each such $\Phi$, define $\Lambda(\Phi)$ by
  \[ \Lambda(\Phi) = \int_{\Ccal_\loc(I,H_\rw)} \varphi(\bu(t'),\bu(t'+t_2-t_1),\ldots, \bu(t'+t_n-t_1)) \;\rd\rho(\bu).
  \]
  According to Theorem \ref{thmcarriervfmeasure}, the measure $\rho$ is carried by $\Wcal_I$, which is contained in the compact set $\Wcal^\sharp_I$. in $\Ccal_\loc([t_0,\infty),B_H(R_1)_\rw)$. Thus, the integral defining $\Lambda(\Phi)$ above is well-defined. It is also straightforward to check that $\Lambda$ is a positive linear function on $\Scal$.   And since $\rho$ is a probability measure, it follows that $\Lambda$ is a bounded linear functional on $\Scal$. Hence, by the density of $\Scal$ in $\Wcal$, the functional $\Lambda$ can be extended to a continuous and positive linear functional on $\Wcal$, which we still denote by $\Lambda$. 

  By the Kakutani-Riesz Representation Theorem, $\Lambda$ can be represented by a measure $\tilde\rho$ on $\Wcal$. Then, we extend $\tilde\rho$ to all $\Ccal_\loc(\RR,H_\rw)$ (still denoting the extension by $\tilde\rho$) by setting $\tilde\rho(\Ecal)=\tilde\rho(\Ecal\cap \Wcal)$, for all Borel subsets $\Ecal$ of $\Ccal_\loc(\RR,H_\rw)$.

  By taking $\varphi=1$, we have $\Phi=1$, so that 
  \[ \tilde\rho(\Ccal_\loc(\RR,H_\rw))=\tilde\rho(\Wcal)=\Lambda(1) = \rho(\Ccal_\loc(I,H_\rw)) =1,
  \]
  which shows that $\tilde\rho$ is a probability measure. Thus, we conclude that $\tilde\rho$ is a Vishik-Fursikov measure over $\RR$.

  Let us now show that $\tilde\rho$ is invariant. Since $\rho$ is invariant, we see that for any $\Phi(\bv)=\varphi(\bv(t_1),\ldots,\bv(t_n))$ in $\Scal$, 
  \begin{multline*}
    \int_{\Wcal} \Phi(\sigma_\tau\bv)\;\rd\tilde\rho(\bv) = \int_{\Ccal_\loc(I,H_\rw)} \varphi(\sigma_\tau\bu(t'),\ldots,\sigma_\tau\bu(t'+t_n-t_1)) \;\rd\rho(\bu) \\
       = \int_{\Ccal_\loc(I,H_\rw)} \varphi(\bu(t'),\ldots,\bu(t'+t_n-t_1)) \;\rd\rho(\bu)= \int_\Wcal \Phi(\bv)\;\rd\tilde\rho(\bv),
  \end{multline*}
  for any $\tau \geq 0$. Since $\Scal$ is dense in $\Wcal$ and $\tilde\rho$ is carried by $\Wcal$, it follows that $\tilde\rho$ is invariant.   

  It remains to prove that $\rho=\Pi_I\tilde\rho$. Consider the subset $\Scal_I$ of $\Ccal(\Wcal_I^\sharp)$ made of functions $\Phi$ of the form  $\Phi(\bv)=\varphi(\bv(t_1), \ldots, \bv(t_n))$, for arbitrary $\bv\in \Wcal$, where $n\in \NN$, $t_1, \ldots, t_n\in I$, $t_1<t_2<\ldots, t_n$, $t_1\leq t'$, and $\varphi: (B_H(R_1)_\rw)^n \rightarrow \RR$ is a (weakly) continuous function. Using $\Wcal_I^\sharp$ is compact, it follows from the Stone-Weierstrass Theorem that $\Scal_I$ is dense in $\Ccal(\Wcal_I^\sharp)$. Moreover, it is straightforward to see that $\Phi\circ\Pi_I$ belongs to $\Scal$ for $\Phi$ in $\Scal_I$. Then, using the invariance of $\rho$, the condition $t_1 \leq t'$, and the fact that $\rho$ is carried by $\Wcal_I$, which is included in $\Wcal_I^\sharp$, we find that
  \begin{multline*}
    \int_\Wcal \Phi(\Pi_I\bv)\;\rd\tilde\rho(\bv) = \int_{\Wcal_I} \varphi(\bu(t'),\bu(t'+t_2-t_1),\ldots, \bu(t'+t_n-t_1)) \;\rd\rho(\bu) \\
      = \int_{\Wcal_I} \varphi(\sigma_{t'-t_1}\bu(t_1),\sigma_{t'-t_1}\bu(t_2),\ldots,\sigma_{t'-t_1}\bu(t_n)) \;\rd\rho(\bu) \\  
      = \int_{\Wcal_I} \varphi(\bu(t_1),\bu(t_2),\ldots,\bu(t_n)) \;\rd\rho(\bu)    = \int_{\Wcal_I} \Phi(\bu)\;\rd\rho(\bu).
  \end{multline*}
  Since now $\Scal_I$ is dense in $\Wcal_I^\sharp$ it follows that
  \[ \int_\Wcal \Phi(\Pi_I\bv)\;\rd\tilde\rho(\bv) = \int_{\Wcal_I^\sharp} \Phi(\bu)\;\rd\rho(\bu),
  \]
  for all $\Phi$ in $\Wcal_I^\sharp$. Since $\rho$ is carried by $\Wcal_I$, which is contained in $\Wcal_I^\sharp$, and $\tilde\rho$ is carried by $\Wcal$, this means that $\rho=\Pi_I\tilde\rho$.
  
  The property that $\tilde\rho$ is a regular measure supported on the compact set $\Wcal$ follows from Theorem \ref{thmcarriervfmeasure} and the fact that, in this case,, $\Wcal=\Wcal_\RR = \Wcal_\RR^\sharp$.
\end{proof}
\medskip

\begin{rmk}
  \label{rmkergodicityonwcal}
  Theorem \ref{thmregularityofinvariantvfmeasures} allows us to restrict the discussion of invariant Vishik-Fursikov measures to the space $\Wcal$. Note, moreover, that, since $\Wcal$ is compact and metrizable, several results of ergodic theory apply, such as those about the existence of ergodic invariant measures; the decomposition of invariant measures into ergodic parts; and that ergodic invariant measures are extremal points of the set of invariant measures on $\Wcal$ (see e.g.  \cite[Section 9.3]{pollicottyuri} and \cite[Section 6]{walters1982}).   
\end{rmk}

\subsection{Properties of Vishik-Fursikov stationary statistical solutions}

The definition of a Vishik-Fursikov stationary statistical solution $\rho_0$ is that it is the projection $\rho_0=\Pi_t\rho$, at an arbitrary time $t\in I$, of an invariant Vishik-Fursikov measure $\rho$ over a certain interval $I$ which is unbounded on the right. In principle, however, we do not have any control on the interval $I$, i.e. given $\rho_0$, there exists some interval $I$ and some invariant Vishik-Fursikov measure $\rho$ over $I$ for which the relation $\rho_0=\Pi_t\rho$, $\forall t\in I$, holds. Fortunately, thanks to the result in Theorem \ref{thmregularityofinvariantvfmeasures}, we can actually choose the interval $I$ that we would like to work with. We state this result in the following form.

\begin{cor}
  \label{regrho0}
  Let $\rho_0$ be a Vishik-Fursikov stationary statistical solution on $H$. Then, given any interval $I$ unbounded on the right, there exists an invariant Vishik-Fursikov measure $\rho$ over $I$ such that $\rho_0=\Pi_t\rho$, for any $t\in I$. The measure $\rho$ is a regular measure supported on the compact set $\Wcal_I^\sharp$ and carried by $\Wcal_I$.
\end{cor}

We deduce, moreover, that Vishik-Fursikov stationary statistical solutions are carried by the weak global attractor $\Acal_\rw = \Pi_0\Wcal$ defined in \eqref{awdef}.

\begin{thm}
  \label{carrierofrho0inaw}
  Let $\rho_0$ be a Vishik-Fursikov stationary statistical solution $H$. Then, $\rho_0$ is
  carried by the weak global attractor $\Acal_\rw=\Pi_0\Wcal$, i.e. $\rho_0(\Acal_\rw)=1$. Since $\Acal_\rw$ is compact, it follows that $\supp\rho_0 \subset \Acal_\rw$.
\end{thm}

\begin{proof}
  From Corollary \ref{regrho0}, there exists an invariant Vishik-Fursikov measure
  $\rho$ over $\RR$ such that $\Pi_0\rho=\rho_0$, and with $\rho$ carried
  by $\Wcal$. Since $\Pi_0^{-1}(\Pi_0\Wcal) \cap \Wcal = \Wcal$, it follows that
  $\rho_0(\Pi_0\Wcal)=\rho(\Wcal)=1$, which proves the result.
\end{proof}
\medskip

We conclude this section with estimates related to the support of a Vishik-Fursikov stationary statistical solution $\rho_0$ on $H$. We know already, from Theorem \ref{carrierofrho0inaw}, that $\rho_0$ is carried by $\Pi_0\Wcal$. Since $\Pi_0\Wcal$ in included in the ball of radius $R_0$ in $H$, where $R_0$ is given by \eqref{defR0} (see Proposition \ref{propwcalinr0}), this yields in particular that $\rho_0$ is carried by this ball. 
We obtain further estimates from the uniform time-average bounds for the solutions in $\Wcal$ given in Lemmas \ref{lemwcalh1ave} and \ref{lemwcaldaave}.
From these time-average estimates in the support $\Wcal$ of $\rho_0$ we 
deduce the following estimates for $\rho_0$.
\begin{thm}
  \label{corboundsrho0}
  Let $\rho_0$ be a Vishik-Fursikov stationary statistical solution $H$. Then,
  $\rho_0$ is carried by $D(A)$ and we have the following estimates
  \begin{equation} 
    \label{estimaterho0inh1}
    \int \|\bu\|_{H^1}^2 \;\rd\rho_0(\bu) \leq \lambda_1^{1/2}\nu^2 G^2,
  \end{equation}
  \begin{equation}
    \label{estimaterho0inda}
    \int |A\bu|_{L^2}^{2/3} \;\rd\rho_0(\bu) \leq c_3\lambda_1^{1/2}\nu^{2/3}G^2,
  \end{equation} 
  and
  \begin{equation}
    \label{estimaterho0inlinfty}
    \int |\bu|_{L^\infty} \;\rd\rho_0(\bu) \leq c_1 c_3^{3/4}\lambda_1^{1/2}\nu G^2.
  \end{equation}
\end{thm}

\begin{proof}
  From Corollary \ref{regrho0}, there exists an invariant Vishik-Fursikov measure $\rho$ over $[0,\infty)$ such that $\Pi_0\rho=\rho_0$, and with $\rho$ carried by $\Wcal_{[0,\infty)}$. Since $\Wcal_{[0,\infty)}=\Wcal\cap\Ucal_{[0,\infty)}$ the estimate \eqref{ineqlemwcalh1ave} in  Lemma \ref{lemwcalh1ave} holds with $t'=0$. Hence, we can write
  \[ \frac{1}{T} \int_0^T \|\bu(t)\|_{H^1}^2 \;\rd t  \leq \lambda_1^{1/2}\nu^2 G^2\left(1 + \frac{1}{\nu\lambda_1 T}\right), 
  \] 
  for all $T>0$ and all $\bu\in \Wcal_{[0,\infty)}$.  
  Since $\rho$ is carried by $\Wcal_{[0,\infty)}$ we integrate this estimate in 
  $\bu$ to find that
  \[ \int_{\Wcal_{[0,\infty)}} \frac{1}{T} \int_0^T \|\bu(t)\|_{H^1}^2 \;\rd t \;\rd\rho(\bu)
       \leq \lambda_1^{1/2}\nu^2 G^2\left(1 + \frac{1}{\nu\lambda_1 T}\right), 
  \]
  for all $T>0$. Using Fubini we rewrite this estimate as
  \[  \frac{1}{T} \int_0^T \int_{\Wcal_{[0,\infty)}} \|\bu(t)\|_{H^1}^2 \;\rd\rho(\bu) \;\rd t 
       \leq \lambda_1^{1/2}\nu^2 G^2\left(1 + \frac{1}{\nu\lambda_1 T}\right).
  \] 
  This can also be written as
  \[ \frac{1}{T} \int_0^T \int_{\Wcal_{[0,\infty)}} \|\Pi_0\sigma_t\bu\|_{H^1}^2 \;\rd\rho(\bu) \;\rd t 
       \leq \lambda_1^{1/2}\nu^2 G^2\left(1 + \frac{1}{\nu\lambda_1 T}\right).
  \]
  Since $\rho$ is invariant with respect to $\{\sigma_t\}_{t \geq 0}$ the integrand in $t$ is independent of $t$, and we find that
  \[ \int_{\Wcal_{[0,\infty)}} \|\Pi_0\bu\|_{H^1}^2 \;\rd\rho(\bu) 
        \leq \lambda_1^{1/2}\nu^2 G^2\left(1 + \frac{1}{\nu\lambda_1 T}\right).
  \]
  Since $T>0$ is arbitrary, we let $T\rightarrow \infty$ to find that
  \[ \int_{\Wcal_{[0,\infty)}} \|\Pi_0\bu\|_{H^1}^2 \;\rd\rho(\bu) 
        \leq \lambda_1^{1/2}\nu^2 G^2.
  \]
  Since $\rho_0=\Pi_0\rho$, this means that
  \[ \int_H \|\bu\|_{H^1}^2\;\rho_0(\bu) \leq \lambda_1^{1/2}\nu^2 G^2,
  \]
  which proves the inequality \eqref{estimaterho0inh1}. 
  The proof for \eqref{estimaterho0inda} is analogous to this one and follows from the estimates in Lemma \ref{lemwcaldaave}, while \eqref{estimaterho0inlinfty} follows from the two inequalities \eqref{estimaterho0inh1} and \eqref{estimaterho0inda} and the use of Agmon's inequality \eqref{agmonineq}.
\end{proof}

\subsection{Properties of time-average stationary statistical solutions}
\label{timeaveragevishikfursikovsec}

We start this section by showing that the two notions of time-average stationary statistical solution and time-average Vishik-Fursikov stationary statistical solution given in Definitions \ref{deftimeavesss} and \ref{defvftimeaverageVFsss} are in fact equivalent. In particular, this means that all the previous results for Vishik-Fursikov stationary statistical solutions apply to time-average stationary statistical solutions (see Remark \ref{rmktimeaveragesss}). 

\begin{thm}
  \label{thmtimeaveFPareVF}
  Given a time-average stationary statistical solution $\mu$ on $H$
  associated with a generalized
  limit $\Lim$ and a weak solution $\bu=\bu(t)$, $t\geq t_0$, $t_0\in\RR$,
  there exists a time-average invariant Vishik-Fursikov measure $\rho$
  on $\Ccal_\loc([t_0,\infty),H_\rw)$, obtained with the same $\Lim$ and the same weak solution $\bu$,
  for which $\Pi_t\rho=\mu$ for any $t\geq t_0$.
\end{thm}

\begin{proof} 
Let $\rho$ be a time-average Vishik-Fursikov measure 
on $\Ccal_\loc([t_0,\infty),H_\rw)$ associated with $\bu$ and $\Lim$. 
Let $R>0$ be sufficiently large so that the orbit of the solution $\bu$
belongs to $B_H(R)$. 

Given a function $\varphi$ in $\Ccal(B_H(R)_\rw)$, the function
$\bv\mapsto \tilde\varphi(\bv) = \varphi(\Pi_t(\bv))=\varphi(\bv(t))$ 
is continuous on $\Ccal_\loc([t_0,\infty),B_H(R)_\rw)$, for any $t\geq t_0$. For this 
function we find, using the definition
of the time-average Vishik-Fursikov measure $\rho$, property \eqref{genliminvariance} of the 
generalized limit, and the
definition of $\mu$ as a time-average stationary statistical solution on $H$, that
\begin{multline*}
   \int_{\Ccal_\loc([t_0,\infty),B_H(R)_\rw)} \varphi\circ\Pi_t(\bv) \;\rd\rho(\bv)
    = \Lim \frac{1}{T}\int_0^T
      \varphi\circ\Pi_t(\sigma_{\tau}\bu))\;\rd \tau \\
    = \Lim \frac{1}{T}\int_0^T \varphi(\bu(t+\tau)))\;\rd \tau
    = \Lim \frac{1}{T}\int_0^T \varphi(\bu(t_0+\tau)))\;\rd \tau  \\
    = \int_H \varphi(\bv) \dmu(\bv).
\end{multline*}
This proves the claim that $\mu = \Pi_t \rho$, for arbitrary $t\geq t_0$. 
\end{proof}
\medskip

From this result relating time-average stationary statistical solutions
with invariant time-average Vishik-Fursikov measures
and the result on the accretion property for the latter we obtain a simpler
proof of the accretion property for the time-average stationary statistical 
solutions in the phase space (\emph{cf.} \cite{foiastemam75,fmrt2001a}) 
as we will show in the next section.

Concerning the support of a time-average invariant measure
$\rho$ associated with a weak 
solution $\bu$ on an interval $[t_0,\infty)$, we know already from 
Theorem \ref{thmcarriervfmeasure} that such a measure
is carried by $\Wcal_{[t_0,\infty)}$. In fact, one can be more
precise and show that such a measure $\rho$ is carried by 
the $\omega$-limit set of the associated weak solution 
under the translation semigroup $\{\sigma_t\}_{t\geq 0}$, defined in \eqref{sigmatdef}.
This is known in the context of time-average stationary statistical solutions
on $H$, as proved in \cite{foiastemam85}. The $\omega$-limit set in this
case is defined as in \eqref{omegalimittrajectory}
The idea of the proof is the same as in \cite{foiastemam85}, and
for this reason we only state the result here:
\begin{prop}
  Let $\rho$ be the time-average invariant Vishik-Fursikov measure 
  on $\Ccal_\loc([t_0,\infty),H_\rw)$, with $t_0\in \RR$, associated with 
  a generalized limit $\Lim$ and a weak solution $\bu=\bu(t)$ on $t\geq t_0$.
  Then, $\rho(\omega_{\Ccal_\loc([t_0,\infty),H_\rw)}(\bu)) =1$. Since this set is compact, we find that $\supp\rho \subset \omega_{\Ccal_\loc([t_0,\infty),H_\rw)}(\bu)$.
\end{prop}
\medskip

The corresponding measure $\mu$ on $H$ which is the projection of
$\rho$ at an arbitrary time $t\geq t_0$ is carried
by the projection of this $\omega$-limit set, which coincides with the
$\omega$-limit set of $\bu$ for the
Navier-Stokes equations in the weak topology of $H$, given by
\eqref{omegalimitphase}. In this way we recover the corresponding well-known 
result (see \cite{foiastemam85}) for time-average stationary statistical
solutions on $H$.

\begin{cor}
  \label{timeavecarriedbyomegalimit}
  Let $\mu$ be a time-average stationary statistical solution on $H$
  associated with a weak solution $\bu$. Then $\mu$ is a 
  time-average Vishik-Fursikov statistical solution associated with this same $\bu$ and
  $\mu$ is carried by $\omega_{H_\rw}(\bu)$. Since this set is compact, we find that $\supp\mu \subset \omega_{H_\rw}(\bu)$.
\end{cor}

\begin{rmk}
  \label{rmktimeaveragesss}
  Since a time-average stationary statistical solution is a Vishik-Fursikov stationary statistical solution (Theorem \ref{thmtimeaveFPareVF}), the Corollary \ref{regrho0}, Theorem \ref{carrierofrho0inaw} and the bounds given in Theorem \ref{corboundsrho0} apply to such a solution.
\end{rmk}

\section{Local regularity of carriers}
\label{seclocalreg}
  
In this section we present local regularity results for Vishik-Fursikov stationary statistical solutions and Vishik-Fursikov measures. As we mentioned in the Introduction, this is motivated by the \emph{Prodi invariance conjecture}, which states that the support of a time-average stationary statistical solution should be more regular in some sense, belonging to a space in which the solutions are unique and strong globally in time. It is a kind of asymptotic regularity result in average, for the solutions of the 3D Navier-Stokes equations. Here, we prove a partial result in this direction, namely that any Vishik-Fursikov stationary statistical solution is carried by a set in which the solutions are locally strong solutions (Theorem \ref{vfssscarriedbyaregprime}), with a similar result for a Vishik-Fursikov measure (Theorem \ref{invariantvfcarriedbywcalregprime}). 

\subsection{Vishik-Fursikov stationary statistical solutions are carried by locally regular solutions}
  
From the characterization of $\Acal_\rw\setminus\Acal_\reg'$ with the estimate on the rate
of blow up of the solutions, we obtain the following important result.

\begin{thm}
  \label{vfssscarriedbyaregprime}
    Any Vishik-Fursikov stationary statistical solution $\rho_0$ is carried by $\Acal_\reg'$, i.e. $\rho_0(\Acal_\reg')=1$.
\end{thm}

\begin{proof}
Since $\rho_0$ is a probability measure supported on $\Acal_\rw$ (Theorem \ref{carrierofrho0inaw}), it suffices to show that $\rho_0(\Acal_\rw\setminus\Acal_\reg')=0$. From Corollary \ref{regrho0} there exists an invariant Vishik-Fursikov measure $\rho$ over $\RR$ such that $\rho_0=\Pi_t\rho$, for any $t\in \RR$, and $\rho$ is carried by $\Wcal$. From the characterization of $\Acal_\rw\setminus\Acal_\reg'$ in \eqref{charactaregminusaregprimeeq}, we have that
  \[ \Pi_0^{-1}(\Acal_\rw\setminus\Acal_\reg'))\cap\Wcal
        \subset \Pi_t^{-1}(H\setminus B_V(\Gamma(t)^{1/2})), \qquad \forall t< 0.
  \]
  Then,
  \begin{multline*}
    \rho_0(\Acal_\rw\setminus\Acal_\reg') = \rho(\Pi_0^{-1}(\Acal_\rw\setminus\Acal_\reg'))
      \leq \rho(\Pi_t^{-1}(H\setminus B_V(\Gamma(t)^{1/2}))) \\ 
      = \rho_0(H\setminus B_V(\Gamma(t)^{1/2}), \qquad \forall t< 0.
  \end{multline*}
  Since $t<0$ is arbitrary and $\Gamma(t) \rightarrow \infty$ as $t\rightarrow 0^-$, if follows from the continuity property of measures and the fact that $\rho_0$ is carried by $V$ (see \eqref{ssscarriedbyV}) that
  \[ \rho_0(\Acal_\rw\setminus \Acal_\reg')
     \leq \rho_0(H\setminus B_V(\Gamma(t)^{1/2})) \rightarrow 
        \rho_0(H\setminus V) = 0,
  \]
  as $t\rightarrow 0^-$. Thus,
  \[  \rho_0(\Acal_\rw\setminus \Acal_\reg') = 0.
  \]
\end{proof}
\medskip

\subsection{Invariant Vishik-Fursikov measures are carried by locally regular solutions}

A similar local regularity result can be given for invariant Vishik-Fursikov measures.
\begin{thm}
  \label{invariantvfcarriedbywcalregprime}
Any Vishik-Furiskov invariant measure $\rho$ on $\RR$ is carried by $\Wcal_\reg'$, i.e. $\rho(\Wcal_\reg')=1$. 
\end{thm}

\begin{proof}
The proof is similar to that of Theorem \ref{vfssscarriedbyaregprime}, or simply use the characterization \eqref{wcalregprimeandacalregprime} of $\Wcal_\reg'$ and the fact that $\Acal_\rw = \Pi_0\Wcal$ to write
\[ \Wcal \setminus \Wcal_\reg' = \Wcal \setminus (\Wcal \cap \Pi_0^{-1}\Acal_\reg') = \Pi_0^{-1}(\Acal_\rw \setminus \Acal_\reg') \cap \Wcal, 
\]
so that, by Theorem \ref{vfssscarriedbyaregprime} and the fact that $\rho$ is carried by $\Wcal$ (Theorem \ref{thmcarriervfmeasure}),
\[ \rho(\Wcal\setminus \Wcal_\reg') = \rho(\Pi_0^{-1}(\Acal_\rw \setminus \Acal_\reg')) = \rho_0(\Acal_\rw\setminus \Acal_\reg') = 0.
\]
\end{proof}

From the relation \eqref{wcalregprimeasunionwcalregshortprime} and Theorem \ref{invariantvfcarriedbywcalregprime}, we see that, for any invariant Vishik-Fursikov measure $\rho$,
\[ 1 = \rho(\Wcal_\reg') = \rho(\bigcup_{\tau>0} \Wcal_{\reg,\tau}') = \rho(\bigcup_{n\in\NN} \Wcal_{\reg,1/n}') = \lim_{n\rightarrow \infty} \rho(\Wcal_{\reg,1/n}') = \lim_{\tau \rightarrow 0} \rho(\Wcal_{\reg,\tau}').
\]
Hence, the measure of $\Wcal_{\reg,\tau}'$ approaches $1$, as $\tau$ goes to zero. This result can actually be made more precise using the characterization \eqref{charactwregtauprimecomplement} of this set, as follows.


\begin{thm} 
  \label{thmestimateforcomplementwcalregtauprime}
  Let $\rho$ be an invariant Vishik-Fursikov measure over $\RR$. Then, 
  \begin{equation}
    \label{estimateforcomplementwcalregtauprime}
    \rho(\Wcal\setminus\Wcal_{\reg,\tau}') 
       \leq \frac{4c_4\lambda_1^{1/2}\nu^{1/2}\tau^{1/2}G}{
          1-2c_4\lambda_1^{1/2}\nu^{1/2}\tau^{1/2}G^{2/3}},
  \end{equation}
  for any $\tau$ such that
  \begin{equation}
    \label{conditionontau}
    0<\tau < \frac{1}{4c_4^2\lambda_1\nu G^{4/3}}.
  \end{equation} 
\end{thm}

\begin{proof}
Let $\tau$ satisfy \eqref{conditionontau} and suppose $\bu\in \Wcal \setminus \Wcal_{\reg,\tau}'$, so that $\bu$ blows up at some point $\beta$ in the interval $(-\tau,\tau)$. Then, for such $\beta$, using the characterization \eqref{charactwregtauprimecomplement}, we see that
\[\Gamma(t-\beta) \leq \|\bu(t)\|_{H^1}^2, \qquad \forall t<\beta.
\]
For $\beta-\tau \leq t < \beta$, we have
\[ |t-\beta| = -t+\beta \leq \tau = |-\tau|.
\]
Thus,
\[ \Gamma(t-\beta) = \frac{\nu^{3/2}}{2c_4|t-\beta|^{1/2}} - \nu^{2/3}|\bbf|_{L^2}^{2/3}
      \geq \Gamma(-\tau) = \frac{\nu^{3/2}}{2c_4\tau^{1/2}} - \lambda_1^{1/2}\nu^2G^{2/3}.
\]  
From the condition \eqref{conditionontau}, we see that $\Gamma(-\tau)>0$, and we can write
\begin{align*}
  1 & = \frac{\Gamma(-\tau)}{\Gamma(-\tau)} 
        = \frac{1}{\Gamma(-\tau)}\frac{1}{\tau}\int_{\beta-\tau}^\beta \Gamma(-\tau)\;\rd t \\
      & \leq \frac{1}{\Gamma(-\tau)}\frac{1}{\tau}\int_{\beta-\tau}^\beta \Gamma(t-\beta)\;\rd t \\
      & \leq \frac{1}{\Gamma(-\tau)}\frac{1}{\tau}\int_{\beta-\tau}^\beta \|\bu(t)\|_{H^1}^2\;\rd t \\
      & \leq \frac{1}{\Gamma(-\tau)}\frac{1}{\tau}\int_{-2\tau}^0 \|\bu(t)\|_{H^1}^2\;\rd t.
\end{align*}
Thus, we have that
\begin{equation}
  \label{estimateoncomplementwcalregtauprime}
  1 \leq \frac{1}{\Gamma(-\tau)}\frac{1}{\tau}\int_{-2\tau}^0 \|\bu(t)\|_{H^1}^2\;\rd t,
    \quad \forall \bu\in  \Wcal\setminus\Wcal_{\reg,\tau}'.
\end{equation}
Let now $\rho$ be an invariant Vishik-Fursikov measure on $\Wcal$. From \eqref{estimateoncomplementwcalregtauprime} we have that
\[ \rho(\Wcal\setminus\Wcal_{\reg,\tau}') 
        \leq \frac{1}{\Gamma(-\tau)}\frac{1}{\tau}\int_{\Wcal\setminus\Wcal_{\reg,\tau}'}
             \int_{-2\tau}^0 \|\bu(t)\|_{H^1}^2\;\rd t\;\rd\rho(\bu).
\]
Using Fubini's theorem, extending the integral to all $\Wcal$, and using the fact that $\rho$ is invariant on $\Wcal$, we find that
\[ \rho(\Wcal\setminus\Wcal_{\reg,\tau}') 
     \leq \frac{1}{\Gamma(-\tau)}\frac{1}{\tau} \int_{-2\tau}^0 \int_\Wcal 
       \|\bu(t)\|_{H^1}^2\;\rd\rho(\bu)\;\rd t
     \leq \frac{2}{\Gamma(-\tau)} \int_\Wcal \|\bu(0)\|_{H^1}^2\;\rd\rho(\bu).
\] 
Using now the estimate \eqref{estimaterho0inh1} for $\rho_0 = \Pi_0\rho$ we obtain
\begin{equation}
  \rho(\Wcal\setminus\Wcal_{\reg,\tau}') \leq \frac{2}{\Gamma(-\tau)} \lambda_1^{1/2}\nu^2 G^2      = \frac{4c_4\lambda_1^{1/2}\nu^{1/2}\tau^{1/2}G}{1-2c_4\lambda_1^{1/2}\nu^{1/2}\tau^{1/2}G^{2/3}},
\end{equation}
which proves \eqref{estimateforcomplementwcalregtauprime}.
\end{proof}


\begin{rmk}
  Perusing the proof of Theorem \ref{thmestimateforcomplementwcalregtauprime} we see that the same estimate is valid in fact for the complement in $\Wcal$ of the larger set
  \[ \Wcal_{\reg,\tau^-}' = \left\{\bu\in \Wcal; \;\bu \text{ is a strong solution on } (-\tau,0]\right\}.
  \]
\end{rmk}

\begin{rmk}
  \label{wcalreginftyergodic}
At the moment, there is no estimate for the set $\Wcal_{\reg,\infty}'$ of global strong solutions. However, it is clear that both sets $\Wcal_{\reg,\infty}'$ and $\Wcal\setminus\Wcal_{\reg,\infty}'$ are invariant by the translation semigroup $\sigma_\tau$. Therefore, in the case that an invariant Vishik-Fursikov measure $\rho$ is ergodic, we must have either $\rho(\Wcal_{\reg,\infty}')=1$ or $\rho(\Wcal_{\reg,\infty}')=0$.
\end{rmk}



\section{Accretion properties of statistical solutions}
\label{secaccretion}

For an invariant measure $\mu$ of a well-defined semigroup $\{S(t)\}_{t\geq 0}$ on a given phase space, it is immediate to deduce that $\mu(S(t)E) = \mu(S(t)^{-1}S(t)E) \geq \mu(E)$, for any $t\geq 0$ and any measurable subset $E$. If $\mu$ is an arbitrary measure but still has the property that $\mu(S(t)E) \geq \mu(E)$, for any $t\geq 0$ and any measurable subset $E$, then this measure is called \emph{accretive.} In the case of the three dimensional Navier-Stokes equations, however, a semigroup is not known to be available and such a notion does not make sense in this way. But we may still define a notion of accretion based on the multivalued evolution maps defined in Section \ref{secmultvaluedmap} (see Definition \ref{defaccretive}). We then prove that any Vishik-Fursikov stationary statistical solution is accretive in this sense (see Theorem \ref{accretionvfss}), extending, with a much simpler proof, a result previously known only for time-average stationary statistical solutions \cite{foiastemam75}.

The accretion property allows us to apply to these measures the recurrence results given in Section \ref{sectionrecurrence}.

\subsection{Accretive measures}

Recall from Lemma \ref{sigmaomegauniversal} that if $E$ is Borel then $\Sigma_t E$ is universally measurable. Hence, the following definition makes sense.

\begin{defs}
  \label{defaccretive}
  A Borel probability measure $\mu$ on the phase space $H$ is said
  to be accretive with respect to the family $\{\Sigma_t\}_{t\geq 0}$ if
  \begin{equation}
    \label{eqaccretion}
    \mu(\Sigma_t E) \geq \mu(E), \quad \forall t\geq 0,
  \end{equation}
  for all Borel subsets $E$ of $H$.
\end{defs}

In \cite{foiastemam75,fmrt2001a} it has been proved that any time-average
stationary statistical solution for the three-dimensional Navier-Stokes
equations is accretive in the sense above (see also \cite{foias73}):

\begin{prop}[{\cite[Theorem 3.4]{foiastemam75}}]
  Any time-average stationary statistical solution on $H$ is accretive
  with respect to $\{\Sigma_t\}_{t\geq 0}$.
\end{prop}

\begin{rmk}
  There is another important related concept which is that of a \emph{semi-invariant}
  Borel probability measure $\mu$ on $H$, for which $\mu(S_R(t)^{-1}E) \leq \mu(E)$, 
  for all $t\geq 0$ and all Borel sets $E$ in $H$, where $S_R(t)$ is the ``strong-solution'' 
  operator defined on a subset $D_R(t)$ of $V$ of initial conditions $\bu_0$ for which there 
  is a unique strong solution $S_R(\cdot)\bu_0$ on $[0,t]$. 
  (See \cite[pages 28, 33, and 37]{foias73}, 
  where such a measure was called accretive, but it is different from the notion of accretive 
  measure currently used).
\end{rmk}

Thanks to the regularity of a Borel probability measure on $H$, the 
accretion property extends to arbitrary measurable sets:
\begin{lem}
  \label{extendedaccretion}
  Let $\mu$ be an accretive Borel probability measure on $H$
  with respect to the family $\{\Sigma_t\}_{t\geq 0}$. 
  Let $t\geq 0$ and suppose $E$ is a $\mu$-measurable set 
  such that $\Sigma_t E$ is also $\mu$-measurable.
  Then $\mu(\Sigma_t E) \geq \mu(E)$.
\end{lem}

\begin{proof} 
We approximate the measure of $E$ from below by the measure
of a compact set $K\subset E$. Since $K\subset E$, we 
have $\Sigma_t K\subset \Sigma_t E$. Then, applying also the accretion 
property to $K$, we find
\[ \mu(\Sigma_t E) \geq \mu(\Sigma_t K) \geq \mu(K).
\]
Taking the supremum in $K\subset E$ compact, we obtain
\[ \mu(\Sigma_t E) \geq \sup \{\mu(K); \; K\subset E, \;K \text{ compact in } H \} = \mu(E),
\]
which completes the proof.
\end{proof}
\medskip

Accretive measures satisfy a strengthened accretion property according
to the following result.
\begin{lem}
  \label{lemstrengthenedaccretion}
  Let $\mu$ be an accretive measure on $H$ for $\{\Sigma_t\}_{t\geq 0}$.
  Then,
  \begin{equation}
    \mu(\Sigma_t E) \geq \mu(\Sigma_s E),
    \quad \forall t\geq s\geq 0,
  \end{equation}
  for all Borel subsets $E$ of $H$.
\qed
\end{lem}
\medskip

\begin{proof}
This result is based on Lemma  \ref{measurabilitysigmats}, which 
guarantees that $\Sigma_s E$ and $\Sigma_t\Sigma_s E$
are measurable, and on Lemma \ref{extendedaccretion}, which extends the 
accretion property to such sets. Then, for $t\geq s\geq 0$, we have
$\Sigma_t E=\Sigma_{t-s+s} E\supset \Sigma_{t-s}\Sigma_s E$,
and using the accretion property
starting from $\Sigma_s E$, we find $\mu(\Sigma_t E)\geq
\mu(\Sigma_{t-s}\Sigma_s E)\geq \mu(\Sigma_s E),$
which gives the strengthened form of accretion.
\end{proof}

The result of Lemma \ref{lemstrengthenedaccretion} was announced in \cite{foias74b} and given in \cite{foiasprodi76} (see also \cite{fmrt2001a}). However, the measurability of these sets
was not completely proved. 

\subsection{Accretion property for Vishik-Fursikov statistical solutions}

First, we prove a form of accretion for time-dependent Vishik-Fursikov statistical solutions. A related result was given in \cite{foiasprodi76}, in which it was shown that, given any Borel probability measure $\mu_0$ on $H$ with finite mean kinetic energy, there exists at least one time-dependent statistical solution $\{\mu_t\}_{t\geq 0}$, in the original sense of Foias and Prodi \cite{foias72,foiasprodi76}, that satisfies the accretion property $\mu_t(\Sigma_t E) \geq \mu_0(E)$, for any Borel subset $E\subset H$ and any $t\geq 0$. Here, we show that this property is true for any time-dependent Vishik-Fursikov statistical solution.

\begin{thm}
  \label{accretionvfss}
  Let $\{\rho_t\}_{t\geq 0}$ be a Vishik-Fursikov statistical solution over $[0,\infty)$.
  Then $\{\rho_t\}_{t\geq 0}$ satisfies
  \[ \rho_t(\Sigma_t E) \geq \rho_0(E),
  \]
  for all Borel sets $E$ in $H$ and all $t\geq 0$.
\end{thm}

\begin{proof} 
  Let $\rho$ be the associated Vishik-Fursikov measure over $[0,\infty)$.
  Since $E$ is Borel, Lemma \ref{sigmaomegauniversal} yields that $\Sigma_t E$
  is universally measurable, hence, $\rho_t$-measurable. 
  
  Since $\rho_t$ is a regular Borel measure on $H$ we have that
  \[ \rho_t(\Sigma_t E)=\inf\left\{ \rho_t(O); \;O\supset \Sigma_t E, \;O \text{ open in } H\right\}.
  \]
  Consider then $O\supset \Sigma_t E$ open in $H$. Such $O$ is a Borel set in $H_\rw$.
  Since $\Pi_t:\Ccal_\loc([0,\infty),H_\rw)\rightarrow H_\rw$ is continuous, for $t\geq 0$, 
  the set $\Pi_t^{-1}O$ is Borel in the space $\Ccal_\loc([0,\infty),H_\rw)$ and $\rho_t(O)=\rho(\Pi_t^{-1}O)$.
  From Theorem \ref{vfsscarriedbyus}, $\rho$ is carried by $\Ucal_{[0,\infty)}$, so that we 
  write
  \[ \rho_t(O) = \rho\left((\Pi_t^{-1}O)\cap \Ucal_{[0,\infty)}\right).
  \]
  Since $O\supset \Sigma_t E$, this means that
  \[ \rho_t(O) \geq \rho(\Pi_t^{-1}\Sigma_tE)\cap\Ucal_{[0,\infty)}).
  \] 
  Notice, now, that
  \begin{equation}
    \label{PitSigmatsupsetPi0}
     (\Pi_t^{-1}\Sigma_tE)\cap\Ucal_{[0,\infty)} \supset (\Pi_0^{-1}E)\cap\Ucal_{[0,\infty)},
  \end{equation}
  which implies that
  \[ \rho_t(O) \geq \rho\left((\Pi_0^{-1}E)\cap\Ucal_{[0,\infty)}\right)=\rho_0(E).
  \]
  Question?:
  \[ (\Pi_t^{-1}\Sigma_tE)\cap\Ucal_{[0,\infty)} \supset (\Pi_s^{-1}\Sigma_sE)\cap\Ucal_{[0,\infty)}, \quad \forall t\geq s\geq 0?
  \]
  Thus, from the regularity property of $\rho_t$, we take the infimum of $\rho_t(O)$ over $O\supset \Sigma_t E$ to deduce that
  \[ \rho_t(\Sigma_t E) \geq \rho_0(E),
  \]
  which proves the desired accretion property for $\{\rho_t\}_{t\geq 0}$.
\end{proof}
\medskip

A simple time-translation of the previous result yields the following corollary.
\begin{cor}
  Let $t_0\in \RR$ and let $\{\rho_t\}_{t\geq t_0}$ be a Vishik-Fursikov statistical solution over 
  $[t_0,\infty)$. Then $\{\rho_t\}_{t\geq t_0}$ satisfies
  \[ \rho_t(\Sigma_{t-t_0} E) \geq \rho_{t_0}(E),
  \]
  for all Borel sets $E$ in $H$ and all $t\geq t_0$.
  \qed
\end{cor}
\medskip

The case of a Vishik-Fursikov stationary statistical solution is a particular
case of the previous result and yields the desired accretion result.

\begin{cor}
  \label{rho0accretive}
  Let $\rho_0$ be a Vishik-Fursikov stationary statistical solution on $H$. Then
  $\rho_0$ is accretive for $\{\Sigma_t\}_{t\geq 0}$.
\end{cor}

Since we have seen that any time-average stationary statistical solution
is a Vishik-Fursikov stationary solution the above result gives a simplified
proof of the accretion for such solutions. We state this result as follows.

\begin{cor}
  \label{timeaverageaccretiveagain}
  Any time-average stationary statistical solution on $H$ in accretive for 
  $\{\Sigma_t\}_{t\geq 0}$
\end{cor}

\begin{rmk}
  In the proof of Theorem \ref{accretionvfss}, if, besides \eqref{PitSigmatsupsetPi0}, we had the more general relation $(\Pi_t^{-1}\Sigma_tE)\cap\Ucal_{[0,\infty)} \supset (\Pi_s^{-1}\Sigma_sE)\cap\Ucal_{[0,\infty)}$, for any $t\geq s\geq 0$, then we would find that $\rho_t(\Sigma_t E) \geq \rho_s(\Sigma_s E)$, showing that $\rho_t(\Sigma_t E)$ is monotonic increasing in $t$. However, we do not know whether the inclusion just mentioned is true for $s>0$. In fact, one cannot rule out the following situation: for some Borel set $E\subset H$ and some $s>0$, there is a solution $\bu\in \Ucal_{[0,\infty)}$ with $\bu(s)\in \Sigma_s E$ and $\bu(t) \notin \Sigma_t E$, for any $t\neq s$.  Then, if $\rho$ is the Dirac delta measure carried by $\bu$, which is a Vishik-Fursikov measure, then $\rho_t(\Sigma_t E)=0$, for $t\neq s$, and $\rho_s(\Sigma_s E)=1$, so that $\rho_t(\Sigma_t E)$ is not monotone. Notice that, in this case, $\bu$ cannot be strongly continuous from the right at $t=s$, otherwise, using Lemma \ref{pastinglemma}, we could paste the restriction of $\bu$ over $[s,\infty)$ with the restriction, over $[0,s]$, of a  solution $\bv$ such that $\bv(0)\in E$ and $\bv(s)=\bu(s)$, which exists since $\bu(s)\in \Sigma_sE$. Then, we would find a solution which starts at $E$ and is equal to $\bu$ for all $t\geq s$, so that $\bu(t)$ would belong to $\Sigma_tE$, for $t\geq s$, which would be a contradiction.
\end{rmk}

\begin{rmk}
  \label{rmkbackwardaccretion}
  One can extend Corollary \ref{rho0accretive} to obtain the accretion property both forwards and
  backwards in time. First of all, one can define the multi-valued 
  evolution map backwards in time by introducing, for any given $t<0$ and any set $E$ in 
  $H$, the set $\Sigma_t E$ of all points $\bw\in H$ such that $\bw=\bu(t)$ and $\bu$ 
  is a weak solution in $\Ucal_{[t,\infty)}$ with the condition $\bu(0)\in E$.
  For $t<0$, we may also define the map
  $\Sigma_t^\sharp E$ given by all the points $\bw\in H$ such 
  that $\bw=\bu(t)$ and $\bu$ is a weak solution in $\Ucal_{[t,\infty)}^\sharp$ with the 
  condition that $\bu(0)\in E$.  As in the forward case, $\Sigma_t E$ and 
  $\Sigma_t^\sharp E$ are universally measurable sets in $H$ for any Borel subset $E$ 
  of $H$. Then, if $I$ is an interval of the form 
  $I=[t_0,\infty)$ or $I=(t_0,\infty)$, with $t_0<0$, or $I=\RR$, and $\{\rho_t\}_{t\in I}$ 
  is a Vishik-Fursikov statistical solution over $I$, it follows that $\{\rho_t\}_{t\in I}$ is 
  backwards accretive with respect to $\{\Sigma_t^\sharp\}_{t\in I, \;t\geq 0}$ in
  the sense that
  \[ \rho_t(\Sigma_t^\sharp E) \geq \rho_0(E),
  \]
  for all Borel sets $E$ in $H$ and all $t\in I$ with $t\leq 0$. 
  Moreover, at the particular time $t=t_0<0$ we have that
  \[ \rho_{t_0}(\Sigma_{t_0} E)\geq \rho_0(R),
  \]
  for all Borel sets $E$ in $H$.
  The proofs follow the same lines as that for the proof of Theorem \ref{accretionvfss}.
  One difference is that now we work with $\Pi_t$ defined on $\Ccal_\loc(I,H_\rw)$. 
  The most significant change is that the inclusions now become
  \[ (\Pi_t^{-1}\Sigma_t^\sharp E)\cap\Ucal_I 
         \supset (\Pi_0^{-1}E)\cap\Ucal_I,
  \] 
  in the result for $\Sigma_t^\sharp$, and
  \[ (\Pi_{t_0}^{-1}\Sigma_{t_0} E)\cap\Ucal_{[t_0,\infty)} 
         \supset (\Pi_0^{-1}E)\cap\Ucal_{[t_0,\infty)},
  \] 
  in the result for $\Sigma_{t_0}$.
  In the particular case that $\rho_0$ is a 
  Vishik-Fursikov stationary statistical solution on $H$, then the second result above 
  together with Corollary \ref{regrho0} yields that
  $\rho_0$ is backward and forward accretive for $\{\Sigma_t\}_{t\in \RR}$ in the sense 
  that 
  \[ \rho_0(\Sigma_t E)\geq \rho_0(E), \quad \forall t\in \RR,
  \]
  for all Borel sets $E$ in $H$.
\end{rmk}

\section{Recurrence results}
\label{sectionrecurrence}

In this section, we address the property of recurrence for accretive measures and Vishik-Fursikov stationary statistical solutions. In the classical theory, given an invariant measure for a dynamical system and a measurable subset $E$ of the phase space, almost all initial conditions in $E$ are recurrent in the sense that their trajectories return to $E$ infinitely often. In Section \ref{secpoincareforaccretive}, we prove a version of this result valid for arbitrary accretive measures (accretive in the sense of Definition \ref{defaccretive}). In Section \ref{secrecurrencevishikfursikovmeasures} we apply the classical Poincar\'e Recurrence Theorem (Theorem \ref{poincarerecurrence}) to invariant Vishik-Fursikov measures in the trajectory space and investigate its consequence to the Vishik-Fursikov stationary statistical solutions, which are the projections of invariant Vishik-Fursikov measures.

As mentioned in the Introduction, other ergodic-type results can be obtained. See, for instance, Remarks \ref{rmkergodicityonwcal} and \ref{wcalreginftyergodic}, and our recent work \cite{frttimeaverage}.

\subsection{Poincar\'e recurrence for accretive measures}
\label{secpoincareforaccretive}

We first give the following result, which is used in the proof of the main result in this section and which is also interesting on its own.

\begin{lem}
  \label{invariantmeasureoforbits}
  Let $E$ be a Borel set in $H$ and $t\geq 0$.  
  Then $\Sigma_t\gamma(E)\subset \gamma(E)$.
  If, moreover, $\mu$ is an accretive Borel measure in $H$ for $\{\Sigma_t\}_{t\geq 0}$,
  then $\mu(\Sigma_t\gamma(E))=\mu(\gamma(E))$.
\end{lem}

\begin{proof}
From the definition of the orbit as 
$\gamma(E)=\bigcup_{s\geq 0} \Sigma_s E$ and
using the fact that $\Sigma_t\Sigma_s E \subset \Sigma_{t+s} E$,
it follows immediately that $\Sigma_t \gamma(E)\subset \gamma(E)$. 

Consider now an accretive measure $\mu$ on $H$ for $\{\Sigma_t\}_{t\geq 0}$.
From Lemma \ref{sigmagammameas}, $\Sigma_t\gamma(E)$ is measurable. 
Then, from the accretion property, $\mu(\Sigma_t \gamma(E)) \geq \mu(\gamma(E))$,
while from the inclusion $\Sigma_t \gamma(E)\subset \gamma(E)$, we have
that $\mu(\Sigma_t \gamma(E)) \leq \mu(\gamma(E))$.
Therefore, $\mu(\Sigma_t \gamma(E)) = \mu(\gamma(E))$.
\end{proof}
\medskip

We now present the main result of this section, concerning the recurrence for accretive measures.

\begin{thm}
  \label{recurrenceforaccretivemeasures}
  Let $\mu$ be an accretive Borel probability measure on $H$ for 
  $\{\Sigma_t\}_{t\geq 0}$
  and let $E$ be a $\mu$-measurable set.
  Then, for $\mu$-almost every $\bu_0\in E$, there exists
  a sequence of positive times $t_n\rightarrow \infty$ 
  such that $(\Sigma_{t_n} \bu_0)\cap E \ne \emptyset$ for all $n\in \NN$.
\end{thm}

\begin{proof} 
Since $\mu$ is a Borel measure, any $\mu$-measurable set 
is of the form $E=E_0\cup E_N$, where $E_0$ is Borel and $E_N$ is of
null measure. Hence, $E_0\subset E$ in $H$ and $\mu(E)=\mu(E_0)$,
so it suffices to show the result for the Borel set $E_0$.

Consider the ``good'' set of recurrent points of $E_0$:
\[ E_0^\rg = \left\{ \bu_0\in E_0; \;\exists (t_n)_{n\in \NN}, \text{ such that } t_n\geq 0, \;t_n\rightarrow \infty, 
    \text{ and } \Sigma_{t_n}\bu_0 \cap E_0 \ne \emptyset, \forall n\in \NN \right\}.
\] 
We need to show that $E_0^\rg$ has full measure in $E_0$, i.e. 
$\mu(E_0\setminus E_0^\rg)=0$. Note, using in particular the
a~priori estimate \eqref{energyestimate}, that the ``bad'' set $E_0\setminus E_0^\rg$ can be decomposed into
\begin{equation}
  \label{badsetdecomposition}
  E_0\setminus E_0^\rg = \bigcup_{n,k\in \NN} E_{n,k}^\rb,
\end{equation}
where the smaller ``bad'' sets $E_{n,k}^\rb$ are given by
\[ E_{n,k}^\rb = \left\{ \bu_0\in E_0\cap B_H(kR_0); 
    \;\Sigma_t\bu_0 \cap E_0 \cap B_H(kR_0)= \emptyset, 
    \;\forall \;t\geq \frac{n}{\nu\lambda_1}\right\}.
\]
Hence, it suffices to show that $\mu(E_{n,k}^\rb)=0$ for each $n$.

Let us first check that $E_{n,k}^\rb$ is a co-analytic subset of the Polish space $B_H(kR_0)_\rw$, i.e. its complement in $B_H(kR_0)_\rw$ is analytic (see Section \ref{measuretheory}). For that purpose, note that
\begin{equation}
  \label{expressionforcomplementEbnk}
  B_H(kR_0)\setminus E_{n,k}^\rb = (B_H(kR_0)\setminus E_0) \bigcup 
     \left(E_0\cap B_H(kR_0)\right) \setminus E_{n,k}^\rb.
\end{equation}
Since $E_0$ is Borel, we just need to check that $E_{n,k}^\rb$ is co-analytic in $E_0\cap B_H(kR_0)$. Thus, we need to show that $\left(E_0\cap B_H(kR_0)\right) \setminus E_{n,k}^\rb$ is analytic. With that in mind, we use that this set can be written as
\begin{multline}
  \label{expressionforcomplementEbnkpartial}
  \left(E_0\cap B_H(kR_0)\right) \setminus E_{n,k}^\rb \\
    = E_0\cap\left(
      \Pi_0 \left(Q \left(\left([n/\nu\lambda_1,\infty)\times \Ucal_{[0,\infty)}(kR_0)\right) 
                  \cap \sigma^{-1}\left(\Pi_0^{-1} E_0\right)\right)\right)\right),
\end{multline}
where $Q$ is the projection that takes an element $(t,\bu)\in [0,\infty)\times \Ccal_\loc([0,\infty),H_\rw)$
to the element $\bu$, in $\Ccal_\loc([0,\infty,H_\rw)$. Indeed, this is a consequence of the following chain of equivalences
\begin{align*}
  \bu_0\in & (E_0\cap B_H(kR_0)) \setminus E_{n,k}^\rb \\   
    & \Leftrightarrow \bu_0\in E_0\cap B_H(kR_0), \;\exists t\geq n/\nu\lambda_1, 
           \;\exists \bv_0\in \Sigma_t\bu_0 \cap E_0\cap B_H(kR_0) \\
    &  \Leftrightarrow \bu_0\in E_0\cap B_H(kR_0), 
           \;\exists t\geq n/\nu\lambda_1, \;\exists \bv\in \Ucal_{[0,\infty)}(kR_0), 
           \;\bv(0)=\bu_0, \;\bv(t)\in E_0 \\
    & \Leftrightarrow \bu_0\in E_0\cap B_H(kR_0), \;\exists t\geq n/\nu\lambda_1, \;\exists 
           \bv\in \Ucal_{[0,\infty)}(kR_0)\cap \Pi_t^{-1}E_0, \;\bv(0)=\bu_0 \\
    & \Leftrightarrow \bu_0\in E_0, \;\exists t\geq n/\nu\lambda_1, 
           \;\bu_0\in \Pi_0 \left(\Ucal_{[0,\infty)}(kR_0)\cap \Pi_t^{-1}E_0\right) \\
    & \Leftrightarrow \bu_0\in E_0, \;\bu_0\in \Pi_0 
           \left(\Ucal_{[0,\infty)}(kR_0)\cap \left(\cup_{t\geq n/\nu\lambda_1}\Pi_t^{-1}E_0\right)\right)   \\
    &  \Leftrightarrow \bu_0\in E_0, \;\bu_0\in \Pi_0 
           \left(Q \left(\left([n/\nu\lambda_1,\infty)\times \Ucal_{[0,\infty)}(kR_0)\right) 
                  \cap \sigma^{-1}\left(\Pi_0^{-1} E_0\right)\right)\right) \\
    & \Leftrightarrow \bu_0\in E_0\cap \left(\Pi_0 \left(Q \left(\left([n/\nu\lambda_1,\infty)\times 
         \Ucal_{[0,\infty)}(kR_0)\right) \cap \sigma^{-1}\left(\Pi_0^{-1} E_0\right)\right)\right)\right).
\end{align*}
This proves the representation \eqref{expressionforcomplementEbnkpartial}. Now we need to show that the right hand side of \eqref{expressionforcomplementEbnkpartial} is analytic. Since $E_0$ is Borel, this amounts to showing that the projection 
\begin{equation}
  \label{futurensetprojected}
  (\Pi_0 \circ Q)\left(\left([n/\nu\lambda_1,\infty)\times \Ucal_{[0,\infty)}(kR_0)\right)\cap \sigma^{-1}\left(\Pi_0^{-1} E_0\right)\right)
\end{equation}
is analytic. Since $E_0$ is Borel in $H$ and
$\sigma$ and $\Pi_0$ are continuous, the set $\sigma^{-1}\left(\Pi_0^{-1} E_0\right)$
is a Borel set in $[0,\infty)\times \Ccal_\loc([0,\infty),H_\rw)$. On the other hand, since $\Ucal_{[0,\infty)}(kR_0)$
is Borel, the set 
\begin{equation}
  \label{futurenset}
  \left([n/\nu\lambda_1,\infty)\times \Ucal_{[0,\infty)}(kR_0)\right) 
     \cap \sigma^{-1}\left(\Pi_0^{-1} E_0\right) 
\end{equation}
is a Borel subset of the space $[0,\infty)\times \Ucal^\sharp_{[0,\infty)}(kR_0)$. According to Lemma \ref{Ucalspaceslem}, the space $\Ucal^\sharp_{[0,\infty)}(kR_0)$ is Polish, and hence so is $[0,\infty)\times \Ucal^\sharp_{[0,\infty)}(kR_0)$. Thus, the set in \eqref{futurenset} is a Borel subset of the Polish space $[0,\infty)\times \Ucal^\sharp_{[0,\infty)}(kR_0)$. Hence, the set in \eqref{futurensetprojected} is the continuous projection, through the projector $\Pi_0\circ Q$ and into the Polish space $B_H(kR_0)_\rw$, of a Borel subset of a Polish space. Since continuous functions between Polish spaces take Borel sets into analytic sets, (see Fact \eqref{continuousimageofanalyticsets} at the end of Section \ref{measuretheory}), it follows that the set in \eqref{futurensetprojected} is analytic in $B_H(kR_0)_\rw$. Hence, we conclude from \eqref{expressionforcomplementEbnkpartial} and \eqref{expressionforcomplementEbnk} that $E_{n,k}^\rb$ is co-analytic.

Since any co-analytic set is universally measurable (see Fact \eqref{analyticcouniversallymeasurable} at the end of Section \ref{measuretheory}), it follows, in particular, that $E_{n,k}^\rb$ is $\mu$-measurable. We now want to show that $E_{n,k}^\rb$ is a null set with respect to the measure $\mu$.

Notice that $\Sigma_t E_{n,k}^\rb \cap E_{n,k}^\rb=\emptyset$, 
for $t\geq n/\nu\lambda_1$, otherwise we would find $\bu_0\in E_{n,k}^\rb\subset E_0$ such that 
$\Sigma_t\bu_0\in E_{n,k}^\rb \subset E_0$, which contradicts the fact that 
$\bu_0\in  E_{n,k}^\rb$. Then, we have
\[ E_{n,k}^\rb \bigcap \left(\bigcup_{t\geq n/\nu\lambda_1} \Sigma_t E_{n,k}^\rb\right) 
     = \emptyset.
\]

Having in mind the lower regularity property of $\mu$ in $H$, we consider an arbitrary strongly compact set $K$ included in $E_{n,k}^\rb$. From this inclusion, it is clear that we also have
\[ K \bigcap \left(\bigcup_{t\geq n/\nu\lambda_1} \Sigma_t K\right) = \emptyset.
\]

Using property \eqref{Sigmaproperty} of $\{\Sigma_t\}_{t\geq 0}$, we find
$\Sigma_{n/\nu\lambda_1}\Sigma_t K \subset \Sigma_{n/\nu\lambda_1+t}K$, so that
\[ \bigcup_{t\geq n/\nu\lambda_1} \Sigma_t K
     = \bigcup_{t\geq 0} \Sigma_{n/\nu\lambda_1+t} K
     \supset \bigcup_{t\geq 0} \Sigma_{n/\nu\lambda_1} \Sigma_t K
     = \Sigma_{n/\nu\lambda_1} \left(\bigcup_{t\geq 0} \Sigma_t K\right)
     =  \Sigma_{n/\nu\lambda_1} \gamma(K).
\]
Hence,
\[ K \subset \left(\bigcup_{t\geq 0} \Sigma_t K\right)
     \setminus \left(\bigcup_{t\geq {n/\nu\lambda_1}} \Sigma_t K\right)
     \subset \gamma(K) \setminus \Sigma_{n/\nu\lambda_1} \gamma(K).
\]
Using Lemma \ref{invariantmeasureoforbits}, we deduce that
\[ \mu(K) 
      \leq \mu(\gamma(K) 
          \setminus \Sigma_{n/\nu\lambda_1} \gamma(K))
      \leq \mu(\gamma(K)) - \mu(\Sigma_{n/\nu\lambda_1} \gamma(K))
      = 0.
\]
Since $\mu$ is a regular measure and $\mu(K)=0$ holds for any strongly compact set $K$ included in $E_{n,k}^\rb$, we obtain that
\[ \mu(E_{n,k}^\rb) = \sup \{\mu(K); \;K\subset E_{n,k}^\rb, \; K \text{ compact in } H \} = 0.
\]
Then, as mentioned in the beginning of the proof, we conclude, using the decomposition \eqref{badsetdecomposition} of the bad set, that $\mu(E_0\setminus E_0^\rg)=0$, so that the ``good'' recurrent set $E_0^\rg$ is of full measure in $E_0$.
\end{proof}
\medskip

\begin{rmk}
  \label{rmknonwandering}
  Theorem \ref{recurrenceforaccretivemeasures} can be interpreted in the following way: 
  Given a $\mu$-measurable set $E$ with $\mu(E)>0$, we have that for $\mu$-almost every
  $\bu_0$ in $E$, there exists a sequence of positive real numbers $t_n\rightarrow \infty$
  and a sequence of weak solutions $\bu_n$ on $[0,\infty)$ with $\bu_n(0)=\bu_0$
  such that $\bu_n(t_n)\in E$.  In particular, we find that the support in $H$ of an arbitrary accretive measure $\mu$ for the 3D Navier-Stokes equations is made only of points which are nonwandering in $H$ with respect to the Leray-Hopf weak solutions. By that we mean that for any $\bu_0$ in $\supp\mu$, for any $\varepsilon>0$, and for an arbitrarily large time $T$, there exists an initial condition $\bv_0$ within a distance $\varepsilon$ of $\bu_0$ in $H$ and a weak solution $\bv=\bv(t)$ defined for all $t\geq 0$ with $\bv(0)=\bv_0$ such that for some $t\geq T$, the point $\bv(t)$ lies within $\varepsilon$ of $\bu_0$ in $H$.
\end{rmk}
\medskip

\begin{rmk}
  We have seen that a stationary statistical solution $\mu$ is carried by $B_H(R_0)$.
  We have also seen that time-average stationary statistical solutions are accretive
  for $\{\Sigma_t\}_{t\geq 0}$. But not necessarily every accretive measure is a
  stationary statistical solution. Theorem \ref{recurrenceforaccretivemeasures} implies 
  that any accretive measure is carried by $B_H(R_0)$. This follows from the
  energy estimate \eqref{energyestimate}, which implies that all the recurrent points must 
  be included in $B_H(R_0)$.
\end{rmk}

\begin{rmk}
  The set of recurrent points defined in the proof of Theorem \ref{recurrenceforaccretivemeasures} 
  can be written as 
  \[ E_0^g = \bigcap_n \bigcup_k \left( E_0 \bigcap \left(\Pi_0 Q 
    \left(\left([n/\nu\lambda_1,\infty)\times \Ucal_{[0,\infty)}(kR_0)\right)
       \bigcap \sigma^{-1}(\Pi_0^{-1}E_0)\right)\right)\right)
  \]
  This is another way of seeing that $E_0^g$ is universally measurable.
\end{rmk}

\subsection{Poincar\'e recurrence for Vishik-Fursikov stationary statistical solutions}
\label{secrecurrencevishikfursikovmeasures}

The previous Theorem \ref{recurrenceforaccretivemeasures} is an adaptation of the Poincar\'e recurrence theorem to a measure which is not necessarily invariant. It applies to any accretive measure. In particular, it applies to Vishik-Fursikov stationary statistical solutions since they are accretive according to Corollary \ref{rho0accretive}. However, since the Vishik-Fursikov stationary statistical solutions are projections of invariant Vishik-Fursikov measures, we can use the dynamic structure in the trajectory space
and obtain a stronger result from a direct application of
the classical Poincar\'e Recurrence Theorem (Theorem \ref{poincarerecurrence}).

We start by writing the recurrence property for invariant Vishik-Fursikov measures, which is an immediate consequence of Theorem \ref{poincarerecurrence} and the fact that $\rho$ is an invariant measure for the translation semigroup $\{\sigma_\tau\}_{\tau \geq 0}$.
\begin{cor}
  \label{recurrenceforvfmeasure}
  Consider an interval $I\subset \RR$ unbounded on the right. Let $\rho$ be an invariant Vishik-Fursikov measure over $I$, and let $\Ecal\subset \Ccal_\loc(I,H_\rw)$ be a Borel set. Then, for $\rho$-almost all trajectories $\bu$ in $\Ecal$, there exists a sequence of positive times $t_n\rightarrow \infty$ such that $\sigma_{t_n}\bu\in \Ecal$.
\end{cor}

When interpreting results such as this in phase space, the following result is useful.
\begin{lem}
  \label{fromtrajectorytophasespacealmosteverywhere}
  Let $\rho$ be a Borel probability measure on $\Ccal_\loc([0,\infty),H_\rw)$ which is carried by $\Wcal_{[0,\infty)}^\sharp$ and let $\rho_0=\Pi_0\rho$. Let $E$ be a Borel set in $H$ and set $\Ecal = \Pi_0^{-1}E_0$. Suppose that $\Ncal$ is a $\rho$-measurable subset of $\Ecal$ with $\rho(\Ncal)=0$. Then, $N=E \setminus (\Pi_0(\Ecal\setminus \Ncal))$ is a $\rho_0$-measurable subset of $E$ with the properties that
  \[ E \setminus N = \Pi_0 (\Ecal \setminus \Ncal), \qquad \Pi_0^{-1}N \subset \Ncal,
  \]
  and
  \[ \rho_0(N) = 0.
  \]
  In particular, for any $\bu_0\in E\setminus N$, the set $\Pi_0^{-1}\{\bu_0\}\cap (\Ecal\setminus \Ncal)$ is nonempty and of full measure in $\Pi_0^{-1}\{\bu_0\} \cap \Ecal$.
\end{lem}

\begin{proof}  
  Since $\Pi_0\Ecal = \Pi_0\Pi_0^{-1} E = E$, we find that
  \[ E\setminus N = E \setminus (E \setminus (\Pi_0(\Ecal\setminus \Ncal))) = E \cap \Pi_0(\Ecal\setminus \Ncal) = \Pi_0\Ecal \cap \Pi_0 (\Ecal\setminus \Ncal) = \Pi_0(\Ecal\setminus \Ncal), 
  \]
  which proves the first claim. We also have that
  \[  \Pi_0^{-1} N = \Pi_0^{-1} ( E \setminus (\Pi_0(\Ecal\setminus \Ncal))) = (\Pi_0^{-1}E) \setminus (\Pi_0^{-1}\Pi_0(\Ecal\setminus \Ncal)) \subset \Ecal \setminus (\Ecal\setminus \Ncal).
  \]
  Then, taking into consideration that $\Ncal\subset \Ecal$, we see that $\Ecal \setminus (\Ecal\setminus \Ncal) = \Ncal$, so that
  \[ \Pi_0^{-1} N \subset \Ncal,
  \]
  which proves the second claim. 
  
  Now, using that $\Pi_0^{-1} N \subset \Ncal$, we see that $\Pi_0^{-1}N$ is a subset of a $\rho$-measurable set of null $\rho$-measure, so that $\Pi_0^{-1}N$ is also $\rho$-measurable with null $\rho$-measure. Hence, by definition, $N$ is a $\rho_0$-measurable set with
  \[ \rho_0(N) = \rho(\Pi_0^{-1} N) \leq \rho(\Ncal) = 0.
  \]

  The claims concerning $\bu_0\in E\setminus N$ follow directly from the construction of $E\setminus N$.
\end{proof}

Now, using Corollary \ref{recurrenceforvfmeasure} and Lemma \ref{fromtrajectorytophasespacealmosteverywhere} we prove the following recurrence result in the phase space $H$.

\begin{thm}
  \label{recurrenceforvfmeasurephasespace}
  Let $\rho_0$ be a Vishik-Fursikov stationary statistical solution and let $E\subset H$ be a Borel set. Suppose that $\rho_0(E)>0$, for the sake of interest. Then, for $\rho$-almost every solution $\bu$ in $\Ecal =\Pi_0^{-1} E$, there exists a sequence of positive times $t_n\rightarrow \infty$ such that $\bu(t_n)\in E$. In particular, for $\rho_0$-almost all $\bu_0$ in $E$, there exists at least one global weak solution $\bu\in \Wcal$ with $\bu(0)=\bu_0$ and a sequence of positive times $t_n\rightarrow \infty$ such that $\bu(t_n)\in E$, for all $n\in \NN$. In other words, for any $\bu_0$ in $E$, the set 
  \[ \Rcal_{E}(\bu_0) = \{ \bu \in \Pi_0^{-1}\{\bu_0\}; \; \exists (t_n)_{n\in \NN}, \;t_n\geq 0, \;t_n\rightarrow \infty, \; \bu(t_n)\in E\}
  \]
  is nonempty for $\rho_0$-almost all $\bu_0$ in $E$, and, moreover, $\Rcal_{E}(\bu_0)$ is of full measure in $\Pi_0^{-1}\{\bu_0\}$.
\end{thm}

\begin{proof}
  Using Corollary \ref{regrho0}, we assume that there exists an invariant Vishik-Fursikov measure $\rho$ on $\Wcal$ such that $\Pi_0\rho = \rho_0$. We apply Corollary \ref{recurrenceforvfmeasure} to this measure $\rho$ and with $\Ecal=\Pi_0^{-1}E$, whose measure by definition is $\rho(\Ecal) = \rho(\Pi_0^{-1}E) = \rho_0(E)>0$. In doing so, we find that there exists a set $\Ncal\subset \Ecal$ such that $\rho(\Ncal) = 0$ and such that for any $\bu\in \Ecal\setminus \Ncal$, there exists a time sequence $t_n\rightarrow \infty$ such that $\sigma_{t_n}\bu \in \Ecal$, for all $n\in \NN$. In this case, $\bu(t_n) = \Pi_0\sigma_{t_n}\bu \in \Pi_0 \Ecal = \Pi_0\Pi_0^{-1} E = E$. In other words, for any $\bu\in \Ecal \setminus \Ncal$, there exists such a sequence $t_n$ such that $\bu(t_n)\in E$, for all $n\in \NN$. 
  
  Set now $N = E \setminus (\Pi_0(\Ecal\setminus \Ncal))$ as in Lemma \ref{fromtrajectorytophasespacealmosteverywhere}. Then, $N$ is of $\rho_0$-null measure and for every $\bu_0\in E\setminus N$, there exists $\bu\in \Ecal\setminus \Ncal$ such that $\bu_0=\Pi_0\bu$, which means that for $\rho_0$-almost every $\bu_0$ in $E$, there exists a trajectory $\bu$ with $\bu(0)=\bu_0$ and a sequence of positive times $t_n\rightarrow \infty$ such that $\bu(t_n)\in E$, for all $n\in \NN$. This also means that, for any $\bu_0\in E\setminus N$, the set $\Rcal_E(\bu_0)$ is nonempty, i.e. it is nonempty $\rho_0$-almost everywhere. 
  
  Finally, since $\Rcal_E(\bu_0)$ clearly contains $\Pi_0^{-1}\{\bu_0\} \setminus \Ncal$, then 
  \begin{multline*} 
    \rho(\Pi_0^{-1}\{\bu_0\}\setminus \Rcal_{E}(\bu_0)) \leq \rho(\Pi_0^{-1}\{\bu_0\} \setminus (\Pi_0^{-1}\{\bu_0\} \setminus \Ncal)) \\ 
    = \rho(\Pi_0^{-1}\{\bu_0\} \cap \Ncal) \leq \rho(\Ncal) = 0,
  \end{multline*}
which shows that $\Rcal_{E}(\bu_0)$ is of full measure in $\Pi_0^{-1}\{\bu_0\}$.
\end{proof}

Notice that this last result is stronger than the one in Theorem \ref{recurrenceforaccretivemeasures} since the solution which returns to $E$ infinitely often is the same, but this result applies only to Vishik-Fursikov stationary statistical solutions. 

Notice also that Theorem \ref{recurrenceforvfmeasurephasespace} holds trivially in the case that $\rho_0(E)=0$, but the result is vacuous.


\begin{thebibliography}{25}

\bibitem{aliprantisborder} C. D. Aliprantis and K. C. Border,
  \emph{Infinite Dimensional Analysis: A Hitchhiker's Guide,}
  3rd. Edition, Springer-Verlag, Berlin-Heidelberg, 2006.

\bibitem{bcfm1995} H. Bercovici, P. Constantin, C. Foias, O. P. Manley,
  Exponential decay of the power spectrum of turbulence,
  \emph{J. Stat. Phys.} 80 (1995), 579--602.

\bibitem{bourbaki69} N. Bourbaki, \emph{\'Elements de math\'ematique.
  Fasc. XXXV. Livre VI: Int\'egration. Chapitre IX: Int\'egration sur
  les espaces topologiques s\'epar\'es,} Actualit\'es Scientifiques
  et Industrielles, no. 1343, Hermann, Paris, 1969.

\bibitem{brownpearcy} A. Brown and C. Pearcy, \emph{Introduction to
  Operator Theory. I. Elements of Functional Analysis,}
  Graduate Texts in Mathematics, no. 55, Springer-Verlag,
  New York-Heidelberg, 1977.

\bibitem{chekrounglattholtz2012} M. D. Chekroun, N. E. Glatt-Holtz. Invariant measures for dissipative dynamical systems: abstract results and applications, \emph{Comm. Math. Phys. 316} (2012) no. 3, 723--761.

\bibitem{constantinfoias} P. Constantin and C. Foias, \emph{Navier-Stokes Equation,} University of Chicago Press, Chicago, 1989.

\bibitem{dapratozabczyk1996} G. Da Prato, J. Zabczyk, \emph{Ergodicity for Infinite Dimensional Systems.} Cambridge University Press, Cambridge, 1996.

\bibitem{dunfordschwartz} N. Dunford and J. T. Schwartz,
  \emph{Linear Operators, I. General Theory}, Pure and Appl. Math. 7,
  Interscience, New York, 1958.

\bibitem{foias72}
  C. Foias, Statistical study of Navier-Stokes equations I,
  \emph{Rend. Sem. Mat. Univ. Padova} 48 (1972), 219--348.

\bibitem{foias73}
  C. Foias, Statistical study of Navier-Stokes equations II,
  \emph{Rend. Sem. Mat. Univ. Padova} 49 (1973), 9--123.

\bibitem{foias74b}
  C. Foias, \emph{Solutions Statistiques des Equations de Navier-Stokes},
  Cours au Coll\`ege de France, 1974, unpublished.

\bibitem{foiasguillopetemam1981} C. Foias, C. Guillop\'e, R. Temam, New a priori estimates for Navier-Stokes equations in dimension 3, \emph{Comm. Partial Differential Equations} 6 (1981), no. 3, 329--359.

\bibitem{fmrt2001a}
  C. Foias, O. P. Manley, R. Rosa, and R. Temam,
  \emph{Navier-Stokes Equations and Turbulence,} Encyclopedia of Mathematics
  and its Applications, vol. 83, Cambridge University Press, 2001.

\bibitem{foiasprodi76} C. Foias and G. Prodi, Sur les solutions statistiques des 
\'equations de Navier-Stokes,  \emph{Ann. Mat. Pura Appl.} 111 (1976), no. 4, 307--330.
  
\bibitem{frt2010c} C. Foias, R. Rosa, and R. Temam, Topological properties of 
  the weak global attractor of the three-dimensional Navier-Stokes equations,
  \emph{Discrete Contin. Dyn. Syst.} 27 (2010), no. 4, 1611-1631.
  
\bibitem{frtssp1} C. Foias, R. Rosa, and R. Temam, Properties of time-dependent statistical solutions 
  of the three-dimensional Navier-Stokes equations, \emph{Annales de l'Institut Fourier,} 63 (2013), no. 6, 2515--2573.
  
\bibitem{frttimeaverage} C. Foias, R. Rosa, and R. Temam, Convergence of time averages of weak solutions of the three-dimensional Navier-Stokes equations, \emph{J. Stat.  Phys.} 160 (2015), 519--531.
  
\bibitem{foiastemam75} C. Foias and R. Temam, On the stationary
  of the Navier-Stokes equations and turbulence,
  Publications Mathematiques d'Orsay, no. 120-75-28 (1975), 38--77.

\bibitem{foiastemam79} C. Foias and R. Temam, Some analytic and
  geometric properties of the solutions of the Navier-Stokes
  equations, \emph{J. Math. Pures Appl.} 58 (1979), 339--368.

\bibitem{foiastemam85}
  C. Foias and R. Temam, The connection between the Navier-Stokes
  equations, dynamical systems, and turbulence theory, \emph{Directions in
  Partial Differential Equations} (Madison, WI, 1985), Publ. Math. Res.
  Center Univ. Wisconsin, 54, Academic Press, Boston, MA, 1987. pp. 55--73.

\bibitem{folland1999} G. B. Folland, \emph{Real Analysis: Modern Techniques and Their Applications,} Second Edition, John Wiley \& Sons, Inc., New York, 1999.

\bibitem{kb1937} N. Krylov, N. N. Bogoliubov, La th\'eorie g\'en\'erale de la mesure dans son application \`a l'\'etude des syst\`emes dynamiques de la m\'ecanique non lin\'eaire, \emph{Ann. Math.} 38 (1937), 65--113.

\bibitem{kuratowski} K. Kuratowski, \emph{Topology}, Vol. 1, 
  Academic Press, New York-London, 1966.

\bibitem{lady63} O. Ladyzhenskaya, \emph{The Mathematical Theory of Viscous
  Incompressible Flow,} Revised English edition, Translated
  from the Russian by Richard A. Silverman Gordon and Breach
  Science Publishers, New York-London, 1963.

\bibitem{leray} J. Leray, Essai sur le mouvement d'un liquide visqueux
  emplissant l'espace, \emph{Acta Math.} 63 (1934), 193--248.

\bibitem{lukaszewickrealrobinson2011} G. Lukaszewicz, J. Real, J., J. C. Robinson, Invariant measures for dissipative systems and generalised Banach limits, \emph{Journal of Dyn. Diff. Equat.} 23 (2011), issue 2, 225--250.

\bibitem{moschovakis} Y. N. Moschovakis, \emph{Descriptive Set Theory},
  North-Holland Publishing Co., Amsterdam-New York, 1980.
  
\bibitem{pollicottyuri} M. Pollicott and M. Yuri, \emph{Dynamical Systems and
  Ergodic Theory,} Cambridge University Press, Cambridge, 1998.  

\bibitem{rudin} W. Rudin, \emph{Real and Complex Analysis,} Third edition,
  McGraw-Hill Book Co., New York, 1987.

\bibitem{scheffer} V. Scheffer, Hausdorff measures and the Navier-Stokes
   equations, \emph{Comm. Math. Phys.} 55 (1977), 97--112.

\bibitem{temam} R. Temam \emph{Navier-Stokes Equations. Theory and numerical
  analysis,} Studies in Mathematics and its Applications, 2
  3rd edition, North-Holland Publishing Co., Amsterdam-New York,
  1984. Reedition in 2001 in the AMS Chelsea series, AMS, Providence.

\bibitem{temam2} R. Temam, \emph{Infinite Dimensional Dynamical Systems
  in Mechanics and Physics,} Applied Mathematical Sciences 68 (2nd
  Edition, 1997), Springer Verlag, New York, 1988.

\bibitem{temam3} R. Temam, \emph{Navier-Stokes Equations and
  Nonlinear Functional Analysis,} 2nd Edition, SIAM, Philadelphia, 1995.

\bibitem{vishikfursikov78} M. I. Vishik and A. V. Fursikov, Translationally homogeneous statistical solutions and individual solutions with infinite energy of a system of Navier-Stokes equations, \emph{Siberian Mathematical Journal}, Vol. 19 (1978), no. 5, 710--729 (Translated from Sibirskii Matematicheskii Sbornik, Vol. 19, no. 5, 1005--1031, September-October 1978.)

\bibitem{vishikfursikov88} M. I. Vishik and A. V. Fursikov, \emph{Mathematical Problems of Statistical Hydrodynamics,} Kluwer, Dordrecht, 1988.
  
\bibitem{walters1982} P. Walters, \emph{An Introduction to Ergodic Theory,} Graduate Texts in Mathematics 79, Springer-Verlag, New York, 1982.

\bibitem{wang2009} X. Wang, Upper semi-continuity of stationary statistical properties of dissipative systems, \emph{Disc. Cont. Dyn. Sys.} 23 (2009), 521--540.

\end{thebibliography}
\end{document}